\documentclass[11pt,reqno]{amsart}

\usepackage{amsthm,amsmath,amssymb,booktabs,xcolor,graphicx}
\usepackage{bbm}
\usepackage[shortlabels]{enumitem}
\usepackage{listings}
\usepackage{xcolor}
\usepackage{mathtools}
\definecolor{codegreen}{rgb}{0,0.6,0}
\definecolor{codegray}{rgb}{0.5,0.5,0.5}
\definecolor{codepurple}{rgb}{0.58,0,0.82}
\definecolor{backcolour}{rgb}{0.95,0.95,0.92}

\lstdefinestyle{mystyle}{
  backgroundcolor=\color{backcolour},   commentstyle=\color{codegreen},
  keywordstyle=\color{magenta},
  numberstyle=\tiny\color{codegray},
  stringstyle=\color{codepurple},
  basicstyle=\ttfamily\footnotesize,
  breakatwhitespace=false,         
  breaklines=true,                 
  captionpos=b,                    
  keepspaces=true,                 
  numbers=left,                    
  numbersep=5pt,                  
  showspaces=false,                
  showstringspaces=false,
  showtabs=false,                  
  tabsize=2
}

\usepackage[margin=1in]{geometry}

\usepackage[hidelinks]{hyperref}

\usepackage[utf8]{inputenc}
\usepackage[T1]{fontenc}

\usepackage[textsize=small,backgroundcolor=orange!20]{todonotes}

\usepackage[hidelinks]{hyperref}
\usepackage{url}

\usepackage{appendix}
\usepackage{etoolbox}
\usepackage[noabbrev,capitalize]{cleveref}
\crefname{equation}{}{} 
\crefname{subsection}{Subsection}{Subsections} 
\AtBeginEnvironment{appendices}{\crefalias{section}{appendix}} 

\usepackage[color,final]{showkeys} 

\colorlet{refkey}{orange!20}
\colorlet{labelkey}{blue!30}

\newtheorem{theorem}{Theorem}[section]

\newtheorem{lemma}[theorem]{Lemma}

\newtheorem*{question*}{Question}

\theoremstyle{definition}
\newtheorem{definition}[theorem]{Definition}

\newtheorem*{definition*}{Definition}

\theoremstyle{remark}
\newtheorem*{remark}{Remark}

\numberwithin{equation}{section}

\newcommand{\mb}{\mathbb}
\newcommand{\mbf}{\mathbf}
\newcommand{\mbm}{\mathbbm}
\newcommand{\mc}{\mathcal}

\newcommand{\ol}{\overline}
\newcommand{\on}{\operatorname}

\newcommand{\wt}{\widetilde}

\newcommand{\e}{\epsilon}
\newcommand{\eps}{\varepsilon}
\newcommand{\kAP}{\on{kAP}}
\newcommand{\nkAP}{\on{\overline{kAP}}}

\allowdisplaybreaks

\title{Local limit theorems for subgraph counts}

\author[Sah]{Ashwin Sah}
\author[Sawhney]{Mehtaab Sawhney}
\address{Massachusetts Institute of Technology, Cambridge, MA 02139, USA}
\email{\{asah,msawhney\}@mit.edu}

\date{}

\begin{document}

\begin{abstract}
We introduce a general framework for studying anticoncentration and local limit theorems for random variables, including graph statistics. Our methods involve an interplay between Fourier analysis, decoupling, hypercontractivity of Boolean functions, and transference between ``fixed-size'' and ``independent'' models. We also adapt a notion of ``graph factors'' due to Janson.

As a consequence, we derive a local central limit theorem for connected subgraph counts in the Erd\H{o}s-Renyi random graph $G(n,p)$, building on work of Gilmer and Kopparty as well as Berkowitz. Among other things, this improves an anticoncentration result of Fox, Kwan, and Sauermann. We also derive a local limit central limit theorem for induced subgraph counts, as long as $p$ is bounded away from a set of ``problematic'' densities, partially answering a question of Fox, Kwan, and Sauermann. We are further able to prove similar results in the more delicate $G(n,m)$ model, which samples a uniformly random graph with $n$ vertices and $m$ edges.

We then demonstrate that the restrictions in our two main results are necessary by exhibiting a disconnected graph for which anticoncentration for subgraph counts at the optimal scale fails for all constant $p$, and finding a graph $H$ for which anticoncentration for induced subgraph counts fails in $G(n,1/2)$. These counterexamples resolve anticoncentration conjectures of Fox, Kwan, and Sauermann in the negative.

Finally, we also examine the behavior of counts of $k$-term arithmetic progressions in subsets of $\mb{Z}/n\mb{Z}$ chosen uniformly of size $m$ (with the restriction that $\min(m,n-m) = \Theta(n)$) and prove a local central limit theorem. Transferring this result to the independent model, we deduce a local limit theorem wherein the behavior is Gaussian at a global scale but has nontrivial local oscillations (according to a Ramanujan theta function). These results improve on results of and answer questions of the authors and Berkowitz, and answer a question of Fox, Kwan, and Sauermann.
\end{abstract}

\maketitle

\section{Introduction}\label{sec:introduction}
Random graph models have been studied in depth since their introduction by Erd\H{o}s and R\'enyi \cite{ER61}. Among the most important statistics of the random variables $G(n,p)$ and $G(n,m)$ are subgraph counts. Their study has been used, for instance, to construct large graphs avoiding certain subgraphs, among numerous other applications.

One natural question concerns the limiting distribution of the count $X_H$ of some fixed graph, say $H$, within such a graph model. In the 1980's several mathematicians studied this problem, for instance showing that for any subgraph $H$ we have convergence to a Gaussian in $G(n,p)$ for $p\in(0,1)$. In particular this proves that if $\mu_H, \sigma_H$ are the mean and standard deviation of $X_H$, the random variable counting the number of appearances of $H$, we have
\[\lim_{n\to\infty} \mb{P}\bigg[a\le \frac{X_H - \mu_{H}}{\sigma_{H}}\le b\bigg] = \frac{1}{\sqrt{2\pi}}\int_a^b e^{-\frac{x^2}{2}}dx.\] Note that these results imply asymptotically the probability that $X_{H}$ lies in a given interval of size $\approx \sigma_{H}$. Using various methods, more quantitative forms of the above have resulted, with error terms of quality $O(n^{-\eps})$. This allows one to control the probability that $X_{H}$ lives in intervals of size $\sigma_Hn^{-\eps}$ for some $\eps > 0$. See e.g. \cite{NW88,Ru88,BKR89}.

A natural question to ask is if one can push this distributional control to pointwise control over probabilities. In particular, is it true that
\[\mb{P}[X_H = x] = \frac{1}{\sqrt{2\pi}\sigma_{H}} \exp\bigg[\frac{-(x-\mu_H)^2}{2\sigma_H^2}\bigg] + o_{n\to\infty}\left(\frac{1}{\sigma_{H}}\right)?\] A result of this form is referred to as a \emph{local central limit theorem}. Local limit theorems have a long history with the first central limit theorem, the De Moivre-Laplace central limit theorem in fact being a local central limit theorem. For independent integer-valued random variables, the seminal result of Gnedenko \cite{G48}, specializes to that as long as there are no obvious modulus obstructions one in fact has a local central limit theorem. Local central limit theorems are now known in a large number of combinatorial situations including the size of the giant component of a random graph \cite{BCM92} (extended to hypergraphs in \cite{B10}) and the number of comparisons for merge-sort of a random permutation \cite{H96}.

In recent studies of subgraph counts, more emphasis has been placed on the idea of \emph{anticoncentration}. The first result in this direction was to Meka, Nyugen, and Vu \cite{MNV16} who proved, as a consequence of a more general anticoncentration result, that each point probability is $n^{-1+o(1)}$. This was vastly improved for connected graphs by the recent work of Fox, Kwan, and Sauermann \cite{FKS19} in which they prove that 
\[\sup_x \mb{P}[X_H = x]\le n^{o(1)}\sigma_H^{-1}\]
for any connected subgraph $H$, using combinatorial methods. However, by design, anticoncentration on its own does not point towards a derivation of a local limit theorem.

Using careful analysis of characteristic functions on different regimes, Gilmer and Kopparty \cite{GK16} showed that the triangle count in $G(n,p)$ indeed exhibits a local central limit theorem. This was improved by Berkowitz \cite{B1}, who additionally proved the result for $r$-cliques \cite{B2}.

We introduce a general framework towards proving such anticoncentration and local limit results, synthesizing many of the advances referenced above along with a notion of ``graph factors'' used in work of Janson \cite{J2}, and introduce a method of transferring results from $G(n,m)$ to $G(n,p)$. This transfer between ``fixed-size'' and ``independent'' models, as we see in the case of $k$-APs, will allow us to establish local limit theorems even when the pointwise behavior is not purely Gaussian. 

Using our framework, we demonstrate optimal anticoncentration for connected subgraphs $H$ by establishing a local central limit theorem, improving on the anticoncentration results of Fox, Kwan, and Sauermann \cite{FKS19}. In particular, we demonstrate an analogous result for the more delicate $G(n,m)$ model for $m/\binom{n}{2}\in (\lambda,1-\lambda)$, a generalization suggested in \cite{FKS19}. (This then transfers to the $G(n,p)$ model.) These results answer a question of Fox, Kwan, and Sauermann \cite{FKS19} in the connected case. As we will see later with $k$-APs, beyond simply being a natural generalization, analyzing the fixed-size model is key in establishing anticoncentration and local limit theorems in broader situations.

We then use similar techniques to study induced subgraph counts in random graphs of constant density. We demonstrate that as long as the density is sufficiently far away from a set of ``problematic'' densities $P_{\text{crit},H}$, the count of induced copies of $H$ exhibit a local central limit theorem as well. This applies to both the $G(n,m)$ model if $m$ is $\Theta(n^2)$ far from $p_{\text{crit}}\binom{n}{2}$ for $p_{\text{crit}}\in P_{\text{crit},H}$ as well as the $G(n,p)$ model for constant $p\notin P_{\text{crit},H}$. This again partially answers a conjecture of Fox, Kwan, and Sauermann \cite{FKS19} regarding the anticoncentration of induced subgraph counts.

We then demonstrate that in a certain sense the previous results are optimal by exhibiting a disconnected graph for which anticoncentration at the optimal scale fails (and, in fact, by a polynomial amount) for all constant $p$, as well as an induced graph $H$ for which anticoncentration at the optimal scale fails (by a polynomial amount) in $G(n,1/2)$. These counterexamples resolve conjectures Fox, Kwan, and Sauermann \cite{FKS19} in the negative.

We also use our techniques to attack a structurally similar problem, that of length $k$ arithmetic progressions (for fixed $k\ge 3$) in a random subset of $\mb{Z}/n\mb{Z}$. Significant attention has been given to understanding the large deviation behavior of $\kAP(S)$ for random subsets $S$, particularly in the regime where the probability $p$ that each element is chosen tends to $0$. Here $\kAP(S)$ denotes the number of $k$-term arithmetic progressions with all elements in $S$. Recently Harel, Mousset, and Samotij \cite{HMS19} (improving on earlier works of Warnke \cite{W17} and Bhattacharya, Ganguly, Shao, and Zhao \cite{BGSZ20}) found precise upper tail bounds for $\kAP(S)$ in the sparse regime, while Janson and Warnke \cite{JW16} proved lower tail bounds. 

We prove a local central limit theorem for $\kAP(S)$ where $S$ is chosen to be uniform over sets of a fixed size $m\in (\lambda n,(1-\lambda)n)$ and $\kAP(S)$ denotes the number of $k$-term arithmetic progressions fully contained in $S$. For the model in which each element is picked independently, however, a local central limit theorem does not hold. Recent work by the authors and Berkowitz \cite{BSS20} demonstrates in a quantitative sense that the distribution behaves similarly to a convolution of discrete Gaussians at two scales, but does not prove a local limit theorem. By transferring our result in the fixed-size model, we can achieve this control, hence answering a question raised in \cite{BSS20}. This also demonstrates optimal anticoncentration for $k$-APs, a problem suggested in \cite{FKS19}. In this case, the local limit proven has Gaussian ``large-scale'' structure, but exhibits nontrivial oscillations at a slightly smaller scale that are ultimately given by a theta function.

\subsection*{Notation}
Throughout we use $f\lesssim g$ to mean $|f|\le Cg$ for some constant $C$, and $o,O$ have their usual meanings. Subscripts denote dependence of the implicit constant.

\subsection{Main results}
We now state the main results of this work. The first main result is a local central limit theorem for connected subgraphs in $G(n,p)$.
\begin{theorem}\label{thm:Gnp-subgraph-local}
Fix a connected graph $H$. Choose $n\ge 1$ with $p\in(\lambda,1-\lambda)$ and sample a graph $G$ from $G(n,p)$. Let the number of times $H$ appears as a subgraph of $G$ be $X_H$. Let $\mu_{H}$ be the mean and $\sigma_{H}$ the standard deviation of $X_H$. Finally define $Z_H = (X_{H} - \mu_H)/\sigma_{H}$. Then we have for any $\eps>0$ that
\[\sup_{z\in(\mb{Z}-\mu_H)/\sigma_H}|\sigma_{H} \mb{P}[Z_H = z] - \mathcal{N}(z)|\lesssim_{\lambda,\eps} n^{\eps-1/2},\qquad\sum_{z\in (\mb{Z} - \mu_H)/\sigma_H} |\mb{P}[Z_H = z] - \mathcal{N}(z)/\sigma_{H} |\lesssim_{\lambda,\eps}n^{\eps-1/2}.\]
\end{theorem}
\begin{remark}
This phrasing with $\eps>0$ can be made more quantitative via adding a large number of logarithm terms. We avoid specifying this dependence for the sake of clarity.
\end{remark}
Our second main result is an anticoncentration result for induced subgraphs provided that they are sufficiently far away from a problematic set of densities.
\begin{theorem}\label{thm:Gnp-induced-local}
Fix a graph $H$. Choose $n\ge 1$ with $p\in(\lambda,1-\lambda)$ and $\lambda$-separated from a set of problematic densities $\mc{P}_{\text{crit}}$ and sample a graph $G$ from $G(n,p)$. (Note here that $\mc{P}_{\text{crit}}$ is explicit given $H$, with size at most $v(H)^2$.) Let the number of times $H$ appears as an induced subgraph of $G$ be $X_H$. Let $\mu_{H}$ be the mean and $\sigma_{H}$ the standard deviation of $X_H$. Finally define $Z_H = (X_{H} - \mu_H)/\sigma_{H}$. Then we have for any $\eps > 0$ that
\[\sup_{z\in(\mb{Z}-\mu_H)/\sigma_H}|\sigma_{H} \mb{P}[Z_H = z] - \mathcal{N}(z)|\lesssim_{\lambda,\eps} n^{\eps-1/2},\qquad\sum_{z\in (\mb{Z} - \mu_H)/\sigma_H} |\mb{P}[Z_H = z] - \mathcal{N}(z)/\sigma_{H} |\lesssim_{\lambda,\eps}n^{\eps-1/2}.\]
\end{theorem}
\begin{remark}
In fact we can show $|\mc{P}_{\text{crit}}|\lesssim v(H)$ and the truth is likely smaller, possibly even constant order. However, as we see next, the existence of this problematic set cannot be avoided.
\end{remark}
As mentioned, we can prove analogous results to \cref{thm:Gnp-subgraph-local,thm:Gnp-induced-local} for $G(n,m)$; we defer the statements to \cref{thm:Gnm-subgraph-local,thm:Gnm-induced-local}.

Given these two results there is a natural question: do we need to exclude disconnected subgraphs in \cref{thm:Gnp-subgraph-local} and the set of problematic densities in \cref{thm:Gnp-induced-local}? In both cases the answer is, surprisingly, yes. These examples resolve conjectures Fox, Kwan, and Sauermann \cite{FKS19} in the negative.
\begin{theorem}\label{thm:counterexample-main}
Let $H$ be the disjoint union of $2$ edges and fix $p\in (0, 1).$ Sample a graph $G$ from $G(n,p)$, and let $X_H$ count subgraphs of $G$ isomorphic to $H$, with $\mu_H,\sigma_H$ the mean and standard deviation. Then 
\[\sup_{x}\mb{P}[X_{H} = x]\gtrsim n^{1/2}\sigma_H^{-1}.\]
Furthermore there is a graph $H'$ on $64$ vertices such that if $G$ is sampled from $G(n,1/2)$, and if $X_{H'}$, $\mu_{H'}$, $\sigma_{H'}$ are defined analogously to before with respect to induced copies of $H$, then
\[\sup_{x}\mb{P}[X_{H'} = x]\gtrsim n^{1/2}\sigma_{H'}^{-1}.\]
\end{theorem}
Note that $X_H$ for $H = K_2 + K_2$ (the union of $2$ disjoint edges) is counting subgraphs and not homomorphisms. For homomorphism counts, it is easy to see that $H = K_2 + K_2$ forms a counterexample to anticoncentration, as every homomorphism count is a square number. The first part of \cref{thm:counterexample-main} is not so straightforward, although this observation that the number of edges ``mostly determines'' the number of counts offers a useful intuition for the failure of anticoncentration in this case.

Finally, we adapt our methods to handle the case of random $k$-term arithmetic progressions.  
\begin{theorem}\label{thm:kAP-local}
Fix $k\ge 3$ and choose $n\ge 1$ with $\gcd(n,(k-1)!)=1$. Let $m/n\in(\lambda,1-\lambda)$ and choose a uniformly random subset $\mb{Z}/n\mb{Z}$ of size $m$ among all sets of this size; let $\mbf{x}$ be the indicator vector. Furthermore let $\mu_{k} = \mb{E}[\kAP(\mbf{x})]$ and $\sigma_{k} = \on{Var}[\kAP(\mbf{x})]$. Finally define $Z_k = (\kAP(\mbf{x}) - \mu_k)/\sigma_{k}$. Then we have for any $\eps > 0$ that
\[\sup_{z\in(\mb{Z}-\mu_k)/\sigma_k}|\sigma_k \mb{P}[Z_k = z] - \mathcal{N}(z)|\lesssim_{\lambda,\eps}n^{\eps-1/4},\qquad\sum_{z\in (\mb{Z} - \mu_k)/\sigma_k}|\mb{P}[Z_k = z] - \mathcal{N}(z)/\sigma_k |\lesssim_{\lambda,\eps}n^{\eps-1/4}.\]
\end{theorem}
As mentioned earlier, by transferring our result in the fixed-size model, we can also achieve a local limit theorem in the model where each element is chosen with a probability $p\in (0,1)$, hence answering a question raised by the authors and Berkowitz \cite{BSS20} and providing an optimal anticoncentration result, a direction suggested by Fox, Kwan, and Sauermann \cite{FKS19}. We stress here that the distribution in the case where each element is chosen with probability $p$ is not a pointwise Gaussian as a local \emph{central} limit theorem in this model is false due to the results of Berkowitz and the authors \cite{BSS20}. Instead the distribution is a mixture of an infinite ensemble of Gaussians, reflected by a theta series, as hypothesized in \cite{BSS20} and discussed further in the final section of the paper. We defer the statement of this result to \cref{thm:kAP-local-ensemble}.

\subsection{Overview of methods}\label{sub:overview}
The methods of this work are Fourier analytic. For the sake of concreteness, we will restrict our attention to the case of connected subgraph counts although the proof for the remaining results are closely related. The main calculation is (essentially) demonstrating that \[\mb{E}[e^{itZ_{H}}]\approx \mb{E}[e^{itZ}]\text{ for } |t|\le \pi\cdot\sigma_H\] where $Z\sim \mathcal{N}(0,1)$ is a standard Gaussian. Then the Fourier inversion formula for lattices gives an expression for the desired pointwise probabilities. We have $3$ different ranges.
\begin{enumerate}[1.]
    \item For $|t|\le n^{\eps}$ we derive a sufficiently strong quantitative central limit theorem which can be used to provide bounds on the characteristic function of $Z_H$. In particular, Stein's method and the method of dependency graphs give a quantitative bound on the Wasserstein distance between $Z_H$ and $Z$ (in the independent model), which ultimately allows control of this range. This is closely related to the proof of a quantitative CLT given in \cite{BKR89}. A ``repair'' argument allows us to transfer this to the fixed-size model.
    \item For $n^\eps\le |t|\le\sigma_Hn^{-\eps}$ we use decoupling arguments generalizing proofs of Berkowitz \cite{B2}, relying on hypercontractive estimates to bound the typical sizes of coefficients of certain characteristic functions. However, since our random variables are constrained to live on a slice of the hypercube, various modifications are necessary. In particular, we prove a decoupling lemma suitable for this situation and prove cancellation for characteristic functions of linear combinations of such random variables.
    \item Finally, for $\sigma_Hn^{-\eps}\le |t|\le\pi\sigma_H$ we again use decoupling arguments related to those given in Berkowitz \cite{B2}. However in our case certain gymnastics are necessary in order to set up the decoupling method in order to hit the very top of the range, with complications arising which are not present in \cite{B2}.
\end{enumerate}
When performing such analysis in general, we see that a notion of ``graph factors'' stemming from work of Janson \cite{J2} is crucial, as it allows us to capture possible degeneracies as well as failures of local central limit theorems or even anticoncentration.

\subsection{Structure of the paper}
In \cref{sec:preliminaries} we introduce necessary preliminaries for our estimates, including hypercontractivity, bounds for characteristic functions, decoupling techniques, and a notion of ``graph factors'' that decompose graph statistics. In \cref{sec:general-graph-characteristic-functions} we prove bounds for characteristic functions of graph statistics in $G(n,m)$ in a high degree of generality; we also explain how that is already enough to deduce optimal anticoncentration for very general graph polynomials. In \cref{sec:subgraph-counts} we specialize to connected subgraph counts and induced subgraph counts, proving $G(n,m)$ versions of \cref{thm:Gnp-subgraph-local,thm:Gnp-induced-local}. In \cref{sec:independent-models} we transfer those versions to $G(n,p)$, establishing \cref{thm:Gnp-subgraph-local,thm:Gnp-induced-local}. In \cref{sec:counterexamples} we prove \cref{thm:counterexample-main}. Finally, in \cref{sec:kAP} we prove \cref{thm:kAP-local} as well as transfer to the corresponding result in the independent model.

\section{Preliminaries}\label{sec:preliminaries}
\subsection{Hypercontractivity}\label{sub:hypercontractivity}
We will repeatedly require hypercontractive estimates which for us serve as tail bounds in a number of applications. These follow directly from theorems stated in O'Donnell's book \cite{O14}, although the results are originally due to Bonami, Beckner, Borell, and others.
\begin{theorem}[{\cite[Theorem~10.24]{O14}}]\label{thm:concentration-hypercontractivity}
Let $f$ be a polynomial in $n$ variables of degree at most $d$, and let $X = (X_i)_{1\le i\le n}$ be a sequence mutually independent boolean random variables such that each value is taken with probability at least $\lambda$. Then for any $t\ge (2e/\lambda)^{d/2}$,
\[\mb P_X\left[|f(X)|\ge t\|f\|_2\right]\le \lambda^d\exp\left(-\frac{d}{2e}\lambda t^{2/d}\right).\]
Here $\|f\|_2^2 = \mb E_X f(X)^2$.
\end{theorem}
\begin{theorem}[{\cite[Theorem~10.21]{O14}}]\label{thm:moment-hypercontractivity}
With the same hypotheses as above, if $q\ge 1$,
\[\mb{E}_X[|f(X)|^{2q}]\le (2q-1)^{dq}\lambda^{d(1-q)}\|f\|_2^{2q}.\]
\end{theorem}
\begin{remark}
We note that this theorem as stated in \cite{B2} has an incorrect exponent on $\lambda$, but it does not affect the results or proofs in any nontrivial fashion. 
\end{remark}
We often deal with a model in which the variables to which we wish to apply hypercontractivity are not independent, but constrained to have a fixed sum. Rather than use hypercontractivity on the slice, we use a trick of Jain \cite{Jai19} which allows one to transfer bounds from the independent model to a fixed sum model via a simple conditioning argument. See the proof of \cite[Lemma~5.4]{Jai19} for an example of this trick.

\subsection{Converting from characteristic function to distributional control}\label{sub:characteristic-function}
\begin{definition}
Let $X$ be a random variable.  Then its \emph{characteristic function} $\varphi_X:\mb{R}\to\mb{C}$ is defined to be
$\varphi_X(t):=\mb{E}_X[e^{itX}].$
\end{definition}

Characteristic functions are very well studied objects, and for sufficiently nice random variables the associated characteristic functions completely determine the random variable (e.g. due to L\'evy's continuity theorem). In particular, we will use the following inversion formula which bounds the $L^\infty$ distance between the probability distribution of a lattice-valued random variable and the standard Gaussian in terms of characteristic functions. Let $\mc{N}(x)$ be the probability density function of the standard normal.
\begin{lemma}[{\cite[Lemma~2]{B2}}]\label{lem:Linfty-distance}
Let $X_n$ be a sequence of random variables
supported in the lattices $\mathcal{L}_n=b_n+h_n\mb{Z}$, then
\[\sup_{x\in\mc{L}_n}|h_n\mathcal{N}(x)-\mb{P}[X_n=x]|\le h_n\left(\int_{-\frac{\pi}{h_n}}^{\frac{\pi}{h_n}}\left|\varphi_{\mc{N}(0,1)}(t)-\varphi_{X_n}(t)\right|dt+e^{-\frac{\pi^2}{2h_n^2}}\right)\]
\end{lemma}
We also use the following conversion to an $L^1$ distance estimate from the standard normal.
\begin{lemma}[{\cite[Lemma~3]{B2}}]\label{lem:L1-distance}
Let $X_n$ be a sequence of random variables supported in the lattice 
$\mathcal{L}_n:=b_n+h_n\mb{Z}$,
and with characteristic functions $\varphi_n$.  Assume that there is $A > 0$ such that the following hold:
\begin{enumerate}[1.]
\item $\sup_{x\in \mathcal{L}_n} |\mb{P}[X_n=x]-h_n\mc{N}(x)|<\delta_n h_n$
\item $\mb{P}[|X_n|>A]\le\epsilon_n$
\end{enumerate}
Then $\sum_{x\in \mathcal{L}_n} |\mb{P}[X_n=x]-\mc{N}(x)|\le 2A \delta_n+\epsilon_n+\frac{h_n}{\sqrt{2\pi}A}e^\frac{-A^2}{2}$.
\end{lemma}

\subsection{Estimates for characteristic functions}\label{sub:estimates-for-characteristic-functions}
We will need a variety of estimates which will be used repeatedly in order to bound the characteristic functions of the random variables which we encounter. The first is essentially a well-known elementary estimate on the cosine function and although this precise result is not required for our setting it is present for comparison with the following estimate which obtains cancellation over a boolean slice $\{\sum_{j=1}^nx_j = s\}$.
\begin{lemma}\label{lem:bernoulli}
Let $Y\sim\on{Ber}(p)$. For any $|t|\le\pi$ we have  
\[|\mb{E}[e^{itY}]|\le 1-\frac{2p(1-p) t^2}{\pi^2}.\]
\end{lemma}
\begin{proof}
Note that 
\begin{align*}
|\mb{E}[e^{itY}]| &= (1 -2p(1-p) + 2p(1-p)\cos t)^{1/2}\\
&\le 1 - p(1-p)(1-\cos t)\\
&\le 1 - \frac{2p(1-p)t^2}{\pi^2}
\end{align*}
where we have used that $1-\cos(t)\ge 2t^2/\pi^2$ for $|t|\le \pi$.
\end{proof}
The more difficult bound we will need is on characteristic functions when restricted to a slice of the hypercube. This is the critical estimate as it allows us to control a number of characteristic functions which will come up.
\begin{lemma}\label{lem:bernoulli-variance}
Let $x_j$ be drawn with $(x_1,\ldots,x_n)$ uniform on $\{0,1\}^n$ subject to $\sum_{j=1}^{n} x_j = s$. Furthermore suppose that $p = s/n$ and $t\in\mb{R}$ is such that $|(a_j-a_k)t|\le\pi$ for all $1\le j,k\le n$. Then
\[|\mb{E}[e^{it\sum_{j=1}^n a_jx_j}]|\le (n+1)\exp[-2p(1-p)t^2\on{Var}[a_j]n/\pi^2],\]
where $\on{Var}[a_j]$ is the variance of the random variable $a_J$, if $J$ is an index uniformly drawn from $[n]$.
\end{lemma}
\begin{proof}
First note that $\binom{n}{k}(\frac{s}{n})^k(\frac{n-s}{n})^{n-k}$ is maximized at $k = s$ and since these values sum to $1$ we have
\[\binom{n}{s}\bigg(\frac{s}{n}\bigg)^s\bigg(\frac{n-s}{n}\bigg)^{n-s}\ge \frac{1}{n+1}.\]
The key idea is that 
\begin{align*}
|\mb{E}[e^{it\sum_{j=1}^{n}a_jx_j}]| &= \bigg|\frac{1}{2\pi i}\oint_{|z| = 1}\frac{\prod_{j=1}^{n}(pe^{ita_j}z+(1-p))}{\binom{n}{s}\big(\frac{s}{n}\big)^s\big(\frac{n-s}{n}\big)^{n-s}z^{s+1}}dz\bigg|\\
&\le \frac{n+1}{2\pi}\Bigg(2\pi \max_{|z| =1}\bigg|\prod_{j=1}^{n}(pe^{ita_j}z+(1-p))\bigg|\Bigg)\\
&\le (n + 1)\bigg(\max_{|z| = 1} \frac{1}{n}\sum_{j=1}^{n}|(pe^{ita_j}z+(1-p))|^2\bigg)^{n/2}\\
&= (n + 1)\bigg(\max_{|z| = 1} \frac{1}{n}\sum_{j=1}^{n}(p^2 + (1-p)^2 + p(1-p)e^{ita_j}z + p(1-p)e^{-ita_j}\bar{z})\bigg)^{n/2}\\
&\le (n + 1)\bigg(1 - 2p(1-p) + 2p(1-p)\Big|\frac{1}{n}\sum_{j=1}^{n} e^{ita_j}\Big|\bigg)^{n/2}\\
&= (n + 1)\bigg(1 - 2p(1-p) + 2p(1-p)\Big(\frac{1}{n^2}\sum_{j=1}^{n}\sum_{k=1}^{n} \cos((a_j-a_k)t)\Big)^{1/2}\bigg)^{n/2}.
\end{align*}
We now use the elementary facts that $\cos x\le 1 - 2(x/\pi)^2$ for $|x|\le \pi$, $1 + x\le e^{x}$ for all $x\in \mb{R}$, and $\sqrt{1-t}\le 1-t/2$ for $t\in [-1,1]$. Then it follows that 
\begin{align*}
(n + 1)\bigg(1 - &2p(1-p) + 2p(1-p)\Big(\frac{1}{n^2}\sum_{j=1}^{n}\sum_{k=1}^{n} \cos((a_j-a_k)t)\Big)^{1/2}\bigg)^{n/2}\\
&\le (n + 1)\bigg(1 - 2p(1-p) + 2p(1-p)\Big(\frac{1}{n^2}\sum_{j=1}^{n}\sum_{k=1}^{n} 1-\frac{2(a_j-a_k)^2t^2}{\pi^2}\Big)^{1/2}\bigg)^{n/2}\\
&\le (n + 1)\bigg(1 - 2p(1-p) + 2p(1-p)\Big(\frac{1}{n^2}\sum_{j=1}^{n}\sum_{k=1}^{n} 1-\frac{(a_j-a_k)^2t^2}{\pi^2}\Big)\bigg)^{n/2}\\
&= (n + 1)\bigg(1- \Big(\frac{1}{n^2}\sum_{j=1}^{n}\sum_{k=1}^{n} \frac{2p(1-p)(a_j-a_k)^2t^2}{\pi^2}\Big)\bigg)^{n/2}\\
&= (n + 1)\bigg(1- \frac{4p(1-p)\on{Var}[a_j]t^2}{\pi^2}\bigg)^{n/2}\\
&\le (n+1)\exp\bigg[\frac{-2p(1-p)\on{Var}[a_j]t^2n}{\pi^2}\bigg]
\end{align*}
and the result follows. It is worth noting that in the second line, the expression within the square root can be verified to still be nonnegative, so the application of the inequalities above is valid.
\end{proof}

\subsection{Decoupling methods}\label{sub:decoupling-methods}
\begin{definition}\label{alpha}
Define the operator $\alpha$ on functions of the form $f(X,Y_1,\ldots,Y_k)$, which outputs the function $\alpha(f)$ given by
\[
\alpha(f)(X,Y_1^{0}, Y_1^{1}, \ldots, Y_k^{0}, Y_k^{1}) :=\sum_{\mbf{v}\in \{0,1\}^k}(-1)^{|\mbf{v}|}f(X,Y^\mbf{v})
\]
where $Y^{\textbf{v}} = (Y_1^{v_1},\ldots,Y_k^{v_k})$. For the sake of notational simplicity define $\mbf{Y} = (Y_1^{0}, Y_1^{1}, \ldots, Y_k^{0}, Y_k^{1})$. 
\end{definition}
This definition initially may seem opaque. However, a van der Corput-style Cauchy--Schwarz argument will allow us to utilize this definition in a critical way. Note that any function which is not dependent on all components of the $Y_{i}$'s is in the kernel of the operator; following \cite{B2}, we often refer to the remaining such functions as \emph{rainbow} functions or terms. The key lemma we need will be a modification of the one in \cite{B2} and thus we repeat the original version; the proof is identical to the one given in \cite{B2}.
\begin{lemma}\label{lem:vdC-indep}
Let $k\ge 0$ and let $(X,Y_1,\ldots,Y_k)$ be mutually independent random variables. Let $\varphi(t) = \mb{E}_{X,Y_1,\ldots,Y_k}e^{itf(X,Y_1,\ldots,Y_k)}$. Then 
\[|\varphi(t)|^{2^k}\le \mb{E}_{\mbf{Y}}|\mb{E}_{X} e^{it\alpha(f)(X,\mbf{Y})}|,\]
where $\mbf{Y} = (Y_1^{0}, Y_1^{1}, \ldots, Y_k^{0}, Y_k^{1})$.
\end{lemma}
\begin{proof}
We prove this by induction on $k$; note that $k=0$ is trivial. For induction define $X' = (X, Y_{k+1})$ and $f^\ast(X',Y_1,\ldots,Y_k) = f(X,Y_1,\ldots,Y_{k+1})$. Note that, applying the inductive hypothesis, 
\begin{align*}
|\varphi(t)|^{2^k} &\le \mb{E}_{Y_1^{0}, Y_1^{1}, \ldots, Y_k^{0}, Y_k^{1}}\big|\mb{E}_{X'} e^{it\alpha(f^\ast)(X',Y_1^{0}, Y_1^{1}, \ldots, Y_k^{0}, Y_k^{1})}\big|\\
&\le\mb{E}_{Y_1^0,Y_1^1,\ldots,Y_k^0,Y_k^1}\mb{E}_X\big|\mb{E}_{Y_{k+1}}e^{it\alpha(f^\ast)(X',Y_1^0,Y_1^1,\ldots,Y_k^0,Y_k^1)}\big|.
\end{align*}
By Cauchy--Schwarz and the triangle inequality we thus have
\begin{align*}
|\varphi(t)|^{2^{k+1}}&\le\mb{E}_{Y_1^0,Y_1^1,\ldots,Y_k^0,Y_k^1}\mb{E}_X\big|\mb{E}_{Y_{k+1}}e^{it\alpha(f^\ast)(X',Y_1^0,Y_1^1,\ldots,Y_k^0,Y_k^1)}\big|^2
\\
&=\mb{E}_{Y_1^{0}, Y_1^{1}, \ldots, Y_k^{0}, Y_k^{1}}\mb{E}_{(X,Y_{k+1}^{0},Y_{k+1}^{1})}e^{it\alpha(f^\ast)((X,Y_{k+1}^{0}),Y_1^{0}, Y_1^{1}, \ldots, Y_k^{0}, Y_k^{1})}e^{-it\alpha(f^\ast)((X,Y_{k+1}^{1}),Y_1^{0}, Y_1^{1}, \ldots, Y_k^{0}, Y_k^{1})}\\
&\le \mb{E}_{Y_1^{0}, Y_1^{1}, \ldots, Y_k^{0}, Y_k^{1},Y_{k+1}^{0},Y_{k+1}^{1}}\big|\mb{E}_{X}e^{it\alpha(f)(X,Y_1^{0}, Y_1^{1}, \ldots, Y_{k+1}^{0}, Y_{k+1}^{1})}\big|
\end{align*}
and the result follows.
\end{proof}
The need for a more complicated version of the above lemma stems from the fact that we will need to consider $f(X,Y_{1},\ldots,Y_{k})$ where $X$, $Y_j$ are tuples of Bernoulli random variables that are then conditioned on the total sum. Let $A$ and $B_j$ for $j\in [k]$ be sets which partition the index set $[n]$. Take independent Bernoulli random variables $(x_i)_{i\in [n]}$, then condition on the event $\sum_{i\in [n]}x_i = S$. Define random variables $X = (x_i)_{i\in A}$ and $Y_j = (x_i)_{i\in B_j}$ and random variables $Z_0 = |X| = \sum_{i\in A}x_i$ and $Z_j = |Y_j| = \sum_{i\in B_j}x_i$. Note that $Z_0+\sum_{j\in [k]}Z_j = S$. The key idea is that given $(Z_j)_{1\le j\le k}$ the random variables $X,Y_j$ are conditionally independent.
\begin{lemma}\label{lem:vdC-dep}
Let $(x_{i})_{i\in [n]}$ be a sequence of independent, identically distributed, Bernoulli random variables conditioned on the event $\sum_{i\in [n]}x_i = S$. Let $A, B_j$ be sets as above with associated random variables $X, Y_j, Z_0, Z_j$ and let $\varphi(t) = \mb{E}_{X,Y_1,\ldots,Y_k}e^{itf(X,Y_1,\ldots,Y_k)}$. Then 
\[|\varphi(t)|^{2^k}\le\mb{E}_{Z_1,\ldots,Z_k,Y_1^0, Y_1^1, \ldots, Y_k^0, Y_k^1}\bigg|\mb{E}_X\bigg[e^{it\alpha(f)(X,\mbf{Y})}|Z_1,\ldots,Z_k\bigg]\bigg|,\]
where the variables $Y_j^\ell$ are conditionally independent $\{0,1\}$-vectors given $(Z_j)_{1\le j\le k}$, equal to a uniform random vector with $Z_j$ values of $1$.
\end{lemma}
\begin{remark}
Note that the presence of $Z_1,\ldots,Z_k$ in the outer expectation is irrelevant, although it serves to clarify the joint distribution of the $Y_j^b$ variables and is useful to the proof method.
\end{remark}
\begin{proof}
By iterated Cauchy--Schwarz or H\"older's inequality,
\begin{align*}
|\varphi(t)|^{2^k} &= \bigg|\mb E_{Z_1,\ldots,Z_k}\mb E_{X,Y_1,\ldots,Y_k}\left[e^{itf(X,Y_1,\ldots,Y_k)}|Z_1,\ldots,Z_k\right]\bigg|^{2^k}\\
&\le\mb E_{Z_1,\ldots,Z_k}\bigg|\mb E_{X,Y_1,\ldots,Y_k}\left[e^{itf(X,Y_1,\ldots,Y_k)}|Z_1,\ldots,Z_k\right]\bigg|^{2^k}.
\end{align*}
Since $Z_0$ is determined by $Z_1,\ldots,Z_k$, we see that $X,Y_1,\ldots,Y_k$ are mutually conditionally independent given $Z_1,\ldots,Z_k$. Hence, by \cref{lem:vdC-indep}, we conclude
\[|\varphi(t)|^{2^k}\le\mb E_{Z_1,\ldots,Z_k}\left[\mb E_{Y_1^0,Y_1^1,\ldots,Y_k^0,Y_k^1}\left[\left|\mb E_X[e^{it\alpha(f)(X,\mbf{Y})}|Z_1,\ldots,Z_k]\right||Z_1,\ldots,Z_k\right]\right],\]
which equals the desired.
\end{proof}

\subsection{Graph factors}\label{sub:graph-factors}
Finally, we define a notion of graph factors that will be critical for the remainder of our analysis regarding graphs. These notions are critical in previous work by Janson \cite{J1} and other results concerning the method of projections. Suppose a random graph is sampled, with the indicator of edge $e\in\binom{[n]}{2}$ denoted $x_e$. Let $\chi_{e} = (x_{e}-p)/\sqrt{p(1-p)}$. Note that if $x_e\sim\text{Ber}(p)$ then $\chi_{e}$ has mean $0$ and variance $1$.
\begin{definition}
Fix a graph $H$ with no isolated vertices and an integer $n\ge |V(H)|$. Then define
\[\gamma_H(\mbf{x}) = \sum_{\substack{E\subseteq\binom{[n]}{2}\\E\simeq H}}\chi_E\]
where $\chi_{S} = \prod_{e\in S}\chi_{e}$. Here $\simeq$ denotes graph isomorphism, specifically between $H$ and the graph spanned by the edges $E$. We call $\gamma_H(\mbf{x})$ the \emph{graph factor} corresponding to the graph $H$.
\end{definition}
\begin{remark}
The empty graph $K_0$ has no isolated vertices, and appears as a subgraph of $\binom{[n]}{2}$ exactly once, so $\gamma_{K_0}(\mbf{x}) = 1$.
\end{remark}
The key property of this collection of functions it that they are orthogonal when the graph is sampled via $G(n,p)$, i.e.,
\begin{equation}\label{eq:orthogonality}
\mb{E}_{G(n,p)}[\chi_e\chi_{e'}] = \mbm{1}_{e=e'}.
\end{equation}
Additionally, the various graph theoretic functions we consider will be expressible in this basis. One of the key motivating results of Janson \cite{J1} is that in the $G(n,p)$ model, for a set of distinct connected graphs $H_1,H_2,\ldots,H_a$, the vector $(\gamma_{H_b}(\mbf{x}))_{b\in [a]}$ (scaled appropriately) approaches a vector of independent Gaussians.

\subsection{Graph notations} Given a graph $G$, we write $V(G)$ and $E(G)$ for the vertex and edge sets, $\ol{E}(G)$ for the set $\binom{V(G)}{2}\setminus E(G)$, and $v(G) = |V(G)|$, $e(G) = |E(G)|$, $\ol{e}(G) = |\ol{E}(G)|$.

\section{Bounds for graph characteristic functions}\label{sec:general-graph-characteristic-functions}
In this section we prove bounds for characteristic functions of quite general random variables associated to graphs, including connected subgraph counts and induced subgraph counts (aside from some potentially problematic edge-counts, as discussed in \cref{sec:introduction}). We will prove these results in the $G(n,m)$ model.

\subsection{Setup}\label{sub:graph-characteristic-setup}
Let $\ell\ge 3$ be an integer and $\lambda\in(0,1/2)$ be a real parameter. We define the notion of a well-behaved graph statistic of ``degree $\ell$'', which will be the central object for our general analysis in this section.
\begin{definition}
An $(\ell,\lambda)$\emph{-factor system} is the following data. Let $\mc{H} = \{H_1,\ldots, H_a\}$ be a set of (nonisomorphic) graphs on at most $\ell$ vertices with no vertex isolated. Further suppose that $\mc{H'}\subseteq\mc{H}$ is such that for each $3\le k\le\ell$ it contains a connected graph on $k$ vertices. Let $\mc{H}_k$ for $0\le k\le\ell$ denote the subset of $\mc{H}$ with $k$ vertices. Let $x_e$ for $e\in\binom{[n]}{2}$ be a random graph drawn according to $G(n,m)$ such that $p = m/\binom{n}{2}\in (\lambda,1-\lambda)$, and let $\chi_e = (x_e-p)/\sqrt{p(1-p)}$ as before. Now consider a linear combination of graph factors,
\begin{equation}\label{eq:general-form}
W = \sum_{H\in\mc{H}}n^{\ell-v(H)}\Delta_H\gamma_H(\mbf{x}),
\end{equation}
where the $\Delta_H$'s are arbitrarily reals chosen such that $|\Delta_H|\le 1/\lambda$ for $H\in\mc{H}$ and $|\Delta_H|\ge\lambda$ for $H\in\mc{H'}$. We call $W$ an $(\ell,\lambda)$\emph{-graph statistic}, or an $(\ell,\lambda)$\emph{-statistic} for short.
\end{definition}
\begin{remark}
In applications, in particular counting subgraphs and induced subgraphs, we will choose the $\Delta_H$'s to be specific functions of $p$ (and to a lesser extent $n$), and $p=m/\binom{n}{2}$ will be constrained so that the $H\in\mc{H'}$ terms satisfy a uniform lower bound as $n\to\infty$ as required by the definition. The purpose of $\mc{H'}$ is that sometimes we may have some terms ``self-cancel'' (e.g. for certain induced subgraph counts) and be exactly zero in the $G(n,m)$ setting. $\mc{H'}$ merely contains guaranteed non-canceling terms that we will need to establish the desired bounds.
\end{remark}

Associated to an $(\ell,\lambda)$-statistic $W$ with factor system $(\mc{H},\mc{H'})$, we will typically write
\[W_k = \sum_{H\in\mc{H}_k}n^{\ell-v(H)}\Delta_H\gamma_H(\mbf{x}),\]
the \emph{order $k$ portion} of $W$. We now define a normalized version of the statistic $W$. Write $\sigma$ for the standard deviation of $W_3$ in the $G(n,p)$ model. We easily see by orthogonality in the $G(n,p)$ model \cref{eq:orthogonality} that, since $W_3$ has a term from $\mc{H'}$, $\sigma/n^{\ell-3} = \Theta_\lambda(n^{3/2})$. Subsequently we will drop all asymptotic dependence on $\mc{H},\mc{H'}$ but keep the $\lambda$ dependence, as it helps clarify the instances where we need coefficients of certain terms $\gamma_H$ to be ``of the correct size''.

As it turns out, $\sigma$ and the standard deviation of $W$ in the $G(n,m)$ model are essentially the same. In particular, $\sigma = \sigma_W(1+O_{\lambda,\eps}(n^{\eps-1/2}))$ for any $\eps>0$. We will comment further on this phenomenon in \cref{sub:standard-deviation-considerations}; however, this is ``as expected'' if we believe that the ``only real effect'' of the $G(n,m)$ model is to constrain the sum of $x_e$ (and hence $\chi_e$). For technical reasons, it will be more convenient for us to work with $\sigma$ initially, and to later use its closeness to $\sigma_W$ to transfer any necessary results. Now we define
\[\mc{K} = \frac{W-W_0-W_2}{\sigma}\]
to be the \emph{normalized version} of the statistic $W$ and let
\[\varphi_{\mc{K}}(t) = \mb Ee^{it\mc{K}}\]
(note there is no graph on one vertex without isolated vertices, so $W_1$ never exists). This is the characteristic function that we will study in depth and is the key object of study. It is worth noting that $\mc{K}$ does not even have mean zero in the $G(n,m)$ model, only in the $G(n,p)$ model. The true mean will also be discussed in \cref{sub:standard-deviation-considerations}.

To see this last remark, note that $\mc{K}$ is a multilinear polynomial in the $\chi_e$'s without a constant term, and the $\chi_e$'s are independent mean $0$ variance $1$ random variables under $G(n,p)$. We will use these facts about the $G(n,p)$ model often without comment. Finally, it is worth noting that on occasion we may switch between the probabilistic models $G(n,m)$ and $G(n,p)$. The switches between models will be clearly marked in the exposition; the results of this section are ultimately about the $G(n,m)$ model and $\mc{K}$ is understood to be drawn from that model.

\subsection{Bounds for $|t|\le n^\eps$}\label{sub:bounds-low}
$\mc{K}$ should limit towards the Gaussian $\mc{N}(0,1)$ as $n\to\infty$, which morally means the desired characteristic function bounds will hold for small $t$. For example, \cite{J1} shows such a central limit theorem for subgraph counts of $G(n,m)$. However, we need quantitative convergence in order to deduce the desired bounds. Furthermore, in $G(n,m)$, the expressions $\gamma_H$ are not orthogonal unlike in $G(n,p)$, which complicates the usual techniques for proving effective central limit theorems.

Therefore, we will first quantitatively show that $W_3/\sigma$ tends to $\mc{N}(0,1)$ in the $G(n,p)$ model, and then transfer over to the $G(n,m)$ setting. Along the way we will also deal with the deviations introduced by the terms $W_k$ for $k\ge 4$; we note that they will be lower-order in size. The proof we give that $W_3/\sigma$ limits to a Gaussian is similar to earlier work on a quantitative central limit theorem \cite{BKR89}. Also, the qualitative version follows from earlier work of \cite{J1}. However, we could not find work making the convergence in our setting quantitative and hence we provide a proof.
\begin{lemma}\label{lem:three-connected-Gaussian}
Let $W$ be an $(\ell,\lambda)$-statistic. Suppose the $\chi_e$ are drawn as in the $G(n,p)$ model, so that they are i.i.d. Then \[\on{Wass}\left(\frac{W_3}{\sigma}, \mathcal{N}(0,1)\right)\lesssim_\lambda\frac{1}{\sqrt{n}},\]
where $\on{Wass}(A,B) := \sup_{\on{Lip}(f)\le 1}|\mb{E}[f(A)-f(B)]|$ denotes the Wasserstein metric.
\end{lemma}
\begin{proof}
We model our proof on the method of dependency graphs for proving quantitative central limit theorems. Let $Y = W_3/\sigma$. Let $I$ be an indexing set for the monomials $R_j, j\in I$ in the expansion of $X_3$, with the constant coefficients included. We can think of $I$ as being the set of triangles and length $2$ paths within a complete graph on $n$ vertices (or just triangles, or just length $2$ paths). For $j\in I$, let $N_j$ be the set of $i\in I$ such that $R_j,R_i$ are supported on sets of variables $\chi_e$ that overlap (i.e., some $\chi_e$ is in both). Note $j\in N_j$. Let $Y_j = 1/\sigma\sum_{k\notin N_j}R_k$. Note that $R_j, Y_j$ are independent. We also record that $\mb E[Y] = 0$ and $\sigma/n^{\ell-3} = \Theta_\lambda(n^{3/2})$, already noted earlier. This stems from the fact that the $\chi_e$ are centered and i.i.d.

The key fact we will use is a version of Stein's Lemma (see e.g. \cite[Theorem~3.1]{Ro11}) which states
\[\on{Wass}(S,\mathcal{N}(0,1))\le \sup\{|\mb{E}f'(S)-Sf(S)|:\|f\|_{\infty}\le 1, \|f'\|_{\infty}\le\sqrt{2/\pi},\|f''\|_{\infty}\le2\}.\]
Now
\begin{align*}
\bigg|\mb{E}[Yf(Y)-f'(Y)]\bigg|&=\bigg|\frac{1}{\sigma}\sum_{j\in I}\mb{E}[R_jf(Y)]-\mb{E}[f'(Y)]\bigg|\\
&\le\frac{1}{\sigma}\bigg|\sum_{j\in I}\mb{E}[R_j(f(Y)-f(Y_j))-R_j(Y-Y_j)f'(Y)]\bigg|\\
&\quad+ \bigg|\frac{1}{\sigma}\sum_{j\in I}\mb{E}[R_j(Y-Y_j)f'(Y)]-\mb{E}[f'(Y)]\bigg|\\
\end{align*}
since $\mb E[R_j] = 0$ and $R_j, f(Y_j)$ are independent. We now bound each of the terms separately. For the first,
\begin{align*}
&\frac{1}{\sigma}\bigg|\sum_{j\in I}\mb{E}[R_j(f(Y)-f(Y_j))-R_j(Y-Y_j)f'(Y)]\bigg|\\
&\qquad\le\frac{1}{2\sigma}\bigg|\sum_{j\in I} \mb{E}[\|f''\|_{\infty}R_j(Y-Y_j)^2]\bigg|\lesssim_\lambda\frac{n^{\ell-3}}{\sigma}\bigg|\sum_{j\in I}\mb{E}[(Y-Y_j)^2]\bigg|\\
&\qquad=\frac{n^{\ell-3}}{\sigma}\sum_{j\in I}\frac{1}{\sigma^2}\mb{E}\bigg[\bigg(\sum_{k\in N_j}R_k\bigg)^2\bigg]\lesssim_\lambda n^4(n^{\ell-3}\sigma^{-1})^3\lesssim_\lambda n^{-\frac{1}{2}}.
\end{align*}
We used that $|I| = O(n^3)$ and $|N_j| = O(n)$ so that the inner expectation has $O(n)$ nonzero terms (by orthogonality \cref{eq:orthogonality}, we can deduce $\mb E[R_iR_j] = 0$ for $i\neq j$). Here we also used that each $R_j$ has coefficients of size $O_\lambda(n^{\ell-3})$.

For the second term note that
\[\frac{1}{\sigma}\sum_{j\in I}\mb{E}[R_j(Y-Y_j)] = \frac{1}{\sigma}\sum_{j\in I}\mb{E}[R_jY] = \mb{E}[Y^2] = 1\]
since $R_j,Y_j$ are independent. Thus
\begin{align*}
&\bigg|\frac{1}{\sigma}\sum_{j\in I}\mb{E}[R_j(Y-Y_j)f'(Y)]-\mb{E}[f'(Y)]\bigg| = \bigg|\mb{E}\bigg[f'(Y)\bigg(\frac{1}{\sigma}\sum_{j\in I}{R_j(Y-Y_j)} - 1\bigg)\bigg]\bigg|\\
&\lesssim\bigg(\frac{1}{\sigma^2}\on{Var}\bigg[\sum_{j\in I}R_j(Y-Y_j)\bigg]\bigg)^{\frac{1}{2}}=\bigg(\frac{1}{\sigma^4}\sum_{\substack{k,m\in I\\t\in N_k, s\in N_{m}}}(\mb{E}[R_kR_tR_mR_s]-\mb{E}[R_kR_t]\mb{E}[R_mR_s])\bigg)^{\frac{1}{2}}\\
&=\frac{1}{\sigma^2}\bigg(\sum_{\substack{k,m\in I\\t\in N_k\setminus{k}, s\in R_{m}\setminus{m}}}(\mb{E}[R_kR_tR_mR_s]-\mb{E}[R_kR_t]\mb{E}[R_mR_s])\\
&\qquad\qquad\qquad+2\sum_{\substack{k,m\in I\\s\in N_{m}\setminus{m}}}(\mb{E}[R_k^2R_mR_s]-\mb{E}[R_k^2]\mb{E}[R_mR_s])+\sum_{k,m\in R}(\mb{E}[R_k^2R_m^2]-\mb{E}[R_k^2]\mb{E}[R_m^2])\bigg)^{\frac{1}{2}}
\end{align*}
Note that $\mb{E}[R_mR_s] = 0$ if $m\neq s$ so the above simplifies to 
\[\lesssim_\lambda\frac{1}{n^{2\ell-3}}\bigg(\sum_{\substack{k,m\in I\\t\in N_k\setminus{k}, s\in N_{m}\setminus{m}}}\mb{E}[R_kR_tR_mR_s] + 2\sum_{\substack{k,m\in I\\s\in N_{m}\setminus{m}}}\mb{E}[R_k^2R_mR_s]+\sum_{k,m\in I}(\mb{E}[R_k^2R_m^2]-\mb{E}[R_k^2]\mb{E}[R_m^2])\bigg)^{\frac{1}{2}}.\]
Now
\begin{equation}\label{eq:three-connected-1}
\sum_{k,m\in I}\mb{E}[R_k^2R_m^2]-\mb{E}[R_k^2]\mb{E}[R_m^2]\lesssim_\lambda n^4(n^{\ell-3})^4
\end{equation}
as $R_m^2$ and $R_k^2$ are independent unless they intersect in an edge and there are $O(n^4)$ such configurations. Next,
\begin{equation}\label{eq:three-connected-2}
\sum_{\substack{k,m\in I\\s\in N_{m}\setminus{m}}}\mb{E}[R_k^2R_mR_{s}] \lesssim_\lambda n^4(n^{\ell-3})^4.
\end{equation}
This is because in order for the term to be nonzero, since $m\neq s$, either $m$ or $s$ must share an edge with $k$, say $m$. If $m$ shares two edges or more, then it is contained within the same vertices as $k$, and we see that $s$ must have this property as well, leading to $O(n^3)$ configurations. If it shares exactly one edge with $k$, then we see $k, m$ span four vertices, and it is easy to see that $s$ cannot introduce a new vertex else some edge will have multiplicity $1$ in the term $R_k^2R_mR_s$. This leads to $O(n^4)$ configurations, and hence justified \cref{eq:three-connected-2}.

Finally,
\begin{equation}\label{eq:three-connected-3}
\sum_{\substack{k,m\in I\\t\in N_k\setminus{k}, s\in N_{m}\setminus{m}}}\mb{E}[R_kR_tR_mR_s]\lesssim_\lambda n^5(n^{\ell-3})^4.
\end{equation}
The reasoning is as follows. Note that $k, t$ overlap on an edge as do $m,s$, and since $t\neq k$ and $s\neq m$ we have that $R_kR_t$ and $R_mR_s$ share some edge as well. So the total graph spanned is connected. Additionally, each edge must be covered at least twice to be nonzero. Now consider how many edges, with multiplicity, are spanned by a term $R_kR_tR_mR_s$, or equivalently the degree of the term.

If the degree is at most $9$ for a nonzero term, each distinct edge is covered at least twice so we must have at most $4$ total edges in the resulting graph. The configuration is connected so has at most $5$ vertices. Hence there are $O(n^5)$ configurations.

For degrees $10$ and $11$, in a nonzero term there must be at least one triangle present among the $R$ factors, and there are at most $5$ edges in the support. Again we see that this leads to at most $5$ vertices. So we have $O(n^5)$ configurations again.

Finally, for degree $12$, every $R$ must correspond to a triangle. Hence in a nonzero term $R_k, R_t$ must be two triangles attached on an edge, and same for $R_m, R_s$. The only way to have a nonzero term is to superimpose these in some way. Thus there are again $4$ vertices, so $O(n^4)$ configurations. This justifies \cref{eq:three-connected-3}.

Using \cref{eq:three-connected-1,eq:three-connected-2,eq:three-connected-3}, we obtain the result.
\end{proof}

Since $x\mapsto\exp(itx)$ is $t$-Lipschitz, this allows for a comparison of the characteristic functions of $W_3/\sigma$ and $\mc{N}(0,1)$ when drawing from the $G(n,p)$ model, which is essentially what we need in this range as outlined in \cref{sub:overview}. However, we need to transfer this information about the $G(n,p)$ model into estimates on the $G(n,m)$ model.
\begin{lemma}\label{lem:low-t}
Let $W$ be an $(\ell,\lambda)$-statistic with normalized version $\mc{K}$. Then for all $\eps > 0$ and $t\in\mb{R}$ we have
\[|\mb{E}[e^{it\mc{K}}]-e^{-t^2/2}|\lesssim_{\lambda,\eps}\frac{|t|}{n^{1/2-\eps}}.\]
\end{lemma}
\begin{remark}
Recall that for $\mc{K}$ the edges are drawn from $G(n,m)$.
\end{remark}
\begin{proof}
We first couple the independent and fixed-size models in order to compare their information. Sample $\chi' = (\chi_{ij}')_{1\le i<j\le n}$ from the $G(n,p)$ model. Then adjust a uniformly random subset of the edges (or non-edges, depending on if there are too many or too few edges) to obtain exactly $m$ edges. Call the resulting random variable $\chi$. By symmetry considerations, $\chi$ is distributed as if it came from the $G(n,m)$ model.

By Azuma--Hoeffding, with probability $1-\exp(-\Omega_\lambda((\log n)^2))$ the number of edges adjusted is $O_\lambda(n\log n)$. Now let $Y = W_3/n^{\ell-3}$ (different than in the proof of \cref{lem:three-connected-Gaussian}), so that the coefficients are constant-order. We claim that for any $\eps > 0$,
\begin{equation}\label{eq:low-t-coupling}
\mb{P}[|Y[\chi]-Y[\chi']|\ge n^{1+\eps}]\lesssim\exp(-\Omega_\lambda((\log n)^2)).
\end{equation}
Let $S_0$ be the set of edges that are changed and $S_1$ be the direction they changed in. Given these two variables, note that $\chi_e,\chi_e'$ are determined for $e\in S_0$, and $\chi_e = \chi_e'$ for $e\notin S_0$. We can therefore expand $Y[\chi]-Y[\chi']$ as a degree at most $3$ polynomial in the variables $\chi_e$, $e\notin S_0$. The coefficients are functions of $S_0,S_1$. Note that the monomials $\chi_T$ are supported only on $|T|\le 3$. We therefore can write
\begin{equation}\label{eq:low-t-expansion}
Y[\chi]-Y[\chi'] = \sum_{|T|\le 3, T\cap S_0 = \emptyset}c_T\chi_T.
\end{equation}

We show that the sum of the squares of the coefficients of this random polynomial is small, with high probability over the randomness of $S_0,S_1$. As noted, with high probability we have $|S_0|\lesssim_\lambda n\log n$. Next, the polynomial \cref{eq:low-t-expansion} only has terms that have interacted with $\chi_e$ for $e\in S_0$, i.e., is supported on $\chi_T$ for which there is a nonempty set $U\subseteq S_0$ of edges with $T\cup U$ a triangle or path of length $2$. Thus in fact \cref{eq:low-t-expansion} is degree at most $2$.

Each $e\in S_0$ is in $O(n)$ triangles or paths of length $2$, so the total number of nonzero coefficients in \cref{eq:low-t-expansion} is $O_\lambda(n^2\log n)$ with high probability. We now bound the sum of squares of coefficients for various types of terms $\chi_T$ in magnitude. It is worth noting before we do this that conditional on $|S_0|, S_1$, we have that the set $S_0$ of edges is uniformly distributed among all possible subsets of the given size.
\begin{enumerate}[1.]
    \item For $|T|=2$, since the coefficients of $Y$ are bounded by $\lambda$, we obtain a total contribution of $O_\lambda(n^2\log n)$ with high probability since \cref{eq:low-t-expansion} has that many terms with high probability.
    \item For $|T|=1$, say $T=\{e\}$, this coefficient is (up to a constant depending on $\lambda$) bounded by the number of triangles or length $2$ paths which contain $e$ and such that the remaining edges are in $S_0$. Now we fix some value $|S_0|\lesssim_\lambda n\log n$ and look over the randomness of $S_0$. We care only about edges incident to $e$, and in particular, the desired coefficient is linearly bounded by how many edges there are.
    
    Now, if we choose $S_0$ instead by selecting each of the $\binom{n}{2}$ edges with probability $(\log n)^2/n$, say, we see that with probability at least $1/3$, say, the number of resulting edges is greater than the desired value of $|S_0|$ fixed above. By symmetry and monotonicity of edge counts, we see that the probability of too many edges incident to $e$, under the desired model, is at most thrice the probability under the independent model. And in the independent model, by Chernoff, with probability $\exp(-(\log n)^2)$ the desired edge count is $O_\lambda((\log n)^2)$ hence the coefficient is $O_\lambda((\log n)^2)$. The total contribution to the sum of squares is $O_\lambda(n^2(\log n)^4)$ with high probability.
    \item For the constant term we are bounding the number of length $2$ paths and triangles which appear in $S_0$. Note that with high probability any edge of $S_0$ appears in $O_\lambda((\log n)^2)$ triangles or length $2$ paths in total by the analysis in item $2$ and thus this coefficient is bounded by $|S_0|(\log n)^2\lesssim_\lambda n(\log n)^3$ with high probability, giving a contribution of $O_\lambda(n^2(\log n)^6)$ to the sum of squares.
\end{enumerate}
In conclusion, using a union bound, we have shown that over the randomness of $S_0,S_1$, with probability $1-\exp(-\Omega_\lambda((\log n)^2))$, the sum of squares of the coefficients in \cref{eq:low-t-expansion} is $O_\lambda(n^2(\log n)^6)$, and also $|S_0| = O_\lambda(n\log n)$.

Now fix $S_0, S_1$ in such a case. We then sample every edge $e\notin S_0$ with a fixed sum depending on $|S_0|$ and $S_1$, since this has the correct distribution for $\chi$. If instead we sampled $e\notin S_0$ from $G(n,p)$, by hypercontractivity (\cref{thm:concentration-hypercontractivity}) on the polynomial \cref{eq:low-t-expansion} we have $|Y[\chi']-Y[\chi]|\ge n^{1+\eps}$ with probability $\exp(-\Omega_\lambda(n^{\eps'}))$. Since $|S_0|\lesssim_\lambda n\log n$, we see that we sample the correct total number of edges $e\notin S_0$ with probability $\exp(-O_\lambda((\log n)^2))$, which is much bigger than the failure probability above. Therefore, if we condition on drawing the right amount of edges, the quality of this estimate is preserved. Thus
\[\mb P[|Y[\chi']-Y[\chi]|\ge n^{1+\eps}|S_0,S_1]\le\exp(-\Omega_\lambda(n^{\eps'}))\]
for every $S_0,S_1$ in one of the cases delineated above. The remainder of cases occur with probability $\exp(-\Omega_\lambda((\log n)^2))$ over the randomness of $S_0,S_1$.

Overall, therefore, we have proven the desired \cref{eq:low-t-coupling}:
\[\mb P[|Y[\chi']-Y[\chi]|\ge n^{1+\eps}]\lesssim\exp(-\Omega_\lambda((\log n)^2)).\]
To complete the proof we note that
\[\mc{K} = \frac{n^{\ell-3}Y}{\sigma}+\frac{\sum_{k\ge 4}W_k}{\sigma}.\]
Letting $\mc{K}_{\text{rem}}$ denote the latter term, thus 
\begin{align*}
|\mb{E}[e^{it\mc{K}}]-e^{-t^2/2}|&\le\mb{E}[\min(|t\mc{K}_{\text{rem}}|,2)] + |\mb{E}[e^{itn^{\ell-3}Y/\sigma}]-e^{-t^2/2}|\\
&\le\mb{E}[\min(|t\mc{K}_{\text{rem}}|,2)] + |\mb{E}[e^{itn^{\ell-3}Y/\sigma}-e^{itn^{\ell-3}Y'/\sigma}]| + |\mb{E}[e^{itn^{\ell-3}Y'/\sigma}]-e^{-t^2/2}|\\
&\lesssim_{\lambda,\eps} |t|/n^{1/2-\eps} + |t|/n^{1/2-\eps} + |t|/\sqrt{n}+\exp(-\Omega_\lambda((\log n)^2)),
\end{align*}
where $Y=Y[\chi]$, $Y'=Y[\chi']$ are coupled together as described above. The last inequality is derived as follows. The third term comes from the fact that $x\mapsto e^{itx}$ is $t$-Lipschitz and using \cref{lem:three-connected-Gaussian} (technically we have to separate into real and imaginary parts). The second term comes from $\sigma/n^{\ell-3} = \Theta_\lambda(n^{3/2})$ and the distance estimate \cref{eq:low-t-expansion} on $Y$ vs. $Y'$.

The first term comes from hypercontractivity (\cref{thm:concentration-hypercontractivity}) once more: $\sigma\mc{K}_{\text{rem}}$ has standard deviation $O_\lambda(n^{\ell-2})$ in the $G(n,p)$ model, hence is at least $n^\eps$ times that with probability $\exp(-\Omega_\lambda(n^{\eps'}))$. With probability $\Theta_\lambda(1/n)$ we draw exactly $m$ edges, hence the integrity of this concentration estimate is maintained even if we condition on drawing $m$ edges. The failure probability is absorbed into the fourth error term when multiplied by $\min(|t\mc{K}_{\text{rem}}|,2)\in [0,2]$. In the non-failure cases, $|t\mc{K}_{\text{rem}}|\lesssim_{\lambda,\eps} |t|/n^{1/2-\eps}$.

The final bound easily implies the result.
\end{proof}
We will see later that for $|t|\le n^{\eps}$, \cref{lem:low-t} immediately establishes estimates of the desired quality.

\subsection{Mean and standard deviation considerations}\label{sub:standard-deviation-considerations} 
This is the promised discussion of the difference between $\sigma$ and $\sigma_W$, which is the true standard deviation in the model, and the difference between $0$ and $\mb{E}[\sum_{3\le k\le \ell}W_k|G(n,m)]$. 
\begin{lemma}\label{lem:std-close}
Let $W$ be an $(\ell,\lambda)$-statistic with $p = m/\binom{n}{2}\in (\lambda,1-\lambda)$ and define $\sigma^2 = \on{Var}[\sum_{k=3}^{\ell}W_k|G(n,p)]$ and $\sigma_W^2 = \on{Var}[W|G(n,m)]$. Then
\[\sigma = \sigma_W(1+O_{\lambda,\eps}(n^{\eps-1/2})).\]
Furthermore, if $\mathcal{K}$ is the normalized version of $W$ then
\[|\mb{E}[\mathcal{K}|G(n,m)]|\lesssim_{\lambda,p} n^{-1/2}.\]
\end{lemma}
\begin{proof}
We will first show that $\sigma = \sigma_W(1+O_{\lambda,\eps}(n^{\eps-1/2}))$ using the coupling between $G(n,m)$ and $G(n,p)$ given in the previous subsection. First note
\[\on{Var}[W|G(n,m)] = \on{Var}\bigg[\sum_{3\le k\le \ell} W_k\bigg|G(n,m)\bigg]\]
as fixing the number of edges fixes $W_2$. Now, in $G(n,p)$, by a trivial calculation we have
\[\on{Var}\bigg[\sum_{3\le k\le \ell} W_k\bigg|G(n,p)\bigg] = (1+\Theta_\lambda(n^{-1}))\on{Var}[W_3|G(n,p)].\]
Therefore, to prove our claim it suffices to show that
\[\on{Var}\bigg[\sum_{3\le k\le \ell} W_k\bigg|G(n,p)\bigg] = (1+O_{\lambda,\eps}(n^{\eps-1/2}))\on{Var}\bigg[\sum_{3\le k\le \ell} W_k\bigg|G(n,m)\bigg].\]
To prove this consider the coupling between $G(n,m)$ and $G(n,p)$ in the previous subsection; that is, sample $G(n,p)$ and then adjust the number of edges to be exactly $m$. Now define
\[Y_k = W_k' - W_k\]
where $W_k'$ has distribution corresponding to that in $G(n,p)$ and $W_k$ the appropriate distribution in $G(n,m)$ for $3\le k\le \ell$. The key claim in the previous subsection was essentially that 
\[\mb{P}[|Y_3|\ge n^{\eps-1/2}\sigma]\lesssim\exp(-\Omega_{\lambda,\eps}((\log n)^2))\]
and this coupling almost immediately gives the desired result. In particular,
\begin{align*}
\on{Var}\bigg[\sum_{3\le k\le\ell}{W_k'}\bigg]&-\on{Var}\bigg[\sum_{3\le k\le\ell}{W_k}\bigg]\\
&= \mb{E}\bigg[\bigg(\sum_{3\le k\le\ell}W_k'\bigg)^2 - \bigg(\sum_{3\le k\le\ell}W_k\bigg)^2\bigg] + 
\mb{E}\bigg[\sum_{3\le k\le\ell}W_k\bigg]^2-\mb{E}\bigg[\sum_{3\le k\le\ell}W_k'\bigg]^2\\
&= \mb{E}\bigg[\bigg(\sum_{3\le k\le\ell}Y_k'\bigg)\bigg(\sum_{3\le k\le\ell}W_k + W_k'\bigg)\bigg] - 
\mb{E}\bigg[\sum_{3\le k\le\ell}Y_k'\bigg]\mb{E}\bigg[\sum_{3\le k\le\ell}W_k + W_k'\bigg].\\
\end{align*}
By using that $W_i$ are mean zero in $G(n,p)$ and hypercontractivity (\cref{thm:concentration-hypercontractivity}) with sub-sampling it follows that 
\[\mb{P}\bigg[\bigg|\sum_{3\le k\le\ell}W_k'\bigg|\ge n^{\eps}\sigma\bigg]\lesssim\exp(-\Omega_{\lambda,\eps}((\log n)^2)),\qquad\mb{P}\bigg[\bigg|\sum_{3\le k\le\ell}W_k\bigg|\ge n^{\eps}\sigma\bigg]\lesssim\exp(-\Omega_{\lambda,\eps}((\log n)^2)).\]
Furthermore since $|\sum_{4\le k \le\ell} Y_k|\le |\sum_{4\le k\le \ell}W_k|+|\sum_{4\le k\le \ell}W_k'|$
it follows, again using hypercontractivity (\cref{thm:concentration-hypercontractivity}) with sub-sampling, that 
\[\mb{P}\bigg[\bigg|\sum_{4\le k\le\ell}Y_k\bigg|\ge n^{\eps-1/2}\sigma_W\bigg]\lesssim\exp(-\Omega_{\lambda,\eps}((\log n)^2))\]
as standard deviation of $\sum_{4\le k\le \ell}W_k$ is $\Theta_\lambda(n^{-1/2})$ smaller than $\sum_{3\le k\le \ell}W_k$. Finally, using that $W_k, W_k'$ are polynomially bounded random variables, substituting in the above analysis gives that 
\[\bigg|\mb{E}\bigg[\bigg(\sum_{3\le k\le\ell}Y_k\bigg)\bigg(\sum_{3\le k\le\ell}W_k + W_k'\bigg)\bigg] - 
\mb{E}\bigg[\sum_{3\le k\le\ell}Y_k\bigg]\mb{E}\bigg[\sum_{3\le k\le\ell}W_k + W_k'\bigg]\bigg|\lesssim_{\lambda,\eps}n^{2\eps-1/2}\sigma^2,\]
and rearranging and taking square roots the desired claim that $\sigma = \sigma_W(1+O_{\lambda,\eps}(n^{\eps-1/2}))$ follows.

We now derive that $|\mb{E}[\sum_{3\le k\le \ell}W_k|G(n,m)]|\lesssim_\lambda n^{-1/2}\sigma_W$ which is equivalent to the second estimate we wish to derive. To see this we simply use linearity of expectation. All we need is that any multi-linear degree $k$ or less monomial in the $\chi_{e}$ has expectation bounded in absolute value by $O_\lambda(n^{-2})$. To see this we compute
\[\mb{E}\big[\chi_{e_j}|\chi_{e_1},\ldots,\chi_{e_{j-1}}\big]\in\bigg[(m-j+1)/\bigg(\binom{n}{2}-j+1\bigg)-p, m/\bigg(\binom{n}{2}-j+1\bigg)-p\bigg]/\sqrt{p(1-p)},\]
which is $O_\lambda(n^{-2})$. Using this estimate directly it follows that
\[\mb{E}\bigg[\sum_{3\le k\le\ell} W_k\bigg]\lesssim_{\lambda,p}\sum_{3\le k\le \ell}n^{\ell-k}(n^{k})O_{\ell,p}(1/n^2)\lesssim_\lambda n^{\ell-2}\lesssim_\lambda n^{-1/2}\sigma_W.\qedhere\]
\end{proof}

\subsection{Bounds for $n^\eps\le|t|\le\sigma n^{-\eps}$}\label{sub:bounds-intermediate}
This subsection is by far the most elaborate in the paper due to various technical computations. At first reading the reader is recommended to take various probability and concentration claims at face value and not delve deeply into the calculations. It may also be useful to think of the $G(n,p)$ case as a model for calculations.

For this section consider the following decoupling. Choose some $1\le k\le\ell-2$. Partition the vertex set into $U_1,\ldots,U_k$ of size $n^\beta$, for some $\beta\in (0,1)$, and let the remaining vertices form $U_0$. We now separate the edge set into $k+1$ classes $B_0,\ldots,B_k$, where an edge between a vertex of $U_i$ and $U_j$ is put in $B_{\max(i,j)}$. We will require that $\beta$ is bounded away from $0$ and $1$ by a constant depending only on $\mc{H}$. Therefore $B_0$ has $\Theta(n^2)$ edges and each $B_i$ for $i\ge 1$ has $\Theta(n^{1+\beta})$ edges. This decoupling is closely related to that in \cite[Sections~9~and~10]{B2}.

Now sample $(Z_i)_{0\le i\le k}$, the number of edges chosen in each $B_i$, as if it is coming from $G(n,m)$, and then sample $X$, the actual vector of edges of $B_0$ (conditional on the $Z_i$). Then sample two independent copies $Y_i^0,Y_i^1$ of the vector of edges in $B_i$ for $1\le i\le k$, conditional on the previous information. Equivalently, we sampled from $G(n,m)$ and then resampled the edges in $(B_i)_{1\le i\le k}$ but preserved the number of edges in each $B_i$.

Define a \emph{suitable} outcome to be if $|Z_i-p|B_i||\le\sqrt{|B_i|}\log{|B_i|}$, say. The key point is that there is an overwhelming probability that all $Z_i$ are suitable by Azuma--Hoeffding and union bounding over a fixed number of events $k+1\le\ell$. Indeed, the probability of failure is $\exp(-\Omega_\lambda((\log n)^2))$.

If we sample the edges of $B_i$ with probability $p$ independently (sampling $i\ge 1$ twice) then we attain any particular suitable vector of edge counts over the $B_i$ with probability at least $\exp(-\Omega_\lambda((\log n)^2))$. Therefore, if in this independent model an event has probability at most $\exp(-\Omega_\lambda((\log n)^3))$, then even in the $G(n,m)$ model within suitable outcomes it occurs with this probability, perhaps weakening the constants in the exponent. (This is a version of the transference trick we used in the small $|t|$ regime as well.) Then we must add back in the unsuitable cases, which account for a probability of at most $\exp(-\Omega_\lambda((\log n)^2))$ by the above application of Azuma--Hoeffding.

Now define coefficients $\delta$, which are functions of $\mbf{Y}$, via 
\[\alpha(W)(X,\mbf{Y}) = \delta_\emptyset+\sum_{e\in B_0} \delta_{e}\chi_e + \sum_{\substack{S\subseteq B_0\\|S|\ge 2}}\delta_S\chi_S.\]
We now proceed to prove an absurd number of bounds on these coefficients with extremely high probability in the $G(n,p)$ model, such that the above argument applies to transfer the high probability to the $G(n,m)$ model. Note that the $\delta$ are polynomials in $\chi_e^b$ for $e\notin B_0$ and $b\in\{0,1\}$. In fact, we have
\[\delta_S = \sum_{H'\simeq H\in\mc{H}}\Delta_Hn^{\ell-v(H)}\prod_{j=1}^k\bigg(\prod_{\substack{e\in E(H')\setminus{S}\\e\in B_j}}\chi_{e}^1-\prod_{\substack{e\in E(H')\setminus{S}\\e\in B_j}}\chi_{e}^0\bigg),\]
where the sum is over subgraphs $H'$ of $\binom{[n]}{2}$ isomorphic to a graph in $\mc{H}$ such that $H'$ contains all $e\in S$, no other edges of $B_0$, and at least $1$ edge in each $B_i$ for $i\ge 1$. Therefore there is at least one vertex in each $U_i$ with $i\ge 1$, and all the vertices of $S$ are included. We will find it convenient to extract a ``main term'' from $\delta_e$, $e\in B_0$, namely
\[\delta_e' = \sum_{H'\simeq H\in\mc{H}_{k+2}}\Delta_Hn^{\ell-k-2}\bigg(\prod_{\substack{e\in E(H')\setminus{S}\\e\in B_j}}\chi_{e}^1-\prod_{\substack{e\in E(H')\setminus{S}\\e\in B_j}}\chi_{e}^0\bigg).\]
To be clear, the sum is over subgraphs $H'$ of $\binom{[n]}{2}$ isomorphic to a graph in $\mc{H}_{k+2}$ such that $H'$ contains $e$, no other edges of $B_0$, and at least one edge in each $B_i$ for $i\ge 1$. Thus it has $k+2$ vertices, which by the above considerations is the smallest number of vertices $H'$ could have. We can easily show by induction on $i\ge 1$ that every vertex is connected within $H'$ to the edge $e$, hence the $H'$ considered must be connected.  And by hypothesis, $\mc{H'}\cap\mc{H}_{k+2}$ has a connected graph, so $\delta_e'$ should be nontrivial. Let $r_e = \delta_e - \delta_e'$ be the remainder. We will now prove the following set of bounds on the sizes of these coefficients.
\begin{lemma}\label{lem:various-bounds}
Let $X$, $\mbf{Y}$, $\delta_S$, $\delta_e$, $\delta_e'$, and $r_e$ be as above. Let $C$ be a suitably large constant. Then have the following concentration bounds (in the $G(n,p)$ model).
\begin{enumerate}
    \item We have that
    \[\mb{P}[\sup_{e\in B_0}|\delta_e|\ge n^{\ell-k-2+k\beta/2}(\log n)^C]\le \exp(-\Omega_{\lambda}((\log n)^{3})).\]
    \item  We have that
    \[\mb{P}[\sup_{e\in B_0}|r_e|\ge n^{\ell-k-5/2+k\beta/2}(\log n)^C]\le \exp(-\Omega_{\lambda}((\log n)^{3})).\]
    \item We have that
    \[\mb{E}\bigg[\sum_{e\in B_0}\delta_e^2\bigg] = \Theta_\lambda(n^{2(\ell-k-1)+k\beta})\]
    and 
    \[\on{Var}\bigg[\sum_{e\in B_0}\delta_e^2\bigg] = O_{\lambda}(n^{4(\ell-k-1)+(2k-1)\beta)}).\]
    \item We have that 
    \[\mb{P}\bigg[\bigg|\sum_{e\in B_0}\delta_e\bigg|\le n^{(\ell-k-1/2)+k\beta/2}(\log n)^{C}\bigg]\le \exp(-\Omega_{\lambda}((\log n)^{3})).\]
    \item We have that 
    \[\mb{P}\bigg[\bigg|\sum_{\substack{S\subseteq B_0\\|S|\ge 2}}\delta_S^2\bigg|\le n^{2(\ell-k)-3+k\beta}(\log n)^{2C}\bigg]\le\exp(-\Omega_{\lambda}((\log n)^{3})).\]
\end{enumerate}
\end{lemma}
We now prove each of the lemma in order with each item corresponding to a separate subsection.
\subsubsection{Proof of \cref{lem:various-bounds} (1)}\label{subsub:linear-term-Linfty}
We wish to show that for all $e\in B_0$ we have that \[|\delta_e|\le n^{\ell-k-2+(k\beta/2)}(\log n)^C\]
with probability $1-\exp(-\Omega_\lambda((\log n)^3))$ for suitable $C$. To see this note $\delta_e$ is a polynomial of bounded degree and sum of squares of coefficients $O_\lambda(n^{2(\ell-k-2) + k\beta})$, and apply hypercontractivity (\cref{thm:concentration-hypercontractivity}). Here $C$ must be chosen large enough in terms of the degree of the polynomial, which is bounded by $\ell$, so can be taken to depend only on $\mc{H}$.
    
The sum of squares estimate is derived as follows. The contributing terms to $\delta_e$ are subgraphs $H'$ as delineated above, with $S = \{e\}$. Say it has $v$ vertices with $w$ outside of $U_0$. Note that $v\ge w+2$ as we have at least $2$ vertices in $U_0$. Then the coefficient is $O_\lambda(n^{\ell-v})$ and there are $n^{\beta w+(v-w-2)}$ choices for the location of the remaining vertices, since two are fixed by $e$. Hence the contribution is $O_\lambda(n^{2(\ell-v)+\beta w+(v-w-2)})$. As $w$ increases this decreases, so the major contribution is from $w=k$, the minimum, and as $v$ increases the resulting expression decreases, so the major contribution is from $v=k+2$, the minimum, yielding the desired bound. It is worth noting for later that the main contributors are those with $v=k+2$ and $w=k$ only, and the next highest term is from $v=k+3$ and $w=k$, which is $n^{-1}$ times smaller.

\subsubsection{Proof of \cref{lem:various-bounds} (2)}
We next show that for all $e\in B_0$, $|r_e|\le n^{\ell-k-5/2+(k\beta/2)}(\log n)^C$ with high probability (of the same quality as before). Indeed, the only point is that $\delta_e'$ contains all the main contributors discussed above, and thus the sum of squares of coefficients in $r_e$ is $O_\lambda(n^{2(\ell-k-2)+k\beta-1})$. Hypercontractivity (\cref{thm:concentration-hypercontractivity}) finishes.

\subsubsection{Proof of \cref{lem:various-bounds} (3)}
We now prove that $\sum_{e\in B_0} \delta_e^2$ concentrates on a value of size $\Theta_\lambda(n^{2(\ell-k-1)+k\beta})$. First we compute the expectation. Note that $\mb E[\chi_{e'}^1-\chi_{e'}^0] = 0$, hence the same is true of the products making up $\delta_e$ by independence. Further, we see that the product of two terms coming from $H_1',H_2'$ in the expansion of $\delta_e^2$ has zero expectation unless they have the same edge set, in which case it is constant. Summing over $H_1'=H_2'$ with $v$ vertices and $w$ outside $U_0$, we obtain $O_\lambda(n^{2(\ell-v)}\cdot n^{\beta w+(v-w-2)})$. This is maximized when $w=k$ and $v=k+2$. Furthermore, if we take a connected graph $H\in\mc{H'}\cap\mc{H}_{k+2}$ and then look at its embeddings containing $e$ with a vertex in each $U_i$ for $i\ge 1$, we find that this contributes $\Omega_\lambda(n^{2(\ell-k-2)+\beta k})$ in the above. So the expectation we obtain is $\Theta_\lambda(n^{2(\ell-k-2)+\beta k})$. Summing over $\Theta(n^2)$ edges in $B_0$, we obtain an expectation of the correct size.

Next, writing $\delta_e^2 = (\delta_e')^2 + 2\delta_er_e-r_e^2$, and using the $L^\infty$ bounds from \cref{lem:various-bounds} (1) and (2) above, we see with high probability that $\sum_{e\in B_0}\delta_e^2$ and $\sum_{e\in B_0}(\delta_e')^2$ differ by $O_\lambda(n^2\cdot n^{2(\ell-k)-9/2+k\beta}(\log n)^{2C})$. This is smaller in magnitude than the expectation, so with high probability this deviation is small.

Now it remains to show the standard deviation of $\sum_{e\in B_0}(\delta_e')^2$ is smaller in magnitude by some power of $n$ compared to the expectation. Then hypercontractivity (\cref{thm:concentration-hypercontractivity}) immediately demonstrates the desired concentration. Note that the variance is
\[\sum_{e_1,e_2\in B_0}(\mb E[\delta_{e_1}'^2\delta_{e_2}'^2]-\mb E[\delta_{e_1}'^2]\mb E[\delta_{e_2}'^2]).\]
First, if $e_1,e_2$ share a vertex, there are $O(n^3)$ choices for them. Using the $L^\infty$ bounds on $\delta_e$ (and $r_e$, hence $\delta_e'$) we see that the contribution to the sum above is $O_\lambda(n^3\cdot n^{4(\ell-k-2)+2k\beta}(\log n)^{4C})$. This bound is acceptable, by a factor of approximately $n^{-1/2}$ in the standard deviation.

Now consider the $O(n^4)$ cases where $e_1,e_2$ do not share a vertex. We write out $(\delta_{e_j}')^2$ as a sum over $H_{j,1},H_{j,2}$:
\begin{align*}
(\delta_{e_j}')^2 = \sum_{\substack{H_{j,1}'\simeq H_{j,1}\in\mc{H}\\H_{j2}'\simeq H_{j,2}\in\mc{H}}}\Delta_{H_{j,1}}\Delta_{H_{j,2}}\prod_{t=1}^k&\bigg(\prod_{\substack{e\in E(H_{j,1})\setminus{S}\\e\in B_t}}\chi_{e}^1-\prod_{\substack{e\in E(H_{j,1})\setminus{S}\\e\in B_t}}\chi_{e}^0\bigg)\\
&\bigg(\prod_{\substack{e\in E(H_{j,2})\setminus{S}\\e\in B_t}}\chi_{e}^1-\prod_{\substack{e\in E(H_{j,2})\setminus{S}\\e\in B_t}}\chi_{e}^0\bigg).
\end{align*}
Therefore we can write the above covariance $\mb E[\delta_{e_1}'^2\delta_{e_2}'^2]-\mb E[\delta_{e_1}'^2]\mb E[\delta_{e_2}'^2]$ as a further sum of covariances, with terms indexed by a choice of $H_{j,b}'$ for $j,b\in\{1,2\}$.

Since we are dealing with $\delta'$ these graphs are connected, with $2$ vertices in $B_0$ forming the prescribed edge and $1$ vertex in each $U_i$ with $i\ge 1$. Consider the union of all these graphs $H_{j,b}'$, $j,b\in\{1,2\}$ (within $K_n$). If any of its edges is only covered once, then we easily see the corresponding covariance will be zero (recall we are currently in the $G(n,p)$ model).

Suppose the union graph has at least three vertices in some $U_i$ with $i\ge 1$. Then one of the vertices is hit by a unique $H_{j,b}'$, which implies some edge is only hit by one. Thus these terms are zero. Therefore the remaining union graphs have at most $2$ vertices in each $U_i$ for $i\ge 1$. Now suppose that for some $U_i$ there is only $1$ vertex. Then the number of configurations that could give rise to this situation is $O_\lambda(n^4\cdot n^{(2k-1)\beta})$, with coefficient of size $O_\lambda(n^{4(\ell-k-2)})$. This gives an acceptable bound as well, by a factor of $n^{-\beta/2}$ in the standard deviation.

Now consider the case where there are exactly two vertices in each $U_i$ with $i\ge 1$. We claim that the remaining terms are zero. It can be nonzero only if every edge of $H_{1,1}',H_{1,2}',H_{2,1}',H_{2,2}'$ is covered more than once in the union of these graphs. However, our graphs $H_{j,b}'$ are connected with $1$ vertex in each $U_i$ for $i\ge 1$. We easily prove by induction on $U_i$ for $i\ge 1$ that to satisfy the edge covering condition, the graphs $H_{j,1}',H_{j,2}'$ have the same vertex set. (Carrying this out requires $e_1,e_2$ to have disjoint vertices.) Now, this implies the vertex sets of $H_{1,1}',H_{1,2}'$ versus $H_{2,1}',H_{2,2}'$ are disjoint in any remaining term. Therefore the edge sets are disjoint so the corresponding variables are independent, leading to a zero term once more.

Overall, we obtain a bound on the variance of quality $O_\lambda(n^{4(\ell-k-1)+(2k-1)\beta})$, so the standard deviation is $O_\lambda(n^{2(\ell-k-1)+k\beta-(\beta/2)})$, which is the desired bound.

\subsubsection{Proof of \cref{lem:various-bounds} (4)}
Next we show \[\bigg|\sum_{e\in B_0} \delta_e \bigg|\lesssim_\lambda n^{\ell-k-1/2+(k\beta/2)}(\log n)^C\] with probability $1-\exp(-\Omega_\lambda((\log n)^3))$. To bound the sum of squares of coefficients of $\sum_{e\in B_0}\delta_e$, note that every term (which is a product over $E(H')\setminus{e}$) is in at most $n^2$ polynomials $\delta_{e'}$ trivially. Therefore, after combining terms in $\sum_{e\in B_0}\delta_e$, by Cauchy, the new sum of squares of coefficients is at most $n^2$ times what we get by not combining, which is in turn $n^2$ times what we obtained in \cref{subsub:linear-term-Linfty}. This gives $O_\lambda(n^4\cdot n^{2(\ell-k-2)+k\beta})$, which is not good enough.
    
But in fact, for terms that contribute the most, namely the $H'$ with $v=k+2$ vertices and $w=k$ of them outside $B_0$, we see that $E(H')\setminus{e}$ can be completed to a valid contributor to some $\delta_{e'}$ only if $e'$ is incident to one of the two vertices in $V(H')\cap B_0$, which yields $2n$ possible polynomials a given term is in. Therefore we obtain $O_\lambda(n^3\cdot n^{2(\ell-k-2)+k\beta}+n^4\cdot n^{2(\ell-k-2)+k\beta-1})$, which yields the result directly upon using hypercontractivity (\cref{thm:concentration-hypercontractivity}).

\subsubsection{Proof of \cref{lem:various-bounds} (5)}
Finally, we prove
\[\sum_{S\subseteq B_0 \atop {|S|\ge 2}}\delta_{S}^2 \lesssim_\lambda n^{2(\ell-k)-3+k\beta}(\log n)^{2C}\]
with probability $1-\exp(-\Omega_\lambda((\log n)^3))$. We first consider each term $\delta_S$ individually. As above, it is a polynomial of bounded degree. The contributing terms are subgraphs $H'$ of $\binom{[n]}{2}$ isomorphic to some graph in $\mc{H}$ such that $H'$ contains the edges of $S$, no other edges of $B_0$, and at least $1$ edge in each $B_i$ for $i\ge 1$. Again suppose it has $v$ vertices, with $a$ of them spanned by the edges in $S$, and $w$ outside of $U_0$. There is at least one vertex in each $U_i$ for $i\ge 1$. Then we obtain an estimate of $O_\lambda(n^{\ell-v})$ for the coefficient over $n^{\beta w+(v-w-a)}$ different terms. Summing over $w\ge k$ we find that the sum of squares of coefficients therefore is $O_\lambda(n^{2(\ell-k-a)+\beta k)})$, similar to earlier. Again we have $|\delta_S|\le n^{\ell-k-a+(k\beta)/2}(\log n)^C$ with probability $1-\exp(-\Omega_\lambda((\log n)^3))$. There are $O(n^a)$ coefficients $S$ spanning $a$ vertices, and summing the squares of the above gives $O_\lambda(n^{2(\ell-k)-a+k\beta}(\log n)^{2C})$. Then summing over $3\le a\le f$ gives $O_\lambda(n^{2(\ell-k)-3+k\beta}(\log n)^{2C})$, as claimed. 
\begin{remark}
Note that in the case $k = \ell-2$, there in fact are no higher terms as such a term would require $H'$ to have at least $\ell+1$ vertices, but all $H\in\mc{H}$ satisfy $v(H)\le\ell$. This will be used later. 
\end{remark}
This (finally) concludes the proof of \cref{lem:various-bounds}. Now we use these bounds to conclude our argument in the intermediate range of $|t|$.

\subsubsection{Deriving characteristic function bounds}
Overall, we showed the above statements with high probability in the $G(n,p)$ model. As noted, this transfers to a statement with high probability in the $G(n,m)$ model via naive conditioning. Now we claim the following bound on the characteristic function.
\begin{lemma}\label{lem:medium-t}
Let $W$ be an $(\ell,\lambda)$-statistic with normalized version $\mc{K}$. Then for all $\eps > 0$ and $|t|\in[n^{\varepsilon},\sigma n^{-\varepsilon}]$ we have
\[|\mb{E}[e^{it\mc{K}}-e^{-t^2/2}]|\lesssim_{\lambda,\eps}n^{-\Omega_{\lambda,\eps}(\log\log n)}.\]
\end{lemma}
\begin{proof}
Note that $e^{-t^2/2}$ is sufficiently small in the necessary range to ignore. Let $X,\mbf{Y}$ be as at the beginning of \cref{sub:bounds-intermediate}. We use \cref{lem:vdC-dep}, obtaining
\[|\varphi_{\mc{K}}(t)|^{2^k}\le\mb E_{\mbf{Y}}\big|\mb E_Xe^{it\alpha(W)(X,\mbf{Y})/\sigma}\big|.\]
Now with probability $1-\exp(-\Omega_\lambda((\log n)^2))$ over the randomness of $\mbf{Y}$, we can assume all the claims regarding the $\delta$ coefficients in \cref{lem:various-bounds} are true. This leaves an error term of size $\exp(-\Omega_\lambda((\log n)^2))$ which we will be able to disregard. We can also impose the condition that $|\sum_{e\in B_0}\chi_e|\lesssim_\lambda B_0^{1/2}\log B_0$ since Azuma--Hoeffding upon revealing the elements of $B_0$ reveals that with high probability over the randomness of $B_0$, its number of edges is as expected (hence the same over the randomness of $\mbf{Y}$, since that fixes the sum over $B_0$). This induces an error term of size $\exp(-\Omega_\lambda((\log n)^2))$, again acceptable.

Now condition on one of the suitable choices of $\mbf{Y}$. Define
\[L = \sum_{e'\in B_0}\delta_{e'}\chi_{e'},\qquad U = \sum_{\substack{S\subseteq B_0\\|S|\ge 2}}\delta_S\chi_S,\]
which are random variables now depending only on $X$ (as $\mbf{Y}$ is fixed). We need to bound
\[\mb E_Xe^{it(L+U)/\sigma},\]
noting we can disregard $\delta_\emptyset$ as $|e^{it\delta_\emptyset/\sigma}|=1$. To bound this quantity, we will adapt the method in \cite[Theorem~3]{B2}. Fix some integer $d\ge 1$ that we will later send to infinity slowly. Now by Taylor's theorem with Lagrange error,
\[\bigg|e^{itU/\sigma} - \sum_{j=0}^d\frac{(itU/\sigma)^j}{j!}\bigg|\le 2\frac{|tU/\sigma|^{d+1}}{(d+1)!},\]
where the $2$ comes from splitting into real and imaginary parts. Note that the interior sum is really a polynomial in the $\chi$'s of degree bounded in terms of $d$, with coefficients at most some polynomial in $n$ of degree bounded by $d$, noting that $|t|\le\pi\sigma$. Therefore, we can write
\[\sum_{j=0}^d\frac{(itU/\sigma)^j}{j!} = \sum_{M\in\mc{M}}a_M\cdot M,\]
where $\mc{M}$ is a set of bounded degree monomials in the $\chi$ variables, and in particular $\sum_{M\in\mc{M}}|a_M| = O(n^D)$ for some $D$ depending on $d$. We see
\[\big|\mb E_Xe^{it(L+U)/\sigma}\big|\lesssim_d\sum_{M\in\mc{M}}\big|a_M\mb E_XMe^{itL/\sigma}\big|+\mb E_X|tU/\sigma|^{d+1}.\]
Now note that $tU/\sigma$ is a polynomial in the $\chi_{e'}$ for $e'\in B_0$ of bounded degree, and by our assumptions on $\mbf{Y}$ we control its sum of squares of coefficients. By hypercontractivity (\cref{thm:concentration-hypercontractivity}) we have
\[\mb P(|tU/\sigma|\ge n^{-\eps}) = \exp(-\Omega_\lambda(n^{\eps'}))\]
as long as $(t^2n^{2\eps}/\sigma^2)n^{2(\ell-k)-3+k\beta} < n^{-\eps}$ for some $\eps' > 0$, using our $L^2$ control of the higher terms \cref{lem:various-bounds} (5). Here hypercontractivity applies in the independent model, but again using our subsampling trick we can make it over the randomness of $X$, which constrains $\sum_{e\in B_0}\chi_e$.

Now the last term has good bounds, since $|tU/\sigma|\ge n^{-\eps}$ occurs with very low probability and $|tU/\sigma|$ is bounded above by some fixed degree polynomial in $n$ always. Indeed, this allows us to bound the last term by $O_d(n^{-\eps(d+1)})$. Alternatively, we could have used the moment form (\cref{thm:moment-hypercontractivity}) of hypercontractivity.

Now each term
\[\big|\mb E_XMe^{itL/\sigma}\big|\le\mb E_{e'\in\on{supp}(M)}|M|\big|\mb Ee^{itL/\sigma}\big|,\]
where the inner expectation is only over $e'\in B_0$ not contained in the monomial $M$. This is all but $O_d(1)$ of them. Now, the inner term is of a form with which we can apply \cref{lem:bernoulli-variance} (say, shifting the $\chi_{e'}$ back to $x_{e'}$). The precise value of the conditioned sum $\sum_{e'\in B_0}\chi_{e'}$ that we chose at the beginning will change exactly what replaces $p$ in the statement of the lemma, but it is say in $(\lambda/2,1-\lambda/2)$ for $n$ sufficiently large, hence bounded away from $\{0,1\}$. Therefore we obtain a bound of quality
\[O_\lambda(n^D)\cdot n^2\exp(-\Omega_\lambda((t^2n^2/\sigma^2)\on{Var}[\delta_{e'}])),\]
where $D$ is some constant depending on $d$. Now the point is we control $\on{Var}[\delta_{e'}]$ because of all the bounds from earlier. Indeed, the average of $\delta_{e'}^2$ concentrates on a value of size $\Theta_\lambda(n^{2(\ell-k-2)+k\beta})$ by \cref{lem:various-bounds} (3) whereas the average of $\delta_{e'}$ is of size $O_\lambda(n^{\ell-k-5/2+(k\beta/2)}(\log n)^C)$ by \cref{lem:various-bounds} (4). Since that is smaller in magnitude than the square root of above, we see that the variance $\on{Var}[\delta_{e'}]$ over all $e'\in B_0$ is of order $\Theta_\lambda(n^{2(\ell-k-2)+k\beta})$. The deletion of $O_d(1)$ terms from the $\delta_{e'}$ does not change the variance from this order of magnitude due to the $L^\infty$ bounds on $\delta_{e'}$ established by \cref{lem:various-bounds} (1) and (2). Therefore if $(t^2n^2/\sigma^2)n^{2(\ell-k-2)+k\beta} > n^\eps$ then this bound is acceptable. Additionally, to apply \cref{lem:bernoulli-variance} we need $|t/\sigma|\cdot|\delta_{e'}|\lesssim 1$, hence $|t/\sigma|n^{\ell-k-2+(k\beta/2)}(\log n)^C\lesssim 1$ suffices.

In conclusion, fixing $\eps > 0$, we have shown for any fixed $d$ that
\[\big|\mb E_Xe^{it(L+U)/\sigma}\big|\lesssim_{d,\lambda}n^{-\eps(d+1)}\]
as long as
\begin{equation}\label{eq:medium-t-conditions}
(t^2n^{2\eps}/\sigma^2)n^{2(\ell-k)-3+k\beta}<n^{-\eps},\quad (t^2n^2/\sigma^2)n^{2(\ell-k-2)+k\beta}>n^{\eps},\quad |t/\sigma|n^{\ell-k-2+(k\beta/2)}(\log n)^C\lesssim 1.
\end{equation}
Now we send $d\to\infty$ slowly, finding ultimately that
\[|\varphi_{\mc{K}}(t)|\lesssim n^{-\Omega_{\lambda,\eps}(d(n))}\]
for some slow growing $d = d(n)$ that is monotonic and limits to infinity. Note that $d(n) = \log\log n$ surely suffices.

Now it remains to calculate which range of $t$ is covered by this computation. The three bounds \cref{eq:medium-t-conditions} show that the range
\[n^{k-(k\beta+1-\eps)/2}\lesssim_\lambda |t|\lesssim_\lambda n^{k-(k\beta+3\eps)/2}\]
certainly is valid. We are allowed to range $1\le k\le\ell-2$ and $0<\beta<1$, although remember the warning that $\beta$ must be bounded away from $\{0,1\}$. Restricting $\ell\beta\in(\eps,1-\eps)$ still allows us to cover the range $t\in [n^{(k-1)/2+2\eps},n^{k-3\eps}]$ for each $k$, say. For $1\le k\le\ell-2$ these intervals overlap and hit the range $[n^{2\eps},n^{\ell-2-3\eps}]$. This almost hits the entire range we want.

However, notice that for $k=\ell-2$, the top value, there are no higher-order terms: see the remark following the proof of \cref{lem:various-bounds} (5). Hence there is no $U$ term and the above analysis is simplified. In particular, the first of the three conditions on $t$ in \cref{eq:medium-t-conditions} can be dropped. So for $k=\ell-2$ we actually cover the larger range governed by
\[(t^2n^2/\sigma^2)n^{2(\ell-k-2)+k\beta}>n^{\eps},\quad |t/\sigma|n^{\ell-k-2+(k\beta/2)}(\log n)^C\lesssim 1,\]
which allows us to cover $n^{k-(k\beta+1-\eps)/2}\lesssim_\lambda |t|\lesssim_\lambda n^{k-(k\beta-1+3\eps)/2}$ when $k=\ell-2$. This lets us cover $|t|\in [n^{(\ell-3)/2+2\eps},n^{\ell-3/2-3\eps}]$, which gets the remaining portion of the range.

Therefore we have hit every necessary $t$ with a bound of the desired quality, taking $\eps$ sufficiently small.
\end{proof}
\begin{remark}
Ensuring the ranges cover everything is where we use the hypothesis that $\mc{H'}\cap\mc{H}_k$ contains a connected graph for all $3\le k\le\ell$. More specifically, this hypothesis is used in the proof of the first part of \cref{lem:various-bounds} (3). Looking closely, we see that this can be weakened; the precise condition coming from our argument is that $\mc{H'}\cap\mc{H}_k$ contains a connected graph for all $k$ in some set $\{k_1,\ldots,k_a\}$, where $k_1=3$, $k_a=\ell$, and $k_{j+1}\le 2k_j-2$.
\end{remark}

\subsection{Further comments}\label{sub:graph-characteristic-comments}
It is worth noting that the above proofs also work for the $G(n,p)$ model with the obvious alterations. In fact, there is significantly less headache because we have independence. The major difference is that $\mc{H}$ should include $K_2$, and now $X_2$ controls the standard deviation. Although this approach would allow us to conclude the necessary theorems about $G(n,p)$, we will instead demonstrate those results via transference from the $G(n,m)$ model, which is a more powerful technique as we will see from our study of $k$-APs as well as anticoncentration counterexamples. 

We would also like to briefly address how the results up to this point are already sufficient to prove anticoncentration for statistics which satisfy the hypotheses of \cref{sub:graph-characteristic-setup}. In particular, using Esseen's concentration inequality \cite{E66} (see \cite{TV06} for a modern treatment) one can convert the bounds we have derived into anticoncentration estimates, losing only a factor of $n^{o(1)}$ versus the optimal bound in $G(n,m)$ (which can bootstrapped to $G(n,p)$). In fact, using a variant of the decouplings provided in \cref{sec:subgraph-counts} one can establish Fourier control up to $|t|\le c_{\mc{H},\lambda}\sigma$ and thus establish optimal anticoncentration for $G(n,m)$, losing only constant factors. Again this can be bootstrapped to $G(n,p)$ (with some care being required).

Note here that these remarks extend to graph statistics such as two times the $C_6$ count plus the number of copies $P_1+P_3$, the disjoint union of paths of length $1$ and $3$. This example may seem rather obscure but note that this statistic has a parity bias due to results of DeMarco and Redlich \cite{DR16} and therefore a local central limit theorem fails. However, using our Fourier analytic methods, since this statistic ultimately takes the form \eqref{eq:general-form}, we can still obtain optimal anticoncentration for random variables that do not satisfy a local central limit theorem, overcoming a theoretical obstruction suggested in \cite{FKS19}.

We remark that difficulties in establishing anticoncentration through Fourier analytic methods arise when statistics display longer scale fluctuations in the pointwise probabilities than simple parity biases. This will be due to ``degeneracy'' in which certain terms in the expansion \eqref{eq:general-form} of the graph statistic are missing or of a different magnitude than expected. However, even such situations are not insurmountable as we will later demonstrate in the case of $k$-term arithmetic progressions in the independent model. In future work we intend on elaborating on these remarks and developing a systematic theory of anticoncentration for graph counts.

\section{Local Limit Theorems for Subgraph Counts}\label{sec:subgraph-counts}
We now prove local limit theorems for subgraph counts and induced subgraph counts. Throughout this section we will work in the $G(n,m)$ model. Specifically, we prove a local limit theorem for subgraph counts of $H$ where $H$ is connected, and induced subgraph counts for any $H$. We specify that the number of edges must be such that $p=m/\binom{n}{2}$ is at least $\lambda$ away from $\{0,1\}$. In the induced case, there may also be up to around $v(H)^2$ ``critical'' values $p_{\text{crit}}$ that $p$ is $\lambda$ apart from. We will see later that, with these caveats, the results in \cref{sec:general-graph-characteristic-functions} can be applied directly. Therefore it remains to bound the necessary characteristic functions in the top range $\sigma n^{-\eps}\le|t|\le\pi\sigma$.

\subsection{Connected subgraph counts}\label{sub:connected-subgraph-counts}
As this is the simpler case, we do it first. Let $H$ be a connected graph on $\ell\ge 2$ vertices. Write
\[W = \sum_{\substack{H'\subseteq\binom{[n]}{2}\\H'\simeq H}}\prod_{e\in E(H')}x_e.\]
If we let $\chi_e = (x_e-p)/\sqrt{p(1-p)}$ as usual, then it will expand into a form such as \eqref{eq:general-form}. In particular,
\begin{equation}\label{eq:W-subgraph-count}
W=\sum_{\substack{H'\subseteq\binom{[n]}{2}\\H'\simeq H}}\prod_{e\in E(H')}(p+\sqrt{p(1-p)}\chi_e)=\sum_{S\subseteq H}p^{e(H)-e(S)}(\sqrt{p(1-p)})^{e(S)}c_{S,H}d_{S,H}\binom{n-v(S)}{\ell-v(S)}\gamma_S(\mbf{x}),
\end{equation}
where the sum is over subgraphs $S$ (lacking isolated vertices) of $H$ up to isomorphism. Here $c_{S,H}$ explicitly equals $(\ell-v(S))!\on{aut}S/\on{aut}H$ and $d_{S,H}$ equals the number of times $S$ appears as a subgraph of $H$, e.g. $d_{K_2,H} = e(H)$. For the empty graph, these values are taken to be $\ell!/\on{aut}H$ and $1$, respectively. This follows from an easy double-counting argument.

In particular, for $p\in(\lambda,1-\lambda)$, we see the coefficient of $\gamma_S(\mbf{x})$ is of size $\Theta_\lambda(n^{\ell-v(S)})$. Furthermore, $H$ has a connected subgraph with $k$ vertices for each $3\le k\le\ell$ since $H$ is connected (e.g. take subtrees of a spanning tree). Thus the results of \cref{sec:general-graph-characteristic-functions} apply. In particular, define $\sigma$ and $\mc{K}$ from $W$ in the same way as in \cref{sub:graph-characteristic-setup}. Then by \cref{lem:low-t}, for $|t|\le n^\eps$ we have
\begin{equation}\label{eq:subgraph-characteristic-low}
|\varphi_{\mc{K}}(t) - e^{-t^2/2}|\lesssim_\lambda\frac{|t|}{n^{\frac{1}{2}-\eps}}
\end{equation}
and by \cref{lem:medium-t}, for $n^\eps\le|t|\le\sigma n^{-\eps}$ we have
\begin{equation}\label{eq:subgraph-characteristic-medium}
|\varphi_{\mc{K}}(t) - e^{-t^2/2}|\lesssim n^{-\Omega_{\lambda,\eps}(\log \log n)}.
\end{equation}

Now we present a decoupling which handles the top range $\sigma n^{-\eps}\le|t|\le\pi\sigma$. 
\begin{lemma}\label{lem:top-connected}
Let $W$ be as in \cref{eq:W-subgraph-count}, and define $\sigma,\mc{K}$ as in \cref{sub:graph-characteristic-setup}. Then for $|t|\le\pi\sigma$,
\begin{equation}\label{eq:subgraph-characteristic-high}
|\varphi_{\mc{K}}(t)|\le \exp(-\Omega_\lambda(n))+\exp\left(-\Omega_\lambda\left(\frac{t^2n}{\sigma^2}\right)\right).
\end{equation}
\end{lemma}
\begin{proof}
Partition the vertex set $[n]$ into $\lfloor n/\ell\rfloor$ cliques of size $\ell$, with at most $\ell-1$ extra vertices that we will essentially ignore. Within each of the cliques take an isomorphic copy $H$ and label its edges $0$ to $e(H)-1$ arbitrarily. Let $\wt{B_0}$ be the union of all unlabeled edges along with those labeled $0$, and $B_i$ be all the edges labeled $i$ for $0\le i \le e(H)-1$. Thus $B_i$ for $1\le i\le e(H)-1$ and $\wt{B_0}$ partition the edges. Finally define $\wt{X} \in \{0,1\}^{\wt{B_0}}$ as the indicator vector of which edges are included in $G(n,m)$ and $Z_i$ for $1\le i\le e(H)-1$ as the number of edges in each set $B_i$ when sampling from $G(n,m)$. Then let $Y_i^0,Y_i^1$ for $1\le i\le e(H)-1$ be two independent samples of the edges within $B_i$ given $Z_i$. Also let $Y_0$ be the indicator of the edges of $\wt{X}$ in $B_0$ only, and $\wt{Y_0}$ be the indicator of the edges in $\wt{B_0}\setminus{B_0}$ only. Now note that 
\[\alpha(W)(\wt{X},\mbf{Y}) = \sum_{e\in\wt{B_0}} \delta_e(\mbf{Y}) x_e\]
with $\delta_e(Y)\in\{0,\pm 1\}$ for all $e\in B_0$ which are labeled. Indeed, for all $e\in B_0$ we have
\[\delta_e(\mbf{Y}) = \prod_{e'\in E(H_e)\setminus{e}}(x_{e'}^1-x_{e'}^0),\]
where $H_e$ is the unique isomorphic copy of $H$ containing $e$ that was embedded into one of the cliques.

Now we claim that with extremely high probability, the number of $e\in B_0$ such that $\delta_e(\mbf{Y}) = 1$ is greater than $\lambda^{2(e(H)-1)}n/(2\ell)$ and the number such that $\delta_e(\mbf{Y}) = 0$ satisfies the same. This is clear in the $G(n,p)$ model, as there are more than $n/(2\ell)$ edges $e\in B_0$, which have mutually independent coefficients which are easily seen to take on the desired values with positive probabilities. In particular, the probability of this event not occurring in the independent model is $\exp(-\Omega_\lambda(n))$.

In the $G(n,m)$ model, we repeatedly use Azuma--Hoeffding. First, it demonstrates that each $Z_i$ for $1\le i\le e(H)-1$ is approximately $pn/\ell$ with high probability, say, within the interval $[pn/(2\ell),(1+p)n/(2\ell)]$ with probability $1-\exp(-\Omega_\lambda(n))$. Conditional on a realization of the $Z_i$, the vectors $Y_i^0,Y_i^1$ are independent and uniform with a fixed sum. By Azuma--Hoeffding again, we can show that $x_{e'}^1 = 1$ and $x_{e'}^0 = 0$ for each $e'\in E(H_e)\setminus{e}$ in at least $\Omega_\lambda(n)$ of our cliques with probability $1 - \exp(-\Omega_\lambda(n))$. Similarly, we can show that $x_{e'}^1 = x_{e'}^0 = 1$ for some $e'\in E(H_3)\setminus{e}$ happens in at least $\Omega_\lambda(n)$ of our cliques with a similar probability.

We also control the number of edges among $Y_0$. Note that its distribution is the same as looking at the number of edges in a specific subset of $G(n,m)$. By Azuma--Hoeffding, with a process revealing edges within $B_0$ one at a time, we see with probability $1-\exp(-\Omega_\lambda(n))$ the fraction of edges chosen in this set is in $(p/2,(1+p)/2)$. Therefore, over the randomness of $Y_i^0$ and $\wt{Y_0}$, say, the number of edges in $Y_0$ is fixed to some value that is $\Theta_\lambda(n)$.

Now we are ready to apply \cref{lem:vdC-dep}. Let $B_0'$ be the set and $Y_0'$ be the vector within $Y_0$ which have coefficients not in $\{0,1\}$ (which is determined after $\mbf{Y}$ is chosen). We obtain
\begin{align*}
|\varphi_{\mc{K}}(t)|^{2^{e(H)-1}}&\le\mb E_{\mbf{Y}}\big|\mb E_{\wt{X}}e^{it\alpha(W)(\wt{X},\mbf{Y})/\sigma}\big|\le\mb E_{\mbf{Y},\wt{Y_0}}\big|\mb E_{Y_0}e^{it\alpha(W)(\wt{X},\mbf{Y})/\sigma}\big|\\
&\le\mb E_{\mbf{Y},\wt{Y_0},Y_0'}\big|\mb E_{Y_0\setminus{Y_0'}}e^{(it/\sigma)\sum_{e\in B_0\setminus{B_0'}}\gamma_e(\mbf{Y})x_e}\big|\le\exp(-\Omega_\lambda(n))+\exp\left(-\Omega_\lambda\left(\frac{t^2n}{\sigma^2}\right)\right),
\end{align*}
the last inequality using that the function is bounded by $1$ in the rare cases delineated above, and using \cref{lem:bernoulli-variance} in the remaining cases in which we know that the inner $x_e$ for $e\in Y_0\setminus{Y_0'}$ are drawn uniformly with a fixed sum depending on $\mbf{Y},\wt{Y_0}, Y_0'$. That sum is $\Theta_\lambda(n)$ in size, and additionally we use that a positive fraction (in terms of $\lambda, H$) of coefficients $\delta_e(\mbf{Y})$ are $1$ as well as $0$. Note that \cref{lem:bernoulli-variance} only applies if $(t/\sigma)\cdot(1-0)\le\pi$, which precisely hits the top of the range.
\end{proof}
Now we are ready to prove a local limit theorem for $G(n,m)$.
\begin{theorem}\label{thm:Gnm-subgraph-local}
Let $H$ be a connected graph, and fix $\lambda > 0$. Choose $n\ge 1$ and $m$ such that $p=m/\binom{n}{2}\in(\lambda,1-\lambda)$, and let $X_H$ be the number of times $H$ appears as a subgraph of the random graph $G(n,m)$. Let $\mu_H, \sigma_H$ be the mean and standard deviation of this random variable. Finally, define $Z_H = (X_H-\mu_H)/\sigma_H$. Then we have
\[|\sigma_H\mb P[Z_H=z]-\mc{N}(z)|\lesssim_{H, \lambda,\eps}n^{\eps-1/2}\]
for all $z\in(\mb{Z}-\mu_H)/\sigma_H$ and
\[\sum_{z\in(\mb{Z}-\mu_H)/\sigma_H}|\mb P[Z_H=z]-\mc{N}(z)/\sigma_H|\lesssim_{H, \lambda,\eps}n^{\eps-1/2}\]
for all $\eps > 0$.
\end{theorem}
\begin{proof}
Let $W, \mc{K}, \sigma$ be defined as earlier. Then, by \cref{lem:Linfty-distance}, we have for $z\in(\mb{Z}-W_0-W_2)/\sigma$ that
\[|\sigma\mb{P}[\mc{K}=z]-\mc{N}(z)|\le e^{-\pi^2\sigma^2/2}+\int_{-\pi\sigma}^{\pi\sigma}|\varphi_{\mc{K}}(t)-\varphi_{\mc{N}(0,1)}(t)|dt\lesssim_{H,\lambda,\eps}\frac{1}{n^{1/2-\eps}}\]
for all $\eps > 0$, combining \eqref{eq:subgraph-characteristic-low}, \eqref{eq:subgraph-characteristic-medium}, and \eqref{eq:subgraph-characteristic-high} for different integration ranges. This is a local central limit theorem, with one minor technical issue, which is that $\mc{K}$ has neither mean $0$ nor variance $1$. In particular,
\[Z_H = \frac{\sigma}{\sigma_H}(\mc{K}-\mb{E}\mc{K}).\]
But now we recall $\sigma/\sigma_H = 1+O_\eps(n^{\eps-1/2})$ and $\mb{E}\mc{K}\lesssim_\eps n^{\eps-1/2}$, which follow from \cref{lem:std-close}. Thus,
\[\mb{P}[Z_H=z] = \mb{P}[\mc{K}=z(\sigma_H/\sigma)+\mb{E}\mc{K}]\]
is near $(1/\sigma)\mc{N}(z(\sigma_H/\sigma)+\mb{E}\mc{K})$, which is near $(1/\sigma_H)\mc{N}(z)$, and the necessary bounds follow using that $\mc{N}(z)$ is Lipschitz. To deduce the second statement, we use what we have already proved along with \cref{lem:L1-distance}. We simply need to verify that $\mb{P}[|Z_H|>n^\eps]$ is small. Since the standard deviation of $Z_H$ is $1$, this follows immediately by hypercontractivity (\cref{thm:concentration-hypercontractivity}) along with our trick of transferring bounds to the slice.
\end{proof}

\subsection{Induced subgraph counts}\label{sub:induced-subgraph-counts}
Let $H$ be a graph, not necessarily connected, with $\ell\ge 3$ vertices. Let $q = -p/(1-p)$, which is negative and bounded away from zero as well as bounded in size in terms of $\lambda$. Write
\[W = \sum_{\substack{H'\subseteq\binom{[n]}{2}\\H'\simeq H}}\prod_{e\in E(H')}x_e\prod_{e\in\ol{E}(H')}(1-x_e)\]
where $\ol{E}(H')$ is the complement of $E(H')$ within the set of all possible edges $\binom{V(H')}{2}$. We expand
\begin{align}
\begin{split}\label{eq:W-induced-subgraph-count}
W &= \sum_{\substack{H'\subseteq\binom{[n]}{2}\\H'\simeq H}}\prod_{e\in E(H')}(p+\sqrt{p(1-p)}\chi_e)\prod_{e\in\ol{E}(H')}(1-p-\sqrt{p(1-p)}\chi_e)\\
&= p^{e(H)}(1-p)^{\ol{e}(H)}\sum_{S\subseteq K_\ell}p^{-e(S)/2}(1-p)^{e(S)/2}f_{S,H}(q)\binom{n-v(S)}{\ell-v(S)}\gamma_S(\mbf{x}),
\end{split}
\end{align}
where the sum is over subgraphs $S$ (lacking isolated vertices) of $K_\ell$ up to isomorphism. Here $f_{S,H}(q)$ is a polynomial in $q$ with positive coefficients, computed as the sum
\[f_{S,H}(q) = \sum_{\substack{S',H'\subseteq K_\ell\\S'\simeq S,H'\simeq H}}q^{|E(S')\setminus{E(H')}|}\]
which is taken over embeddings of $S, H$ into $K_\ell$. In particular, $f_{S,H}$ is a nonzero polynomial for each subgraph $S$ of $K_\ell$. For the empty graph, we obtain the constant polynomial $\ell!/\on{aut}H$.

Now, in order for $W$ to satisfy the hypotheses of \cref{sub:graph-characteristic-setup}, we need there to be a term $\gamma_S$ on $k$ vertices for each $3\le k\le\ell$ which has the correct order of magnitude. In order to ensure this, we merely need $q$ to be bounded away from a root of $f_{S,H}$. Simply let $S = K_{1,k}$. Then we see $f_{S,H}$ has degree at most $k$, hence has at most $k$ roots. Therefore as long as $q$ is bounded away from a set of at most $3+4+\cdots+\ell<\ell^2$ values, or equivalently $p$ is bounded away (say by $\lambda$) from at most $\ell^2$ values as well as $\{0,1\}$, the necessary hypotheses will be satisfied.

In particular, define $\sigma, \mc{K}$ from $W$ in the same way as in \cref{sub:graph-characteristic-setup}. Then for $|t|\le n^\eps$ we have by \cref{lem:low-t} that
\begin{equation}\label{eq:induced-characteristic-low}
|\varphi_{\mc{K}}(t) - e^{-t^2/2}|\lesssim_\lambda\frac{|t|}{n^{\frac{1}{2}-\eps}}
\end{equation}
and for $n^\eps\le|t|\le\sigma n^{-\eps}$ we have by \cref{lem:medium-t} that
\begin{equation}\label{eq:induced-characteristic-medium}
|\varphi_{\mc{K}}(t) - e^{-t^2/2}|\lesssim n^{-\Omega_{\lambda,\eps}(\log \log n)}.
\end{equation}

Now we present a decoupling which handles the top range $\sigma n^{-\eps}\le|t|\le\pi\sigma$. 
\begin{lemma}\label{lem:top-induced}
Let $W$ be as in \cref{eq:W-induced-subgraph-count}, and define $\sigma,\mc{K}$ as in \cref{sub:graph-characteristic-setup}. Suppose $p$ is bounded away by $\lambda$ from a set of $\ell^2$ values. Then for $|t|\le\pi\sigma$,
\begin{equation}\label{eq:induced-subgraph-characteristic-high}
|\varphi_{\mc{K}}(t)|\le \exp(-\Omega_\lambda(n))+\exp\left(-\Omega_\lambda\left(\frac{t^2n}{\sigma^2}\right)\right).
\end{equation}
\end{lemma}
\begin{proof}
For simplicity, without loss of generality we assume $H$ is connected. We can do this because replacing $H$ with its complement in $K_\ell$, replacing $p$ by $1-p$, and replacing $W$ by $\ell!\binom{n}{\ell}/\on{aut}H-W$ keeps the random variable the same, and either $H$ or its complement is connected.

After doing this, we use the same decoupling as in \cref{sub:connected-subgraph-counts}. Partition the vertex set $[n]$ into $\lfloor n/\ell\rfloor$ cliques of size $\ell$, with at most $\ell-1$ extra vertices. Label the vertices of $H$ by $[\ell]$ arbitrarily, and its edges from $0$ to $e(H)-1$. Within each clique take an isomorphic copy of this labeled $H$. Let $\wt{B_0}$ be the union of all unlabeled edges along with those labeled $0$, and $B_i$ be all the edges labeled $i$ for $0\le i \le e(H)-1$. Thus $B_i$ for $1\le i\le e(H)-1$ and $\wt{B_0}$ partition the edges. Finally define $\wt{X} \in \{0,1\}^{\wt{B_0}}$ as the indicator vector of which edges are included in $G(n,m)$ and $Z_i$ for $1\le i\le e(H)-1$ as the number of edges in each set $B_i$ when sampling from $G(n,m)$. Let $Y_i^0,Y_i^1$ for $1\le i\le e(H)-1$, similarly, be independent copies of the $G(n,m)$ draw given $Z_i$. Also let $Y_0$ be the indicator of the edges of $\wt{X}$ in $B_0$ only, and $\wt{Y_0}$ be the indicator of the edges in $\wt{B_0}\setminus{B_0}$ only. Though the decoupling is the same, the resulting decoupled function is more complex. We first show that we can write
\[\alpha(W)(\wt{X},\mbf{Y}) = \delta_\emptyset(\mbf{Y},\wt{Y_0})+\sum_{e\in B_0} \delta_e(\mbf{Y},\wt{Y_0})x_e\]
for polynomials $\delta_e$, $e\in B_0$. To prove this, we consider which terms in the definition of $W$ (in the $x$ basis) provide a term in $\alpha(W)$ dependent on $x_e$. For a term to not become zero, it must have an edge from each of $B_i$ for $i\ge 1$, as well as the edge $e\in B_0$. Consider a term corresponding to a copy $H'$ of $H$. We claim that $V(H') = V(H_e)$, where $H_e$ is the copy of $H$ within the clique that contains $e$. First, since the above shows we have an edge in $B_i$ for each $i\ge 1$, and since $H$ has no vertex isolated, we have that $V(H')$ has at least one vertex with each label from $[\ell]$. Since $v(H') = \ell$, this means it has each label exactly once. Now, if vertex labels $a$ and $b$ are connected by an edge labeled $c$ within our labeled version of $H$, then to ensure an edge from $B_c$ exists in $H'$, the unique elements of $V(H')$ labeled by $a$ and $b$ must be in the same one of the $\lfloor n/\ell\rfloor$ cliques. This fact, along with $e\in E(H')$ and the connectedness of $H$, immediately shows that $V(H') = V(H_e)$.

But now this means the only terms with a dependence on $x_e$ are the terms
\[W_e = \sum_{\substack{H'\subseteq\binom{V(H_e)}{2}\\H'\simeq H}}\prod_{e'\in E(H')}x_{e'}\prod_{e'\in\ol{E}(H')}(1-x_{e'}).\]
These terms do not contain any other $x_{e'}$ for $e'\in B_0$, hence we obtain only linear terms in $\alpha(W)$ when collecting in the variable set $\{x_e\}_{e\in B_0}$. That is, $\alpha(W)$ is of the claimed form. Not only that, but we now know how to explicitly compute each $\delta_e$.

In fact, we will merely compute $\delta_e(\mbf{Y},\mbf{0})$. In this case,
\[W_e = \prod_{e'\in E(H_e)}x_{e'}\]
since the remaining terms corresponding to $H'\neq H_e$ have a factor of $x_{e'}$ for $e'\in\wt{B_0}\setminus{B_0}$, which were set to $0$, and since the terms from $e'\in\ol{E}(H_e)$, equal to $1-x_{e'}$, merely become $1$. Therefore $\delta_e(\mbf{Y},\mbf{0})$ equals what it did in \cref{sub:connected-subgraph-counts}, namely
\[\delta_e(\mbf{Y},\mbf{0}) = \prod_{e'\in E(H_e)\setminus{e}}(x_{e'}^1-x_{e'}^0).\]
In fact, this formula still holds as long as just $x_{e'} = 0$ for all $e'\in\ol{E}(H_e)$ (rather than all of $\wt{Y_0}$).

Now we claim that with extremely high probability, the number of $e\in B_0$ such that $\delta_e(\mbf{Y}) = 1$ is greater than $\lambda^{\ell^2}n/(2\ell)$ and the number such that $\delta_e(\mbf{Y}) = 0$ satisfies the same. This is clear in the $G(n,p)$ model, as there are more than $n/(2\ell)$ edges $e\in B_0$, and so long as $x_{e'} = 0$ for all $e'\in\ol{E}(H_e)$ and
\[\prod_{e'\in E(H_e)\setminus{e}}(x_{e'}^1-x_{e'}^0) = 1\]
(or $0$, respectively) we have the desired event for $e$. Independence finishes: the probability of this event not occurring in the independent model is $\exp(-\Omega_\lambda(n))$.

In the $G(n,m)$ model, as in the proof of \cref{lem:top-connected} we repeatedly use Azuma--Hoeffding. First, it demonstrates that each $Z_i$ for $1\le i\le e(H)-1$ is approximately $pn/\ell$ with high probability, say, within the interval $[pn/(2\ell),(1+p)n/(2\ell)]$ with probability $1-\exp(-\Omega_\lambda(n))$. Conditional on a realization of the $Z_i$, the vectors $Y_i^0,Y_i^1$ are independent and uniform with a fixed sum. By Azuma--Hoeffding again, we can show that $x_{e'}^1 = 1$ and $x_{e'}^0 = 0$ for each $e'\in E(H_e)\setminus{e}$ as well as $x_{e'} = 0$ for all $e'\in\ol{E}(H_e)$, simultaneously, in at least $\Omega_\lambda(n)$ of our cliques with probability $1 - \exp(-\Omega_\lambda(n))$. Similarly, we can show that $x_{e'}^1 = x_{e'}^0 = 1$ for some $e'\in E(H_3)\setminus{e}$ happens in at least $\Omega_\lambda(n)$ of our cliques with a similar probability.

We also control the number of edges among $Y_0$. Note that its distribution is the same as looking at the number of edges in a specific subset of $G(n,m)$. By Azuma--Hoeffding, with a process revealing edges within $B_0$ one at a time, we see with probability $1-\exp(-\Omega_\lambda(n))$ the fraction of edges chosen in this set is in $(p/2,(1+p)/2)$. Therefore, over the randomness of $Y_i^0$ and $\wt{Y_0}$, say, the number of edges in $Y_0$ is fixed to some value that is $\Theta_\lambda(n)$.

Now we are ready to apply \cref{lem:vdC-dep}. Let $B_0'$ be the set and $Y_0'$ be the vector $(x_e)$ for $e\in B_0$ satisfying $\delta_e\not\in\{0,1\}$ (which depends on $\mbf{Y},\wt{Y_0}$). We obtain
\begin{align*}
|\varphi_{\mc{K}}(t)|^{2^{e(H)-1}}&\le\mb E_{\mbf{Y}}\big|\mb E_{\wt{X}}e^{it\alpha(W)(\wt{X},\mbf{Y})/\sigma}\big|\le\mb E_{\mbf{Y},\wt{Y_0}}\big|\mb E_{Y_0}e^{it\alpha(W)(\wt{X},\mbf{Y})/\sigma}\big|\\
&\le\mb E_{\mbf{Y},\wt{Y_0},Y_0'}\big|\mb E_{Y_0\setminus{Y_0'}}e^{(it/\sigma)\sum_{e\in B_0\setminus{B_0'}}\delta_e(\mbf{Y},\wt{Y_0})x_e}\big|\le\exp(-\Omega_\lambda(n))+\exp\left(-\Omega_\lambda\left(\frac{t^2n}{\sigma^2}\right)\right),
\end{align*}
the last inequality using that the function is bounded by $1$ in the rare cases delineated above, and using \cref{lem:bernoulli-variance} in the remaining cases in which we know that the inner $x_e$ for $e\in Y_0\setminus{Y_0'}$ are drawn uniformly with a fixed sum depending on $\mbf{Y},\wt{Y_0}, Y_0'$. That sum is $\Theta_\lambda(n)$ in size, and additionally we use that a positive fraction (in terms of $\lambda, H$) of coefficients $\delta_e(\mbf{Y})$ are $1$ as well as $0$. Note that \cref{lem:bernoulli-variance} only applies if $(t/\sigma)\cdot(1-0)\le\pi$, which precisely hits the top of the range.
\end{proof}

Now, in exactly the same way as for \cref{thm:Gnm-subgraph-local}, we deduce a local limit theorem for $G(n,m)$.
\begin{theorem}\label{thm:Gnm-induced-local}
Let $H$ be a graph, and fix $\lambda > 0$. There is a set $\mc{P}_{\text{crit}}$ of size at most $v(H)^2$ such that the following holds. Choose $n\ge 1$ and $m$ such that $p=m/\binom{n}{2}$ is $\lambda$-separated from $\{0,1\}\cup\mc{P}_{\text{crit}}$, and let $X_H$ be the number of times $H$ appears as an induced subgraph of the random graph $G(n,m)$. Let $\mu_H, \sigma_H$ be the mean and standard deviation of this random variable. Finally, define $Z_H = (X_H-\mu_H)/\sigma_H$. Then we have
\[|\sigma_H\mb P[Z_H=z]-\mc{N}(z)|\lesssim_{H, \lambda,\eps}n^{\eps-1/2}\]
for all $z\in(\mb{Z}-\mu_H)/\sigma_H$ and
\[\sum_{z\in(\mb{Z}-\mu_H)/\sigma_H}|\mb P[Z_H=z]-\mc{N}(z)/\sigma_H|\lesssim_{H, \lambda,\eps}n^{\eps-1/2}\]
for all $\eps > 0$.
\end{theorem}
\begin{remark}
In fact, as noted in \cref{sec:introduction}, we can reduce our bound on the number of critical values to $O(v(H))$. This follows from the remark at the end of \cref{sub:bounds-intermediate}, which demonstrates that we only need to ensure a dyadically separated set of $k$ have nontrivial coefficients. Of course, since there are many more graphs on $k$ vertices than merely $K_{1,k}$, and only one of them must be nonzero, it is likely that there are even fewer critical values than we can prove.
\end{remark}

\section{Independent Models}\label{sec:independent-models}
In this section we deduce a local limit theorem for subgraph counts of connected graphs in $G(n,p)$ from the corresponding result for $G(n,m)$. It is worth remarking that all the earlier calculations done to prove the $G(n,m)$ case can be done analogously and with much more simplicity (as variables are actually independent) to directly prove the $G(n,p)$ case. However, the method of transfer is still important as it will allow us to prove more general results such as a local limit theorem in the $k$-AP case. For the sake of not belaboring the issue we prove the reduction only for \cref{thm:Gnp-subgraph-local}, as the analysis for \cref{thm:Gnp-induced-local} is completely analogous.
\begin{proof}[Proof of \cref{thm:Gnp-subgraph-local}]
Let $X_H$ denote the number of copies of $H$. Let $\sigma_H^2 = \on{Var}[X_H|G(n,p)]$ and $\mu_H = \mb{E}[X_H|G(n,p)]$. We now recall from \cref{sub:connected-subgraph-counts} that
\[X_H=\sum_{\substack{H'\subseteq\binom{[n]}{2}\\H'\simeq H}}\prod_{e\in E(H')}(p+\sqrt{p(1-p)}\chi_e)=\sum_{S\subseteq H}p^{e(H)-e(S)}(\sqrt{p(1-p)})^{e(S)}c_{S,H}d_{S,H}\binom{n-v(S)}{\ell-v(S)}\gamma_S(\mbf{x}),\]
where the sum is over subgraphs $S$ (lacking isolated vertices) of $H$ up to isomorphism. Recall that as before, $c_{S,H}$ explicitly equals $(\ell-v(S))!\on{aut}S/\on{aut}H$ and $d_{S,H}$ equals the number of times $S$ appears as a subgraph of $H$, e.g. $d_{K_2,H} = e(H)$. Let 
\[X_{2} = p^{e(H)-1}(\sqrt{p(1-p)})c_{S,K_2}e(H)\binom{n-2}{\ell-2}\gamma_{K_2}(\mbf{x})\]
and 
\[X_{\text{rem}} = X_H - p^{e(H)}\binom{n}{\ell}\frac{\ell!}{\on{aut}H} - X_2.\]
First note by direct computation that if $p\in (\lambda,1-\lambda)$ then
\[\on{Var}[X_{\text{rem}}|G(n,p)] = (1+O_\lambda(\log n/n)) \on{Var}[X_{\text{rem}}|G(n,p')]\]
if $p' = (1+\Theta(\log n/n))p$. Given this and the deductions in \cref{lem:std-close} that 
\[\on{Var}[X_H|G(n,m)] = \on{Var}[X_{\text{rem}}|G(n,m)] = (1+O_{\lambda,\eps}(n^{\eps-1/2}))\on{Var}[X_{\text{rem}}|G(n,q)]\]
(here $q = m/\binom{n}{2}\in(\lambda,1-\lambda)$), we find for any $m,m'\in [p\binom{n}{2} - n\log n, p\binom{n}{2}+n\log n]$ that 
\begin{align*}
\on{Var}[X_H|G(n,m)] & = (1+O_{\lambda,\eps}(n^{\eps-1/2})) \on{Var}[X_H|G(n,m')]\\
& = (1+O_{\lambda,\eps}(n^{\eps-1/2})) \on{Var}[X_{\text{rem}}|G(n,p)].
\end{align*}
From now we denote $\sigma^2 = \on{Var}[X_{\text{rem}}|G(n,p)]$. We now explicitly use that the expectation of $X_H$ varies essentially linearly given the number of edges. In particular, for $m\in [p\binom{n}{2}-n\log n, p\binom{n}{2}+n\log n]$, note by linearity of expectation that
\begin{align*}
\mb{E}[X_H|G(n,m)] &= \binom{n}{\ell}\frac{\ell!}{\on{aut}{H}}\prod_{i=0}^{e(H)-1}\bigg(\frac{m-i}{\binom{n}{2}-i}\bigg)\\
&= \binom{n}{\ell}\frac{\ell!}{\on{aut}{H}}\bigg(p^{e(H)} + e(H)p^{E(H)-1}\bigg(\frac{m-p\binom{n}{2}}{\binom{n}{2}}\bigg)\bigg)(1+\Theta_\lambda((\log n)^2/n^2)).
\end{align*}
An essentially similar estimate was derived in \cref{lem:std-close} for more general graph statistics. Now, for the sake of clarity define 
\[f(m) = \binom{n}{\ell}\frac{\ell!}{\on{aut}{H}}\bigg(p^{e(H)} + e(H)p^{e(H)-1}\bigg(\frac{m-p\binom{n}{2}}{\binom{n}{2}}\bigg)\bigg).\]

Finally we are in a position to explicitly calculate the distribution of $X_H$ under $G(n,p)$. Let $\sigma_m^2 = \on{Var}[X_H|G(n,m)]$ and $\mu_m = \mb{E}[X_H|G(n,m)]$. Now note that
\begin{align*}
\mb{P}[X_H = x] &= \sum_{m\in \mb{Z}}\mb{P}[X_H = x|G(n,m)]\mb{P}\big[\sum x_e = m\big]\\
&= \sum_{\substack{m\in [p\binom{n}{2} - n\log n,\\ ~\quad p\binom{n}{2} + n\log n]}}\mb{P}[X_H = x|G(n,m)]\mb{P}\big[\sum x_e = m\big] + \exp(-\Omega_\lambda((\log n)^2))
\end{align*}
where we have used Chernoff to bound the probability that number of edges deviates too far from the mean. For the sake of clarity we will implicitly assume that $x$ is within $\sigma_H(\log n)^C$ of the mean; for $x$ outside this range and $C$ sufficiently large the probability of attaining $x$ is super-polynomially small by hypercontractivity (\cref{thm:concentration-hypercontractivity}) so the desired statement is trivial. This assumption will be used implicitly later on. Now let $\mc{M}_{x}$ denote the set of $m$ such that
\[|x-f(m)|\le\sigma(\log n)^C\]
and
\[m\in \bigg[p\binom{n}{2} - n\log n, p\binom{n}{2} + n\log n\bigg]\]
for a suitably large $C$. Now suppose that $m\in [p\binom{n}{2} - n\log n, p\binom{n}{2} + n\log n]\backslash\mc{M}_x$. Then
\begin{align*}
\mb{P}[X_H = x | G(n,m)] & \le \mb{P}\big[|X_{\text{rem}}| \ge \sigma(\log n)^C/2\big]/ \mb{P}\big[\sum x_e = m\big]\\
& \lesssim\exp(-\Omega_\lambda((\log n)^2)),
\end{align*}
using that $\mb{P}[\sum x_e = m]\gtrsim\exp(-O_\lambda((\log n)^2))$ and then choosing $C$ sufficiently large so that the bound coming from hypercontractivity (\cref{thm:concentration-hypercontractivity}) on the numerator is sufficiently strong. The key point is that since $f(m)$ is a linear function with slope $\Theta_\lambda(n^{\ell -2})$ we have $|\mc{M}_x| = \Theta_\lambda(n^{1/2}(\log n)^C)$. Thus we have that \begin{align*}
\mb{P}[X_H = x] &= \sum_{\substack{m\in [p\binom{n}{2} - n\log n,\\ ~\quad p\binom{n}{2} + n\log n]}}\mb{P}[X_H = x|G(n,m)]\mb{P}\big[\sum x_e = m\big] + \exp(-\Omega_\lambda((\log n)^2))\\
& = \sum_{m\in \mc{M}_x}\mb{P}[X_H = x|G(n,m)]\mb{P}\big[\sum x_e = m\big] + \exp(-\Omega_{\lambda}((\log n)^2)).
\end{align*}
Now using \cref{thm:Gnm-subgraph-local} and that $\sigma_{m}$ is approximately equal to $\sigma$, the last summation equals
\begin{align*}
\sum_{m\in \mc{M}_x}\bigg(\frac{1}{\sigma_m}\mathcal{N}\left(\frac{x-\mu_m}{\sigma_{m}}\right)&+O_\lambda\bigg(\frac{n^{\frac{\eps-1}{2}}}{\sigma}\bigg)\bigg)\mb{P}\big[\sum x_e = m\big]\\
& = \sum_{m\in \mc{M}_x}\frac{1}{\sigma_m}\mathcal{N}\left(\frac{x-\mu_m}{\sigma_{m}}\right)\mb{P}\big[\sum x_e = m\big] + O_\lambda\left(\frac{|\mc{M}_{x}|}{n^{\frac{3-\eps}{2}}\sigma}\right)
\end{align*}
where use that probability of having a given number of edges is $O_\lambda(1/n)$. Now note that $\sigma_H$ is order $n^{1/2}$ larger than $\sigma$. Therefore the error term can be seen to be $O_\lambda(n^{\eps-1/2}\sigma_H^{-1})$, which is the correct magnitude. Now $\sigma_m = (1+O_\lambda(n^{\eps-1/2}))\sigma$ and $\mu_m = f(m) + O_\lambda((\log n)^2n^{-1/2}\sigma)$ for all $m\in [p\binom{n}{2}-n\log n, p\binom{n}{2}+n\log n]$ by the remarks which began the section. It follows that
\begin{align*}
\sum_{m\in \mc{M}_x}\frac{1}{\sigma_m}\mathcal{N}\left(\frac{x-\mu_{m}}{\sigma_{m}}\right)&\mb{P}\big[\sum x_e = m\big] \\
&= \sum_{m\in \mc{M}_x}\frac{1}{\sigma}\mathcal{N}\left(\frac{x-f(m)}{\sigma}\right)\mb{P}\big[\sum x_e = m\big] + O_\lambda\bigg(\frac{n^{\eps-1/2}}{\sigma_H}\bigg).
\end{align*}
At this point the rest is elementary calculation. Let $m^\ast$ be the solution to $f(m^\ast) = x$ and note that $|m-m^\ast|\lesssim_\lambda n^{1/2}(\log n)^C$ since $f$ has slope $\Theta_\lambda(\sigma n^{-1/2})$. This is enough to conclude that $\mb{P}[\sum x_e = m]$ is essentially constant over $m\in\mc{M}_x$, close enough to replace the above with
\begin{align*}
\mb{P}\big[\sum x_e=\lfloor m^\ast\rfloor\big]&\sum_{m\in\mc{M}_x}\frac{1}{\sigma}\mc{N}\left(\frac{x-f(m)}{\sigma}\right) \\
&= \mb{P}\big[\sum x_e=\lfloor m^\ast\rfloor\big]\sum_{m\in\mb{Z}}\frac{1}{\sigma}\mc{N}\left(\frac{x-f(m)}{\sigma}\right)+O_\lambda(\exp(-(\log n)^2))
\end{align*}
without increasing the error term. Now since $\mc{N}$ is continuous, unimodal, and integrable, and since $f$ has slope
\[\eta = \binom{n}{\ell}\frac{\ell!}{\on{aut}H}\frac{e(H)p^{e(H)-1}}{\binom{n}{2}},\]
standard results on Riemann approximation show that this equals
\[\mb{P}\big[\sum x_e = \lfloor m^\ast\rfloor\big]\frac{1+O_\lambda(n^{-1/2})}{\eta}.\]
Again the error term is acceptable, and using Stirling's approximation shows that this is approximately
\[\frac{1}{\sqrt{p(1-p)\binom{n}{2}}}\frac{1}{\eta}\mc{N}\left(\frac{m^\ast-p\binom{n}{2}}{\sqrt{p(1-p)\binom{n}{2}}}\right)\]
Finally we note that by calculation that
\[\frac{m^{*}-p\binom{n}{2}}{\sqrt{p(1-p)\binom{n}{2}}} = \frac{x - \mb{E}[X_H|G(n,p)]}{\eta\cdot\sqrt{p(1-p)\binom{n}{2}}},\text{ and }\eta\cdot\sqrt{p(1-p)\binom{n}{2}} = (1+O_{H,p}(1/n))\sigma_H.\]
These two estimates, combined with the rest in this proof finally give that
\[\mb{P}[X_H = x] = \frac{1}{\sigma_H}\mc{N}\left(\frac{x-\mu_H}{\sigma_H}\right) + O_\lambda\bigg(\frac{n^{\eps-1/2}}{\sigma_H}\bigg).\]
To deduce the necessary $L^1$ bound, use \cref{lem:L1-distance} and hypercontractivity (\cref{thm:concentration-hypercontractivity}).
\end{proof}

\section{Counterexamples}\label{sec:counterexamples}
In this section we establish counterexamples to some anticoncentration conjectures of Fox, Kwan, and Sauermann \cite{FKS19}. The main technical result is that the following class of graph-related polynomials do not exhibit anticoncentration. 
\begin{theorem}\label{thm:counterexample}
Let $\chi_{e} = (x_{e}-p)/\sqrt{p(1-p)}$ where $x_e$ are independent Bernoulli random variables with expectation $p\in (0,1)$ for all $e\in\binom{[n]}{2}$. Suppose that
\[F(\mbf{x}) = \sum_{H:\, v(H)\le \ell} \binom{n-v(H)}{\ell-v(H)}\Phi_H\gamma_H(\mbf{x})\]
where $\Phi_H$ are constants independent of $n$ and satisfy
\begin{enumerate}[1.]
    \item $\Phi_H = 0$ for all connected graphs on $3$ and $4$ vertices
    \item $\Phi_H\neq 0$ for $H$ being an edge and $H$ being the disjoint union of two edges.
\end{enumerate}
Furthermore suppose that $F$ is integer valued. Then there exists a sequence $y_n$ such that
\[\mb{P}[F(\mbf{x}) = y_n]\gtrsim_{H, p} n^{3/2-\ell}.\]
\end{theorem}
\begin{remark}
Recall that $\gamma_H$ is only defined if $H$ has no isolated vertices. Also, note that the standard deviation of $F$ is of order $n^{\ell-1}$ since $\Phi_{K_2}\neq 0$ and therefore anticoncentration fails by order $n^{1/2}$. Furthermore such polynomials are $\mc{U}_3\cup\mc{U}_4^{c}$-proportional (and not $\mc{U}_2$-proportional) in the notation of Janson \cite{J2}. This result is ultimately derived from the results in \cite{J1} along with a conditioning argument.
\end{remark}
\begin{proof}
Note that the probability that the number of edges is $\lfloor p\binom{n}{2}\rfloor$ is $\Omega_p(n^{-1})$. Due to results of Janson \cite{J1}, statistics satisfying the hypothesis of this theorem converge to Gaussians of standard deviation $\Theta_p(n^{\ell-5/2})$ (see \cite[Theorem~III.8]{J1}) in $G(n,m)$. The reason for this is 
\[2\gamma_{K_2+K_2}(\mbf{x}) = (\sum_{e}\chi_{e})^2 - \sum_{e}\chi_{e}^2 - 2\sum_{i,j,k}\chi_{i,j}\chi_{i,k}\]
\[ = \gamma_{K_2}(\mbf{x})^2 - \sum_{e}\chi_{e}^2 - 2\gamma_{K_{1,2}}(\mbf{x}),\]
hence the number of edges determines the first two terms. (This identity also appears as \cite[Theorem~4]{J2}.) Thus, given the number of edges the standard deviation coming from the $K_2+K_2$ term is only $\Theta_p(n^{\ell-5/2})$ in $G(n,m)$ (for $m$ near $p\binom{n}{2}$) instead of $\Theta_p(n^{\ell-2})$ as would be ``typical''.

In particular, by the aforementioned theorem of Janson, we have a sequence $\alpha_{n}$ and $\beta_{n} = \Theta_{F,p}(n^{\ell-5/2})$ such that
\[\frac{F(\mbf{x})-\alpha_{n}}{\beta_{n}}\xrightarrow{d}\mc{N}(0,1)\]
if the edges are sampled in $G(n,m)$ with $m = \lfloor p\binom{n}{2}\rfloor$. Therefore we have that
\[\mb{P}[|F(\mbf{x})-\alpha_{n}|\le n^{\ell-5/2}] \ge \mb{P}\Bigg[|F(\mbf{x})-\alpha_{n}|\le n^{\ell-5/2} \cap \sum_e x_e = \Big\lfloor p\binom{n}{2}\Big\rfloor\Bigg]\]
\[\gtrsim_{p}(c_H + o(1))/n\gtrsim_{p,H} 1/n\]
where in the final step we have used convergence in distribution to a Gaussian. The theorem then follows as one of the integers in the range $[\alpha_n-n^{\ell-5/2},\alpha_n + n^{\ell-5/2}]$ has the desired property. 
\end{proof}
This immediately disproves the conjecture of Fox, Kwan, and Sauermann regarding anticoncentration of subgraph counts as the polynomial counting the number of copies of $K_2 + K_2$ trivially satisfies the conditions of \cref{thm:counterexample}, so it in turn proves the first part of \cref{thm:counterexample-main}. We remark that in fact that any disjoint union of edges also trivially satisfies the conditions of \cref{thm:counterexample}.

We similarly disprove the conjecture of Fox, Kwan, and Sauermann regarding anticoncentration of induced subgraph counts (the second part of \cref{thm:counterexample-main}) by constructing a graph $H$ for which the polynomial counting induced subgraphs of $H$ satisfies the conditions of \cref{thm:counterexample}. However the construction here is significantly more intricate. We sketch in greater detail how to arrive at such a graph. The construction here is closely related to the method used by K\"{a}rrmann \cite{K94} to construct superproportional graphs. The algorithm we use to search for an appropriate graph is based on that work.

The key idea here is to note that whether $\Phi_{H}$ vanishes or not is a condition that is independent of $n$ and is only based on the density of various $4$-vertex subgraphs present in $H$. In particular note that for a given graph $H$ the number of induced copies of $H$ is
\[X_{H} = \sum_{\substack{V\subseteq [n]\\|V|=v(H)}}\sum_{\substack{E\subseteq\binom{V}{2}\\E\simeq H}}\prod_{e\in E}x_e\prod_{e\in \binom{V}{2}\backslash E}(1-x_e).\]
Letting $\chi_e = (x_e-p)/\sqrt{p(1-p)}$ we have
\[\frac{X_{H}}{p^{e(H)}(1-p)^{\ol{e}(H)}} = \sum_{\substack{V\subseteq [n]\\|V|=v(H)}}\sum_{\substack{E\subseteq \binom{V}{2}\\E\simeq H}}\prod_{e\in E}(1+\sqrt{(1-p)/p}\chi_e)\prod_{e\in \binom{V(H)}{2}\backslash E}(1-\sqrt{p/(1-p)}\chi_e)\]
where we recall $\ol{e}(H) = \binom{v(H)}{2}-e(H)$.
This can be rewritten as 
\[\frac{X_{H}}{p^{e(H)}(1-p)^{\ol{e}(H)}} = \sum_{\substack{V\subseteq [n]\\|V|=v(H)}}\sum_{\substack{E\subseteq \binom{V}{2}\\E\simeq H}}\sum_{T\subseteq\binom{V}{2}}\prod_{e\in E\cap T}\Big(\sqrt{\frac{1-p}{p}}\chi_e\Big)\prod_{e\in T\setminus{E}}\Big(-\sqrt{\frac{p}{1-p}}\chi_e\Big).\]
Let $\delta_T$ be the coefficient of $\gamma_T(\mbf{x})$ in the expansion of $H$.  Let $q = -p/(1-p)$ and $\on{ind}(S,H)$ be the number of induced subgraphs of $H$ isomorphic to $S$. We find that for any $T$ without isolated vertices,
\[\delta_T = \left(\sqrt{\frac{1-p}{p}}\right)^{e(T)}\frac{\binom{n}{v(H)}v(H)!\on{aut}T}{\binom{n}{v(T)}v(T)!\on{aut}H}\sum_{v(S)=v(T)}\delta_{S,T}(q)\on{ind}(S,H)\]
for appropriate polynomials $\delta_{S,T}(q)$ with positive coefficients. Here the sum ranges over the isomorphism classes $S$ of graphs on $v(T)$ vertices. Explicitly,
\[\delta_{S,T}(q) = \sum_{\substack{T'\subseteq\binom{V(S)}{2}\\T'\simeq T}}q^{|E(T')\setminus{E(S)}|}.\]
This formula is derived by conditioning on which size $v(T)$ subset of $V$ is used to embed $T$ in. Such a subset yields an induced subgraph $S$ of the copy of $H$ considered. Then, summing over all embeddings of $T$ gives the result.

Suppose $v(T)\le 4$. Instead of conditioning on a size $v(T)$ subset of $V$, condition on a size $4$ subset of $V$ (possibly overcounting) that contains the copy of $T$ we wish to count. A similar calculation yields
\[\delta_T\propto\sum_{v(S)=4}\delta_{S,T}(q)\on{ind}(S,H),\]
with $\delta_{S,T}$ defined in the same way as before. We do not bother computing the exact pre-factors, as we only ultimately care whether a term is zero or not. The following is a table of $\delta_{S,T}$. The graph $D_k$ is the empty graph on $k$ vertices, $P_k$ the path on $k$ edges, and $+$ denotes disjoint union while $-$ denotes complementation within $K_4$. Thus $-(K_2+D_2)$, for instance, is $K_4$ without an edge.
\begin{figure}[h!]
\centering
\scalebox{0.65}{
\begin{tabular}{c|c|c|c|c|c|c|c|c|c|c|c|c|c|c|c|c}
$K_4$-embedding:
            &$1$  &$6$   &$12$       &$4$         &$3$      &$12$  &$1$    &$6$           &$12$           &$4$         &$3$          \\
\hline Induced $S$
$\setminus T$ 
            &$K_0$&$K_2$ &$P_2$      &$K_3$       &$K_2+K_2$&$P_3$ &$K_4$  &$K_4-K_2$     &$K_4-P_2$ &$K_{1,3}$   &$C_4$       \\
\hline
$D_4$       &$1$  &$6q$  &$12q^2$    &$4q^3$      &$3q^2$   &$12q^3$    &$q^6$  &$6q^5$        &$12q^4$        &$4q^3$      &$3q^4$     \\
$K_2+D_2$   &$1$  &$5q+1$&$8q^2+4q$  &$2q^3+2q^2$ &$2q^2+q$ &$6q^3+6q^2$  &$q^5$  &$q^5+5q^4$    &$4q^4+8q^3$    &$2q^3+2q^2$ &$q^4+2q^3$  \\
$P_2+K_1$   &$1$  &$4q+2$&$5q^2+6q+1$&$q^3+2q^2+q$&$q^2+2q$ &$2q^3+8q^2+2q$ &$q^4$  &$2q^4+4q^3$   &$q^4+6q^3+5q^2$&$q^3+2q^2+q$&$2q^3+q^2$ \\
$K_2+K_2$   &$1$  &$4q+2$&$4q^2+8q$  &$4q^2$      &$2q^2+1$ &$4q^3+4q^2+4q$ &$q^4$  &$2q^4+4q^3$   &$8q^3+4q^2$    &$4q^2$      &$q^4+2q^2$\\
$K_3+K_1$   &$1$  &$3q+3$&$3q^2+6q+3$&$3q^2+1$    &$3q$     &$6q^2+6q$  &$q^3$  &$3q^3+3q^2$   &$3q^3+6q^2+3q$ &$q^3+3q$    &$3q^2$     \\
$P_3$       &$1$  &$3q+3$&$2q^2+8q+2$&$2q^2+2q$   &$q^2+q+1$&$q^3+5q^2+5q+1$ &$q^3$  &$3q^3+3q^2$   &$2q^3+8q^2+2q$ &$2q^2+2q$   &$q^3+q^2+q$\\
$K_{1,3}$   &$1$  &$3q+3$&$3q^2+6q+3$&$q^3+3q$    &$3q$     &$6q^2+6q$  &$q^3$  &$3q^3+3q^2$   &$3q^3+6q^2+3q$ &$3q^2+1$    &$3q^2$    \\
$C_4$       &$1$  &$2q+4$&$8q+4$     &$4q$        &$q^2+2$  &$4q^2+4q+4$  &$q^2$  &$4q^2+2q$     &$4q^2+8q$      &$4q$        &$2q^2+1$  \\
$-(P_2+K_1)$&$1$  &$2q+4$&$q^2+6q+5$ &$q^2+2q+1$  &$2q+1$   &$2q^2+8q+2$ &$q^2$  &$4q^2+2q$     &$5q^2+6q+1$    &$q^2+2q+1$  &$q^2+2q$    \\
$-(K_2+D_2)$&$1$  &$q+5$ &$4q+8$     &$2q+2$      &$q+2$    &$6q+6$    &$q$    &$5q+1$        &$8q+4$         &$2q+2$      &$2q+1$       \\
$K_4$       &$1$  &$6$   &$12$       &$4$         &$3$      &$12$  &$1$    &$6$           &$12$           &$4$         &$3$          
\end{tabular}}
\caption{Table of polynomials $\delta_{S,T}(q)$}
\end{figure}
Now let $p=1/2$ so that $q=-1$. We wish $\delta_T$ to be nonzero for $T=K_2$ but zero for connected graphs of size $3$ and $4$. This ultimately imposes $8$ linear conditions on $\on{ind}(S,H)$ when $v(S) = 4$, beyond the additional condition that the sum of all these is $\binom{v(H)}{4}$. Finally, induced subgraph counts of degree $4$ satisfy an additional (quadratic) relation due to the fact that both $\on{ind}(K_2,H)$ and $\on{ind}(K_2,H)^2$ can be expanded as a sum of these $\on{ind}(S,H)$ for $v(S) = 4$ (the resulting equations have a dependence on $v(H)$).

Thus, after we fix $v(H)$, the eleven counts
\[\on{ind}_4(H) = (\on{ind}(D_4,H),\on{ind}(K_2+D_2,H),\ldots,\on{ind}(K_4,H))\]
ought to be constrained by one parameter. As it turns out, we end up needing $16$ to divide $\binom{v(H)}{4}$, as well as satisfy some square root integrality constraints which constrain this one parameter. The fact that $\delta_{K_2}$ must be nonzero corresponds to the fact that $H$ must not have exactly $\binom{v(H)}{2}/2$ edges.

Ultimately, the smallest possible example satisfies $v(H) = 64$ with
\[\on{ind}_4(H) = (11835, 67163, 126632, 31723, 39646, 119198, 39646, 27941, 111504, 52035, 8053).\]
There are actually $10$ resulting vectors ($5$ after removing complementation symmetry). There is also a vector satisfying the constraints above except that $\delta_{K_2} = 0$; this in fact is precisely the vector derived in \cite{K94} as necessary for a $64$-vertex superproportional graph. The above vector forces us to have $976$ edges, which is the closest possible to $\binom{64}{2}/2 = 1008$ without actually being $1008$.

Now, using a modification of an algorithm due to \cite{K94}, we find such a graph $H$. The adjacency matrix is provided in \cref{app:adjacency-matrix}. Code in Java for both constructing and verifying the construction is provided in the arXiv listing of the paper. Finally, since this $H$ has the number of induced copies $X_H$ satisfy the hypotheses of \cref{thm:counterexample}, we see that the second half of \cref{thm:counterexample-main} is proven.

\section{Local Limit Theorem for $k$-Term Arithmetic Progressions}\label{sec:kAP}
Fix $k\ge 3$, which we will treat as constant throughout. Then fix $\lambda\in (0,1/2)$ and choose $n\ge 1$ with $\gcd(n,(k-1)!) = 1$. Then choose an integer $m$ with $p=m/n\in (\lambda,1-\lambda)$. We will show that
\[\kAP(\mbf{x}) = \sum_{a\in \mb{Z}/n\mb{Z}}\sum_{d\in [n/2]} \prod_{i=0}^{k-1}x_{a+id}\]
satisfies a local central limit theorem when we uniformly sample $\mbf{x} = (x_i)_{i\in\mb{Z}/n\mb{Z}}$ among $\{0,1\}$ vectors with $\sum x_i = m$. Write $y_i = (x_i-p)/\sqrt{p(1-p)}$ in order to expand into a $p$-biased basis, and let $y_T = \prod_{i\in T}y_i$. We obtain the expression
\[\kAP'(\mbf{y})=\sum_{\ell=3}^{k}\sum_{a\in \mb{Z}/n\mb{Z}}\sum_{d\in [n/2]} \sum_{S\in {\binom{[k]}{\ell}}} p^{k - \frac{|S|}{2}}(1-p)^{\frac{|S|}{2}}\prod_{i \in S} y_{a+id}\]
after removing the linear and quadratic terms, which are deterministic given $m$. Let $\sigma$ be the standard deviation of $\kAP'(\mbf{y})$ if $\mbf{y}$ were drawn independently (instead of with fixed sum); we will often switch between the independent model and the constrained model, and will note these shifts as they come. Note that $\mb{E}[\kAP'(\mbf{y})] = 0$ in the independent model, and also note $\sigma = \Theta_{k,p}(n)$. We prove that
\[|\mb{E}[e^{it\kAP(\mbf{y})/\sigma}] - e^{-t^2/2}|\]
is small for all $t\in [-\pi \sigma, \pi\sigma]$ in the constrained model. This will later be used to characterize the distribution of $\kAP(x)$ in the independent model, although both results are of interest. 

\subsection{Bounds for $|t|\le n^\eps$} To handle these cases, we see that $\kAP'(\mbf{y})$ is Gaussian in the independent model, and then transfer to the conditioned model.

Define for $k$-order linear forms the norm
\[\|A\|_{\text{op}} = \sup_{\|v_i\|_2=1}|A(v_1,v_2,\ldots,v_k)|.\]
Next define the $k^{\text{th}}$ derivative (tensor) operators for $f\in C^{k}(\mb{R}^n)$ as 
\[\langle D^kf(x),(u_1,\ldots,u_k)\rangle = \sum_{i_1,i_2,\ldots,i_k\in [n]}\frac{\partial^k f}{\partial x_{i_1}\ldots \partial x_{i_k}}(u_1)_{i_1}\ldots(u_k)_{i_k}\]
for vectors $u_1,\ldots,u_k\in\mb{R}^n$, and let
\[M_r(g) = \sup_{x\in\mb{R}^n} \|D^rg(x)\|_{\text{op}}.\]
Finally define
\[\nkAP^{\ell}(\mbf{y}) = \frac{1}{\sigma_{\ell}} \sum_{a\in \mb{Z}/n\mb{Z}}\sum_{d\in [n/2]} \sum_{S\in {\binom{[k]}{\ell}}}\prod_{i \in S} y_{a+id}\]
where $\sigma_\ell$ is chosen so that $\on{Var}[\nkAP^{\ell}(\mbf{y})] = 1$ where $y_i = (x_i - p)/\sqrt{p(1-p)}$ if $x_i\sim\on{Ber}(p).$ The key technical result of \cite{BSS20} is the following quantitative central limit theorem.
\begin{theorem}\label{thm:invariance}
Let $\gcd(n,(k-1)!) = 1$ and let $y_i$ be defined as before and $z_{i}'$ be standard normal random variables. Then for any $C^3$ function $\psi:\mb{R}^{k-1}\to \mb{R}$ we have
\[\bigg| \mb{E}[\psi(z_{1}',z_{3}',\ldots,z_{k}'))-\psi(\nkAP^{1}(\mbf{y}), \nkAP^{3}(\mbf{y}),\ldots,\nkAP^{k}(\mbf{y}))]\bigg|
\lesssim_{k,p} \frac{M_3(\psi)+M_2(\psi)}{n^{1/2}}.\]
\end{theorem}
Using this we prove the following lemma. 
\begin{lemma}\label{lem:kap-characteristic-low}
For all $\eps>0$ we have
\[\left|\mb{E}_x\left[e^{it\kAP'(\mbf{y})/\sigma}|\sum y_i=0\right] - e^{-t^2/2}\right|\lesssim_{k,\lambda,\eps} t/n^{1/4-\eps} + t^3/n^{1/2}.\]
\end{lemma}
\begin{remark}
The $\lambda$ dependence merely reflects that $p$ is bounded away from $0$ and $1$ in terms of $\lambda$, which controls the constants in all inequalities.
\end{remark}
\begin{proof}
The proof is similar to the proof of \cref{lem:low-t}, except that we have a slightly worse bound in the independent model to start with. The key idea is to bootstrap from the previous result. We perform the following procedure:
\begin{enumerate}[1.]
    \item Sample each element $y_i$ to be $\sqrt{1-p}/\sqrt{p}$ with probability $p$ and $-\sqrt{p}/\sqrt{1-p}$ with probability $1-p$. 
    \item Adjust a random subset of the $\sqrt{1-p}/\sqrt{p}$ to $-\sqrt{p}/\sqrt{1-p}$ or vice versa so that the sum is as desired. Specifically, determine which direction the adjustment needs to be and uniformly sample the correct amount in the correct direction.
\end{enumerate}
Let $y_i$ be the initial sample and $y_i'$ be the sample of elements adjusted. Let $S_{0}$ a random variable indicating the set of changed random variables and $S_{1}$ which direction the change occurred in. We will show that
\[\mb{P}[|\kAP'(\mbf{y})-\kAP'(\mbf{y'})|\ge n^{3/4+\eps}]\]
is super-polynomially small for any $\eps>0$. First we find $\mb P[|S_0|\ge n^{1/2}\log n] = \exp(-\Omega_\lambda((\log n)^2))$. Next, note that given $S_0,S_1$ we can write $\kAP'(\mbf{y})-\kAP'(\mbf{y'})$ as a bounded degree polynomial in $y_j$ for $j\notin S_0$ since we know exactly what changed and how. More specifically, write it as $\sum_{T\subseteq S_0^c}a_Ty_T$. Note $a_T$ are random variables that are functions of $S_0,S_1$. Every term $y_T$ in the expansion must correspond to a subset of a $k$-AP (i.e. a term in the original $\kAP'$ polynomial) that hits $S_0$ (so that its coefficient may change). Furthermore, each coefficient in $\kAP'$ is bounded in terms of $\lambda, k$. Thus
\[\sum_{T\subseteq S_0^c}|a_T|\lesssim_{k,\lambda}|S_0|n\lesssim_{k,\lambda}n^{3/2}\log n,\]
the latter inequality occurring with high probability over the randomness of $S_0,S_1$. Now we claim with high probability that
\[\sum_{T\subseteq S_0^c}a_T^2\lesssim_{k,\lambda}n^{3/2}(\log n)^2.\]
To do this we make the following three observations, assuming that $|S_0|$ is small (since it occurs with high probability). Before we dive in, we note that given $|S_0|,S_1$, the set $S_0$ is distributed uniformly. Also, the contributing terms to $a_T$ are all subsets $S$ of $S_0$ such that $S\cup T$ is a subset of a $k$-AP of size at least $3$.
\begin{enumerate}[1.]
    \item For all $|T|\ge 2$ note that $|a_T|\lesssim_{k,\lambda} 1$ as given $2$ elements in a $k$-AP there are only $\Theta_k(1)$ ways to extend it. Thus the $L^1$ bound above suffices to establish the estimate for these coefficients.
    \item For $|T| = 1$, say $T = \{b\}$, recall that $a_T$ is a sum over subsets $S$ of $S_0$ such that $S\cup\{b\}$ is a subset of a $k$-AP of size at least $3$. Given $|S_0|$, we will show this count is small with high probability over the randomness of $S_0$. Note that this random variable is monotonic. Additionally, if we consider picking elements of $S_0$ independently with probability $(\log n)^2n^{-1/2}$ instead, we see with probability greater than $1/3$, say, that $|S_0|$ is bigger than it needs to be. These facts combined show that it suffices to bound the count of valid $S$ in this new process with high probability.
    
    Now partition the set of possible $S$ into $\Theta_k(1)$ pieces of size $\Theta_k(n)$ such that in each piece the elements do not intersect. (This can be accomplished explicitly, or by noting that the intersection graph has bounded degree, and coloring this graph then rebalancing.) By Chernoff with probability at least $1-\exp(-(\log n)^4)$ the coefficient is at most $O_k((\log n)^4)$. Square it and multiply by $n$ after a union bound.
    \item Finally for the constant coefficient this is bounded (up to factors depending on $k$ and $\lambda$) by the number of such subsets of $|T|\ge 3$ in $S_0$ so that they are a subset of a $k$-AP. Now consider the martingale where each element of $S_0$ is revealed on at a time. Since each element can be in at most $\lesssim_{k}|S_0|$ such sets we have concentration of order $|S_0|^{3/2}$ by Azuma--Hoeffding and thus this coefficient is bounded by $n^{3/4}\log n$ with high probability as the expectation is $\lesssim n^{3/4}$, which gives bounds of the claimed quality. (Note that one can by a more subtle argument bound this coefficient by $n^{1/2}(\log n)^C$ but this easier bound is sufficient.)
\end{enumerate}
In summary, we have shown that the above bound on these coefficients occurs with probability $1-\exp(-\Omega_{k,\lambda}((\log n)^2))$ over the randomness of $S_0,S_1$.

Now suppose we are in one of these cases (and also suppose $|S_0|\le n^{1/2}\log n$, since we can), and fix the values $S_0,S_1$. We have that $\kAP'(\mbf{y})-\kAP'(\mbf{y'})$ is a polynomial of bounded degree in $\mbf{y}$, which are now drawn independently with a fixed sum based on $S_0$, which is at most $O_\lambda(n^{1/2}\log n)$. Shift back to the boolean model, so that we have some polynomial in $x\in\{0,1\}^{S_0^c}$ where we condition on a sum of size $pn+O(n^{1/2}\log n)$. Consider an independent model of selecting the $x$'s with the same sum. Its probability $q$ is bounded away from $0,1$ in terms of $\lambda$. By hypercontractivity (\cref{thm:concentration-hypercontractivity}), we see that the probability of $|\kAP'(\mbf{y})-\kAP'(\mbf{y'})|\ge n^{3/4+\eps}$ in the independent model of selecting $S_0^c$ is $\exp(-\Omega_{k,\lambda}(n^{\eps'}))$ where $\eps'$ depends only on $\eps, k$. Therefore in the conditioned sum model of $S_0^c$, which occurs with probability $\Omega(1/n)$ in the independent model, we still have that this event occurs with small probability (otherwise the independent model would have a bigger probability).

Now we can add back in the cases where $S_0,S_1$ are not sufficiently nice to merit the above bounds hence the above deduction. Overall, we find (over all the randomness) that
\[\mb P[|\kAP'(\mbf{y})-\kAP'(\mbf{y'})|\ge n^{3/4+\eps}] = \exp(-\Omega_{k,\lambda,\eps}((\log n)^2)).\]
Finally, note that $\kAP'(\mbf{y}) = \sum_{\ell=2}^{k}\sigma_{\ell}\nkAP^{\ell}(\mbf{y})$ where $\sigma_{\ell} = \Theta_{k,l}(n)$. Thus 
\begin{align*}
|\mb{E}[e^{it\kAP'(\mbf{y'})/\sigma}]&-e^{-t^2/2}|\\
&\le \bigg|\mb{E}\bigg[e^{it\sum_{\ell = 3}^{k}\sigma_{\ell}\nkAP^{\ell}(\mbf{y})/\sqrt{\sum_{\ell =3}^{k}\sigma_{\ell}^2}}-e^{it\sum_{\ell = 3}^{k}\sigma_{\ell}\nkAP^{\ell}(\mbf{y'})/\sqrt{\sum_{\ell =3}^{k}\sigma_{\ell}^2}}\bigg]\bigg| \\
&\quad+ \bigg|\mb{E}\bigg[e^{it\sum_{\ell = 3}^{k}\sigma_{\ell}\nkAP^{\ell}(\mbf{y})/\sqrt{\sum_{\ell =3}^{k}\sigma_{\ell}^2}}]-e^{-t^2/2}\bigg]\bigg|\\
&\lesssim_{k,\lambda,\eps}\exp(-\Omega_{k,\lambda,\eps}((\log n)^2)) + t|n^{3/4+\eps}/n| + t^3/n^{1/2}\\
&\lesssim_{k,\lambda,\eps}t/n^{1/4-\eps} + t^3/n^{1/2}
\end{align*}
where in the final inequality we have used the coupling inequality between the two distributions, as well as the fact that $M_3(g)$ for $g(\mbf{a}) = e^{it(\mbf{r}\cdot\mbf{a})}$, $\mbf{r}\in\mb R^k$, is $O_{k,\mbf{r}}(t^3)$.
\end{proof}
Note that this bound allows us to handle any $|t|\le n^{1/8-\eps}$ but we will only use the the bound up to $|t|\le n^\eps$ to obtain better bounds.

\subsection{Bounds for $n^{\eps}\le|t|\le n^{1-\eps}$}\label{sub:kap-intermediate}
For this section, relying on decoupling techniques, on a first pass the reader is advised to ignore various technical maneuvers required to deal with the fact that our random variables are constrained to live on a slice of the hypercube.

Fix a real number $\beta\in (0,1)$ and set $S = \lfloor n^\beta\rfloor$ (we will take care to ensure $\beta$ is bounded away from the endpoints). Then take $A_j = \lfloor jn/(10k^2) + n/2\rfloor + 2[jS, (j+1)S)$ for $1\le j \le k-1$. Note that these are essentially intervals of length $n^\beta$ that are spaced a constant fraction. It follows easily that any rainbow $k$-APs containing an element in $A_1, A_{2},\ldots, A_{k-1}$ (and an element in their complement) are forced to be genuine $k$-APs in $\mb{Z}$; this follows trivially noting that all the sets are sufficiently close to the center of $\mb{Z}/n\mb{Z}$ and the common difference between them is $ n/(10k) + \Theta(n^\beta)$. Set $T = (\mb{Z}/n\mb{Z})\backslash\cup_{i=1}^{k-1}A_{i}$, and note $|T| = \Theta_k(n)$.

Now consider the following random process: select a uniform random subset of $\mb{Z}/n\mb{Z}$ of size $m$, then resample the elements not in $T$ conditional on the outcome within $T$. Let $X$ be the indicator vector of the subset of $T$ chosen, and $Y_j^b$ for $b\in\{0,1\}$ be the subset of $A_i$ chosen, as well as the resample. As introduced in the decoupling section, let $\mbf{Y}=(Y_{1}^{0},Y_{1}^{1},\ldots,Y_{k-1}^{0},Y_{k-1}^{1})$. We can alternatively view this process as sampling $Z_i$, the number of edges chosen in each $A_i$ if a subset is chosen with $m$ elements uniformly, and then sampling $X$ and two independent copies of $Y_i^0,Y_i^1$ (conditional on $Z_i$).

Let $x_i$ for $i\in T$ be $1$ or $0$ depending on if $i\in X$, and let $x_i^b$ for $i\in A_j$ and $b\in\{0,1\}$ be $1$ or $0$ depending on whether $i\in Y_j^b$. Now note that, recalling from \cref{sub:decoupling-methods} that non-rainbow functions are in the kernel of $\alpha$,
\[\alpha(\kAP)(X,\mbf{Y}) = \sum_{i\in T} x_i\sum_{\substack{\text{Rainbow }k\text{-AP}\\\mc{A}\text{ including }i}}\prod_{j\in\mc{A}}(x_j^1-x_j^0)\]
hence 
\[\alpha(\kAP)(X,\mbf{Y}) = \sum_{i\in T}a_ix_i\]
where 
\[a_i = \sum_{\substack{\text{Rainbow }k\text{-AP}\\\mc{A}\text{ including }i}}\prod_{j\in\mc{A}}(x_j^1-x_j^0).\]
We aim to prove for $n^{\eps}\le |t|\le n^{1-\eps}$ that 
\[|\varphi_{\kAP'/\sigma}(t)|\lesssim_\eps\exp(-\Omega_{k,\lambda}(n^{\eps'}))\]
where $\eps'$ depends only on $k,\eps$. To begin, by \cref{lem:vdC-dep} we have
\[|\varphi_{\kAP'/\sigma}(t)|^{2^k}\le\mb E_{\mbf{Y}}\big|\mb E_Xe^{it\alpha(\kAP)(X,\mbf{Y})/\sigma}\big|,\]
where the change from $\kAP'$ to $\kAP$ occurs merely because $\kAP',\kAP$ are the same up to a constant in the conditioned model.

Hence it will suffice to show with high probability over the randomness of $\mbf{Y}$ that the inner expectation is small. Since given $\mbf{Y}$, $X$ is chosen uniformly with some fixed sum, we will be able to apply \cref{lem:bernoulli-variance}. To do so, we need to know that $\on{Var}[a_i]$ is large. We also need $|a_i|$ to be not too large for all $i\in T$.

To prove this occurs with high probability, it will be more convenient to pretend that $\mbf{Y}$ is sampled independently with probability $p$. To transfer from this to the true distribution, we use the following argument (which was also used for the graph statistic results in \cref{sub:bounds-intermediate}).

Define a \emph{suitable} outcome to be if $|Z_i-p|A_i||\le\sqrt{|A_i|}\log |A_i|$, say. By Azuma--Hoeffding and union-bounding over the fixed number of sets considered, there is an overwhelming probability that all $Z_i$ are suitable. In particular, probability of failure is $\exp(-\Omega_\lambda((\log n)^2))$.

If we sample the elements of $A_i$ with probabiility $p$ independently (sampling twice) then we attain any suitable number of elements in all $B_i$ with probability at least $\exp(-\Omega_\lambda((\log n)^2))$. Thus if in this independent model, an event has probability at most $\exp(-\Omega_\lambda((\log n)^3))$, then even in the conditioned size model within suitable outcomes it occurs with this probability, perhaps weakening the constants in the exponent (we used a similar trick in the small $|t|$ regime as well.) Then we must add back in the unsuitable cases, which account for a probability of at most $\exp(-\Omega_\lambda((\log n)^2))$ as noted above.

We now proceed to prove the desired control (with high probability) on $\on{Var}[a_i]$ in the independent model of selecting the $Y_j^b$, which as shown above will transfer to the desired model. We do so by proving a number of bounds on these coefficients. Note that the $a_i$ are polynomials in the $y_j^b$, which are now being selected independently with probability $p$. It is worth noting that $a_i$ is a nonzero polynomial only for $\Theta_k(n^\beta)$ values $i\in T$, as the rainbow $k$-APs can only include within $T$ elements from two regions of prescribed width. To be more precise we prove the following lemma.
\begin{lemma}\label{lem:various-bounds-kAP}
Let $X$, $\mbf{Y}$, $a_i$, $T$ be as above. Let $C$ be a suitably large constant. Then have the following concentration bounds in the model where each element is sampled with probability $p$.
\begin{enumerate}
    \item We have that
    \[\mb{P}\bigg[\sup_{i\in T}|a_i|\ge n^{\beta/2}(\log n)^C\bigg]\le \exp(-\Omega_{\lambda}((\log n)^{3})).\]
    
    \item We have that 
    \[\mb{P}\bigg[\bigg|\sum_{i\in T}a_i\bigg|\le n^{\beta}(\log n)^{C}\bigg]\le \exp(-\Omega_{\lambda}((\log n)^{3})).\]
    \item We have that
    \[\mb{E}\bigg[\sum_{i\in T}a_i^2\bigg] = \Theta_\lambda(n^{2\beta})\]
    and 
    \[\on{Var}\bigg[\sum_{i \in T}a_i^2\bigg] = O_{\lambda}(n^{3\beta/2}(\log n)^{2C}).\]
\end{enumerate}
\end{lemma}

\subsubsection{Proof of \cref{lem:various-bounds-kAP} (1)}
We show $|a_i|\lesssim_k n^{\beta/2}(\log n)^C$ with high probability. Note that $a_i$ is composed of $O_k(n^\beta)$ terms with constant coefficients. Therefore the sum of squares of its coefficients is $O_k(n^\beta)$, and hypercontractivity (\cref{thm:concentration-hypercontractivity}) immediately gives the desired result for $C$ chosen large enough depending on $k$ (which bounds the degree of the polynomial considered). Recall that in fact $a_i = 0$ except for $\Theta_k(n^\beta)$ values in admissible ``target'' regions to the left and right of the $A_i$.

\subsubsection{Proof of \cref{lem:various-bounds-kAP} (2)}
We show $\sum_{i\in T} a_i\lesssim_k n^\beta(\log n)^C$ with high probability. Note that $a_i$ is composed of $O_k(n^\beta)$ terms with constant coefficients, and each term extends to be included in at most $2$ polynomials $a_i$. Therefore the sum of squares of coefficients of the total polynomial is $O_k(n^{2\beta})$ because of our bound on nonzero polynomials. By hypercontractivity (\cref{thm:concentration-hypercontractivity}) the result follows, with $C$ chosen large enough depending on $k$ (which bounds the degree of the polynomials considered) to obtain a good enough concentration.

\subsubsection{Proof of \cref{lem:various-bounds-kAP} (3)}
We show $\sum_{i\in T}a_i^2$ concentrates on a value of size $\Theta_{k,\lambda}(n^{2\beta})$. Note that $\mb E[a_i^2] = \Theta_{k,\lambda}(n^\beta)$ for most $i\in T$ with nonzero $a_i$ since the only nonzero terms come from choosing the same rainbow $k$-AP twice, which yields $\Theta_k(n^\beta)$ possibilities for $i$ near the middle of the two nonzero target regions. Near the fringes, there could be less, but this does not affect the order of magnitude given.

Now it suffices to show that the standard deviation of $\sum_{i\in T}a_i^2$ (over the randomness of the independently chosen $x_j^b$) is of smaller order by a power of $n$; hypercontractivity (\cref{thm:concentration-hypercontractivity}) will then give the desired concentration. The desired variance is
\[\sum_{i,j\in T}(\mb E[a_i^2a_j^2]-\mb E[a_i^2]\mb E[a_j^2]).\]
For terms $i=j$, there are $O_k(n^\beta)$ of them with nonzero $a_i$. By the $L^\infty$ bound above, with high probability $|a_i|\lesssim_k n^{\beta/2}(\log n)^C$, hence the expectation terms are bounded by $n^{2\beta}(\log n)^{4C}$ for a total of $O_{k,\lambda}(n^{3\beta}(\log n)^{4C})$, which is acceptable.

For terms $i\neq j$, expand into further covariance terms coming from considering two rainbow $k$-APs $\mc{A}_1,\mc{A}_2$ including $i$ and $\mc{B}_1,\mc{B}_2$ including $j$:
\[\mb E[a_i^2a_j^2]-\mb E[a_i^2]\mb E[a_j^2] = {\sideset{}{^\ast}\sum_{\substack{\mc{A}_1,\mc{A}_2\text{ incl. }i\\\mc{B}_1,\mc{B}_2\text{ incl. }j}}}\prod_{b\in\{1,2\}}\bigg[\prod_{j\in\mc{A}_b}(x_j^1-x_j^0)\prod_{j\in\mc{B}_b}(x_j^1-x_j^0)\bigg],\]
where $\sum^\ast$ denotes a sum over all the $k$-APs being rainbow.

The covariance term coming from these vanishes unless every element in the union of the four progressions is covered at least twice. If $\mc{A}_1 = \mc{A}_2$ we see that either there are $O_k(1)$ choices for $\mc{B}_1,\mc{B}_2$ (so they both hit $\mc{A}_1$) otherwise they must be equal and disjoint from $\mc{A}_1$. In the latter case, we have independence so the term is zero. In the former case, we have a term of size $O_{k,\lambda}(n^\beta)$ as there are $O_k(n^\beta)$ choices for $\mc{A}_1$. Summing over all $i,j$ we obtain a contribution of $O_{k,\lambda}(n^{3\beta})$, which is acceptable. A similar analysis holds if $\mc{B}_1 = \mc{B}_2$.

If both pairs are unequal, we see that $\mc{A}_1,\mc{A}_2$ are disjoint and the same for the others. Hence $\mc{A}_1\cup\mc{A}_2$ spans $2k-2$ disjoint vertices and the same for $\mc{B}_1\cup\mc{B}_2$, hence these unions must be identical else the term is zero. After choosing $i,\mc{A}_1,\mc{A}_2$ we see there are $O_k(1)$ choices for $\mc{B}_1,\mc{B}_2$ and $j$. This gives a contribution of $O_{k,\lambda}(n^{3\beta})$ again.

Overall, we have shown that the variance is $O_{k,\lambda}(n^{3\beta}(\log n)^{4C})$ hence the standard deviation is $O_{k,\lambda}(n^{3\beta/2}(\log n)^{2C})$, which is indeed much smaller than $n^{2\beta}$, so we are done.

This (finally) concludes the proof of \cref{lem:various-bounds}. Now we use these bounds to conclude our argument in the intermediate range of $|t|$.

\subsubsection{Putting it together}
\begin{lemma}\label{lem:kap-medium-t}
For $\eps > 0$ and $|t|\in[n^{2\varepsilon},\sigma n^{-4\varepsilon}]$ we have
\[\left|\mb{E}_x\left[e^{it\kAP'(\mbf{y})/\sigma}|\sum y_i=0\right] - e^{-t^2/2}\right|\lesssim \exp(-\Omega_{k,\lambda}((\log n)^2)) + \exp(-\Omega_{k,\lambda}(n^{\eps})).\]
\end{lemma}
\begin{proof}
Note that $e^{-t^2/2}$ is sufficiently small in this range that it can be ignored. Let $X,\mbf{Y}$ be as at the beginning of \cref{sub:kap-intermediate}. By \cref{lem:vdC-dep}, we have
\[\left|\mb{E}_x\left[e^{it\kAP'(\mbf{y})/\sigma}|\sum y_i=0\right]\right|^{2^{k-1}}\le\big|\mb E_Xe^{it\alpha(\kAP')(X,\mbf{Y})/\sigma}\big|.\]

Now, with high probability over $\mbf{Y}$ (in the conditioned model) we have the bounds given in \cref{lem:various-bounds-kAP}. In particular, the probability of failure is $\exp(-\Omega_{k,\lambda}((\log n)^2))$. In particular, the average of $a_i^2$ is $n^{2\beta-1}$ in magnitude (since most of these values are zeros when $\beta$ is small, this could be less than $1$) while the average of $a_i$ is $n^{\beta-1}(\log n)^C$ in magnitude. The square of the latter is much smaller than the former, hence we see that $\on{Var}[a_i]$ is of order $\Theta_{k,\lambda}(n^{2\beta-1})$. Also, $|a_i|\le n^{\beta/2}(\log n)^C$ for all $i\in T$. This allows us to apply \cref{lem:bernoulli-variance} to deduce in these cases that
\[\big|\mb E_Xe^{it\alpha(\kAP')(X,\mbf{Y})/\sigma}\big|\lesssim\exp(-\Omega_{k,\lambda}((t^2n/\sigma^2)\on{Var}[a_i]))\lesssim\exp(-\Omega_{k,\lambda}(t^2n^{2\beta-2}))\]
as long as $|t/\sigma|n^{\beta/2}(\log n)^C\lesssim_\lambda 1$. Therefore the bound is good as long as $t\in [n^{1-\beta+\eps},n^{1-\beta/2-\eps}]$, say. Now varying $\beta$ between $5\eps$ and $1-\eps$ (taking care to keep it bounded away from the endpoints) we see this covers the range $[n^{2\eps},n^{1-4\eps}]$, which is good enough for our purposes.
\end{proof}

\subsection{Bounds for $n^{1-\eps}\le |t|\le \pi\sigma$}
For this section we develop a series of sets upon which the decoupling estimates will be formed. They stem from a tensor product construction which we outline below. The key difficulty is in ensuring that the decoupled expression has many coefficients that are $\pm 1$, so that we can apply \cref{lem:bernoulli-variance} up to the values $t=\pm \pi\sigma$.

\begin{lemma}
Suppose $k\ge 4$. Let $A_i = \{k+4i, k+4i+2\}$ for $1\le i\le k-1$. Then the only $k$-term arithmetic progressions with $1$ element in each of the $A_i$ are $\{k+4i\}$ for $0\le i\le k-1$, $\{k+4i\}$ for $1\le i\le k$, $\{k+4i+2\}$ for $0\le i\le k-1$, and $\{k+4i + 2\}$ for $1\le i\le k$. Note that arithmetic progressions here means an arithmetic progression in $\mb{Z}$
\end{lemma}
\begin{proof}
Since the arithmetic progression in $\mb{Z}$ contains elements in $A_1$ and $A_2$ we must have that the common difference is in the set $2,4,6$. Using this the result follows as extending the progression from $A_1, A_2$ to $A_3$ forces the common difference to be $4$.
\end{proof}

For $k=3$ let $A_1 = \{16,18\}$ and $A_2 = \{22,32\}$. Then the $3$-APs which contain exactly one element in $A_1$ and $A_2$ are $\{0,16,32\}$, $\{10,16,22\}$, $\{16, 19, 22\}$, $\{16, 22, 28\}$, $\{18,22,26\}$, $\{4, 18, 32\}$, $\{14,18,22\}$,$\{16,24,32\}$, $\{18, 20, 22\}$, $\{18, 25, 32\}$, $\{18, 32, 46\}$, and $\{16, 32, 48\}$. Finally for $k=3$ define  \[A_0 = \{0,4,10,14, 19, 24, 20, 25, 26, 28, 46, 48\}\] and for $k\ge 4$ define \[A_0 = \{k,k+2,5k, 5k+2\},\]
which are precisely the elements to which these $k$-APs extend. Note that all the extensions are to distinct integers.
\begin{lemma}
Let $n$ be sufficiently large in terms of $k$. Embed $B_1 ,\ldots,B_{k-1}$ into $\mb{Z}/\ell\mb{Z}$ where $\ell = 100k^{3}$ where $B_i = 10k^2+A_i$, i.e. each element is shifted by a constant. Then the only $k$-terms APs with exactly $1$ element in each $B_i$ are as in the previous lemma.
\end{lemma}
\begin{proof}
For $k=3$ the proposition is trivial. For $k\ge 4$ note that two consecutive terms of the $k$-AP lie in a pair of $B_{i}$, $B_{j}$ and thus the common difference is bounded by $5k$ and therefore the AP cannot wrap around the edges. The result then follows by the previous lemma. 
\end{proof}
Finally define $C_i^u$ for $0\le i\le k-1$ to be the set of $u$-digit numbers in base $\ell = 100k^{3}$ with digits only in $B_i$. Choose $\ell$ so that $\ell^{u} = \Theta_k(n^{1/2})$ and thus $|C_0^u| = \Theta_k(n^{\delta_k})$ for some constant $\delta_k > 0$. Now embed $C_i^u$ into $\mb{Z}/n\mb{Z}$ as the sets $D_i^u$ by adding $\lfloor n/2\rfloor$ to all the elements. The key fact which follows from the previous lemmas is that the all $k$-APs with $1$ element in $D_i^u$ for $1\le i \le k-1$ complete into $D_0^u$. Indeed, we see that the common difference is $\Theta_k(n^{1/2})$ hence again there is no wrap-around as the numbers are all near $n/2$. 

We now use our decoupling lemma. Let $X,Y_j^0$ for $1\le j\le k-1$ be a fixed-sum sample of $m = pn$ elements from $\mb{Z}/n\mb{Z}$ where $Y_j^0$ is the subset of $D_j^u$ for $j\ge 1$ chosen and $X$ is the subset of $(\mb{Z}/n\mb{Z})/\cup_{j=1}^{k-1}D_j^u$ chosen. Then resample $Y_j^1$ for $1\le j\le k-1$ having the same sum as $Y_j^0$. Note that 
\[\alpha(\kAP')(X,\mbf{Y})(x) = \sum_{i\in D_0^u} x_i\prod_{\substack{j\in\text{Unique rainbow}\\\text{k-AP including }i}}(y_j^1-y_j^0)\]
\[= \sum_{i\in D_0^u} a_ix_i,\]
using that each $i$ is in a rainbow $k$-AP by construction only if it is in $D_0^u$, in which case it is in a unique such progression.

\begin{lemma}\label{lem:high-t-kAP}
Let $\kAP,\kAP',\sigma$ be as above. Then for $|t|\le\pi\sigma$,
\begin{equation}\label{eq:kAP-high}
|\varphi_{\kAP'/\sigma}(t)|\le\exp(-\Omega_{k,\lambda}(n^{\delta_k}))+\exp(-\Omega_{k,\lambda}(t^2n^{\delta_k}/\sigma^2)).
\end{equation}
\end{lemma}
\begin{proof}
Sample $X,\mbf{Y}$ as above. The key claim is that with high probability over the randomness of $\mbf{Y}$, we have a positive fraction of the coefficients $a_i$ for $i\in D_0^u$ are $0$, $1$, and $-1$. (Note that these are the only possible coefficients as each rainbow AP includes a unique element in $D_0^u$.).

We first consider the independent model of sampling. In it, we see there is a $\Theta_{k,\lambda}(1)$ probability of obtaining each of $\{0,\pm 1\}$ as a coefficient for each $i\in D_0^u$. By Azuma--Hoeffding, this translates to a probability of $1-\exp(-\Omega_{k,\lambda}(n^{\delta_k}))$ that each value occurs in a $\Theta_{k,\lambda}(1)$ fraction of the $a_i$ for $i\in D_0^u$. Now, if instead we sample $X,Y_1^0,\ldots,Y_{k-1}^0$ with constrained sum, and then resample $Y_1^1,\ldots,Y_{k-1}^1$, similar bounds hold by repeatedly applying Azuma--Hoeffding: first, the number of elements $X,\mbf{Y}$ are concentrated near a $p$ fraction, and then the coefficients $a_i$ are concentrated near some positive fraction. (See \cref{lem:top-connected,lem:top-induced} for similar arguments in the graph statistic setting.)

Now we apply \cref{lem:vdC-dep}. We again have
\[|\varphi_{\kAP'/\sigma}(t)|^{2^{k-1}}\le\mb E_{\mbf{Y}}\big|\mb E_Xe^{it\alpha(\kAP)(X,\mbf{Y})/\sigma}\big|,\]
and now let $E$ be the subset of $D_0^u$ with coefficient in $\{0,1\}$, with corresponding vector $X'$ (this only depends on the randomness of $\mbf{Y}$). Then we have
\[|\varphi_{\kAP'/\sigma}(t)|^{2^{k-1}}\le\mb E_{\mbf{Y},X\setminus{X'}}\bigg|\mb E_{X'}e^{(it/\sigma)\sum_{j\in E}a_jx_j}\bigg|.\]
By the above considerations, with high probability over the randomness of $\mbf{Y}$ we have a positive proportion of $E$ being $0,1$, and $|E| = \Theta_{k,\lambda}(n^{\delta_k})$. Then $\on{Var}[a_j]$ over $j\in E$ is $\Theta_{k,\lambda}(1)$, hence we deduce by \cref{lem:bernoulli-variance} that
\[|\varphi_{\kAP'/\sigma}(t)|^{2^{k-1}}\le\exp(-\Omega_{k,\lambda}(n^{\delta_k}))+\exp(-\Omega_{k,\lambda}(t^2n^{\delta_k}/\sigma^2)).\]
This also only applies if $|t/\sigma|\cdot(1-0)\le\pi$, which precisely hits the top of the range.
\end{proof}
Therefore, for $t$ in the given range, we deduce the desired quality of bounds.
 
\subsection{Deriving the final result}
We are ready to prove \cref{thm:kAP-local}, and then transfer the result to the independent setting.
\begin{proof}[Proof of \cref{thm:kAP-local}]
Recall from earlier that
\[\kAP(\mbf{y})=\sum_{\ell=0}^{k}\sum_{a\in \mb{Z}/n\mb{Z}}\sum_{d\in [n/2]} \sum_{S\in {\binom{[k]}{\ell}}} p^{k - \frac{\ell}{2}}(1-p)^{\frac{\ell}{2}}\prod_{i \in S} y_{a+id},\]
which differs from $\kAP'(\mbf{y})$ differ by a deterministic constant given the sum of $y$, call it $Y$, and thus our various decoupling estimates apply. To be more explicit, let
\begin{align*}
\mu = \kAP(\mbf{y})-\kAP'(\mbf{y}) & = p^k\binom{n}{2} + p^{k-\frac{1}{2}}(1-p)^{\frac{1}{2}}k\frac{n-1}{2}Y+\frac{p^{k-1}(1-p)}{2}\binom{k}{2}\bigg(Y^2+\frac{2p-1}{\sqrt{p(1-p)}}Y-n\bigg)\\
& = p^k\binom{n}{2}-np^{k-1}(1-p)\binom{k}{2}/2,
\end{align*}
the last expression coming from the facts
\[\sum_{i\neq j}y_iy_j = Y^2 - \sum_i y_i^2,\qquad y_i^2+\frac{2p-1}{\sqrt{p(1-p)}}y_i-1=0.\]
Let $Z' = (\kAP(\mbf{y})-\mu)/\sigma = \kAP'(\mbf{y})/\sigma$ and set $\varphi_{n}(t) = \mb{E}[e^{itZ'}]$ and $\varphi(t) = \mb{E}[e^{itZ}]$ where $Z\sim\mathcal{N}(0,1)$. Now note that  \begin{align*}
\bigg(\int_{-\pi\sigma}^{\pi\sigma}&|\varphi(t) - \varphi_{n}(t)|~dt\bigg)/\sigma\\
&= \bigg(\int_{|t|\le n^{\eps}}|\varphi(t) - \varphi_{n}(t)|~dt + \int_{n^{\eps}\le |t|\le \sigma\cdot n^{-\eps}}|\varphi(t) - \varphi_{n}(t)|~dt \\
&\qquad \qquad + \int_{\sigma\cdot n^{-\eps}\le |t|\le \pi\cdot \sigma}|\varphi(t) - \varphi_{n}(t)|~dt\bigg)/\sigma\\
&\lesssim_{\lambda} \bigg(\int_{|t|\le n^{\eps}}|t|/n^{1/4-\eps}~dt + \int_{n^{\eps}\le |t|\le \sigma\cdot n^{-\eps}}\exp(-\Omega_{k,\lambda}(n^{\eps'}))~dt\\
&\qquad \qquad + \int_{\sigma\cdot n^{-\eps}\le |t|\le \pi\cdot\sigma} \exp(-\Omega_{k,\lambda}(n^{\delta_k}))+\exp(-\Omega_{k,\lambda}(t^2n^{\delta_k}/\sigma^2))~dt\bigg)/\sigma\\
&\lesssim_{\lambda} 1/(\sigma\cdot n^{1/4-3\eps}).
\end{align*}
The bounds applied were derived in the previous subsections. Given this we are almost able to derive the necessary result; however once again we have the issue that $\mu, \sigma$ are not exactly the true mean or standard deviation $\mu_k, \sigma_k$ of $\kAP$.

Using techniques completely analogous to \cref{lem:std-close}, we can use the coupling in the proof of \cref{lem:kap-characteristic-low} to see that $\sigma_k = \sigma(1+O_{\lambda,\eps}(n^{\eps-1/4}))$, and we also have $\mu_k = \mu(1+O_{\lambda,\eps}(n^{-2}))$ via explicit calculation (similar to in \cref{lem:std-close}). This finishes.
\end{proof}
A similar analysis to the transfer given in \cref{sec:independent-models} allows us to obtain a local limit theorem for $k$-APs in the independent model. However, the local behavior is not Gaussian but rather comes from the superimposition of an infinite ensemble of Gaussians. For the following theorem note that we have, in the independent model,
\[(\on{Var}[\kAP'(\mbf{y})|\on{Ber}(p)])^{1/2} = c_{k,p}n (1+O_{k,p}(1/n))\]
and 
\[(\on{Var}[\kAP(\mbf{y})|\on{Ber}(p)])^{1/2} = p^{k-1/2}(1-p)^{1/2}kn^{3/2}/2(1+O_{k,p}(1/n))\]
for some constant $c_{k,p}$ depending only on $k,p$. It is worth noting that $c_{k,p}$ can be seen to be continuous on $p\in[\lambda,1-\lambda]$. Finally, define $C_{k,p}$ via
\[\frac{\sigma_k}{\sigma\sqrt{n}\sqrt{p(1-p)}} = \frac{1}{C_{k,p}}(1+O_{k,p}(1/n)).\]

\begin{theorem}\label{thm:kAP-local-ensemble}
Fix $k\ge 3$ and let $p\in(\lambda,1-\lambda)$. Choose $n\ge 1$ with $\gcd(n,(k-1)!)=1$ and sample a random set with indicator vector $\mbf{x}$, with each element drawn independently with probability $p$. Furthermore let $\mu_{k} = \mb{E}[\kAP(\mbf{x})]$ and $\sigma_{k} = \on{Var}[\kAP(\mbf{x})]$. Finally define $Z_k = (\kAP(\mbf{x}) - \mu_k)/\sigma_{k}$. Then we have for any $\eps > 0$ that
\[\sup_{z\in(\mb{Z}-\mu_k)/\sigma_k}\bigg|\sigma_k \mb{P}[Z_k = z] - \mathcal{N}(z)\sum_{m\in\mb{Z}}\frac{1}{C_{k,p}}\mathcal{N}\Bigg(\frac{pn + \sqrt{p(1-p)}z\sqrt{n}+\frac{(1-p)(k-1)(1-z^2)}{2} - m}{C_{k,p}}\Bigg)\bigg|\lesssim_{\lambda,\eps}n^{\eps-1/4},\]
where $C_{k,p}$ is defined as above.
\end{theorem}
\begin{proof}
Let $\kAP$ denote the number of $k$-term arithmetic progressions. In what follows we suppress $k$ dependence in asymptotic notation. Let $\sigma_k^2 = \on{Var}[\kAP|\on{Ber}(p)]$ and $\mu_k = \mb{E}[\kAP|\on{Ber}(p)]$ define the standard deviation and mean in the independent model. We now recall from the proof of \cref{thm:kAP-local} above that if $Y = \sum y_i$ then
\[\kAP(\mbf{y}) = p^k\binom{n}{2} + p^{k-\frac{1}{2}}(1-p)^{\frac{1}{2}}k\frac{n-1}{2}Y+p^{k-1}(1-p)\binom{k}{2}\frac{Y^2+\frac{2p-1}{\sqrt{p(1-p)}}Y-n}{2}+\kAP'(\mbf{y}).\]
For the sake of simplicity let 
\[f(Y) = p^k\binom{n}{2} + p^{k-\frac{1}{2}}(1-p)^{\frac{1}{2}}k\frac{n-1}{2}Y+p^{k-1}(1-p)\binom{k}{2}\frac{Y^2+\frac{2p-1}{\sqrt{p(1-p)}}Y-n}{2}.\]
First note that continuity of $c_{k,p}$ implies that
\[\on{Var}[\kAP'(\mbf{x})|\on{Ber}(p)] = (1+O_\lambda(\log n/n^{1/2})) \on{Var}[\kAP'(\mbf{x})|\on{Ber}(p')]\]
if $p' = (1+O_\lambda(\log n/n^{1/2}))p$. Given this and the deductions at the end of the proof of \cref{thm:kAP-local} that
\[\on{Var}[\kAP(\mbf{x})|\sum x_i = m] = \on{Var}[\kAP'(\mbf{x})|\sum x_i = m] = (1+O_{\lambda,\eps}(n^{\eps-1/4}))\on{Var}[\kAP'|\on{Ber}(q)]\]
(here $q = m/n\in(\lambda,1-\lambda)$), it follows for any $m,m'\in [pn - n^{1/2}\log n, pn+n^{1/2}\log n]$ that 
\begin{align*}
    \on{Var}[\kAP(\mbf{x})|\sum x_i=m] &= (1+O_{\lambda,\eps}(n^{\eps-1/4}))\on{Var}[\kAP(\mbf{x})|\sum x_i=m']\\
& = (1+O_{\lambda,\eps}(n^{\eps-1/4})) \on{Var}[\kAP'(\mbf{x})|\on{Ber}(p)].
\end{align*}
From now we denote $\sigma^2 = \on{Var}[\kAP'|\on{Ber}(p)]$, so $\sigma = \Theta_\lambda(n)$. We now explicitly use that the expectation of $\kAP$ varies with the function $f(y)$ to deduce our local limit theorem. In particular consider $m\in [pn-n^{1/2}\log n, pn+n^{1/2}\log n]$ and let $y = (m-pn)/\sqrt{p(1-p)}$. Note $y = O_\lambda(n^{1/2}\log n)$. Then by linearity of expectation we have that
\begin{align*}
\mb{E}[\kAP(\mbf{x})|\sum x_i = m] &= \binom{n}{2}\prod_{i=0}^{k-1}\bigg(\frac{m-i}{n-i}\bigg)\\
&= \bigg(\frac{m}{n}\bigg)^{k}\binom{n}{2}-\frac{n(\frac{m}{n})^{k-1}(1-\frac{m}{n})\binom{k}{2}}{2} + O_\lambda(1)\\
&= f(y) + O_\lambda(n^{1/2}\log n)
\end{align*}
We are now in position to explicitly calculate the distribution of $\kAP$ under the independent model. Let $\sigma_m^2 = \on{Var}[\kAP(\mbf{x})|\sum x_i = m]$ and $\mu_m = \mb{E}[\kAP(\mbf{x})|\sum x_i = m]$. Now note that
\begin{align*}
\mb{P}[\kAP(x) &= x] = \sum_{m\in \mb{Z}}\mb{P}\big[\kAP(x) = x|\sum x_i = m\big]\mb{P}\big[\sum x_i = m\big]\\
& = \sum_{\substack{m\in [pn - n^{1/2}\log n,\\ ~\quad pn + n^{1/2}\log n]}} \mb{P}[\kAP(x) = x|\sum x_i = m]\mb{P}\big[\sum x_i = m\big] + \exp(-\Omega_\lambda((\log n)^2))
\end{align*}
where we have used Chernoff to bound the probability that number of elements deviates too far from the mean. For the sake of clarity we will implicitly assume that $x$ is within $\sigma_k(\log n)^C$ of the mean; for $x$ outside this range and $C$ sufficiently large the probability of attaining $x$ is super-polynomially small by hypercontractivity (\cref{thm:concentration-hypercontractivity}) so the desired statement is trivial. This assumption will be used implicitly later on. Now let $\mc{M}_{x}$ denote the set of $m$ such that
\[|x-f(y)|\le\sigma(\log n)^C\]
and
\[m\in [pn - n^{1/2}\log n, pn + n^{1/2}\log n]\]
for a suitably large $C$. (As before, $y = (m-pn)/\sqrt{p(1-p)}$.) Now suppose that $m\in [pn - n^{1/2}\log n, pn + n^{1/2}\log n]\backslash\mc{M}_x$. Then
\begin{align*}
\mb{P}[\kAP(\mbf{x}) = x | \sum x_i = m] & \le \mb{P}\big[|\kAP'(\mbf{x})| \ge \sigma(\log n)^C/2\big]/ \mb{P}\big[\sum x_i = m\big]\\
&\lesssim\exp(-\Omega_\lambda((\log n)^2)),
\end{align*}
using that $\mb{P}[\sum x_e = m]\gtrsim\exp(-O_\lambda((\log n)^2))$ and then choosing $C$ sufficiently large so that the bound coming from hypercontractivity (\cref{thm:concentration-hypercontractivity}) on the numerator is sufficiently strong. The key point is that since $f'(y)$ is a linear function with slope $p^{k-1/2}(1-p)^{1/2}k(n-1)/2 + O_\lambda(n^{1/2}\log n)$ for $|y|\lesssim_\lambda n^{1/2}\log n$, we deduce $|\mc{M}_x| = \Theta_\lambda((\log n)^C)$. Thus we have that 
\begin{align*}
\mb{P}[\kAP(x) = x] &= \sum_{\substack{m\in [pn - n^{1/2}\log n,\\ ~\quad pn + n^{1/2}\log n]}}\mb{P}\big[\kAP(x) = x|\sum x_i = m\big]\mb{P}\big[\sum x_i = m\big] + \exp(-\Omega_\lambda((\log n)^2))\\
& = \sum_{m\in \mc{M}_x}\mb{P}\big[\kAP(x) = x|\sum x_i = m\big]\mb{P}\big[\sum x_i = m\big] + \exp(-\Omega_{\lambda}((\log n)^2)).
\end{align*}
Now using \cref{thm:kAP-local} and that $\sigma_{m}$ is approximately equal to $\sigma$, the last summation equals
\begin{align*}
\sum_{m\in \mc{M}_x}&\bigg(\frac{1}{\sigma_m}\mathcal{N}\left(\frac{x-\mu_m}{\sigma_{m}}\right)+O_\lambda\bigg(\frac{n^{(\eps-1)/4}}{\sigma}\bigg)\bigg)\mb{P}\big[\sum x_i = m\big]\\
& = \sum_{m\in \mc{M}_x}\frac{1}{\sigma_m}\mathcal{N}\left(\frac{x-\mu_m}{\sigma_{m}}\right)\mb{P}\big[\sum x_e = m\big] + O_\lambda\left(\frac{|M_{x}|}{n^{\frac{3-\eps}{4}}\sigma}\right)
\end{align*}
where use that probability of having a given number of elements is $O_\lambda(1/n^{1/2})$. Now note that $\sigma_k$ is order $n^{1/2}$ larger than $\sigma$. Therefore the error term can be seen to be $O_\lambda(n^{\eps-1/4}\sigma_H^{-1})$, which is the correct magnitude. Now $\sigma_m = (1+O_\lambda(n^{\eps-1/4}))\sigma$ and $\mu_m = f(y) + O_\lambda((\log n)^2n^{-1/2}\sigma)$ for all $m\in [pn-n^{1/2}\log n, pn+n^{1/2}\log n]$ by the remarks which began the proof. It follows that
\[\sum_{m\in \mc{M}_x}\frac{1}{\sigma_m}\mathcal{N}\left(\frac{x-\mu_{m}}{\sigma_{m}}\right)\mb{P}\big[\sum x_e = m\big] = \sum_{m\in \mc{M}_x}\frac{1}{\sigma}\mathcal{N}\left(\frac{x-f(y)}{\sigma}\right)\mb{P}\big[\sum x_e = m\big] + O_\lambda\bigg(\frac{n^{\eps-1/4}}{\sigma_k}\bigg).\]
At this point the rest is elementary, but nontrivial, calculation. Let $y^\ast$ be the solution to $f(y^{*}) = x$ and let $m^\ast=y^\ast\sqrt{p(1-p)}+pn$ be the corresponding $m$. Note that $|m-m^\ast|\lesssim_\lambda (\log n)^C$ since $f$ has slope $\Theta_\lambda(\sigma)$ on $\mc{M}_x$. This is enough to conclude that $\mb{P}[\sum x_e = m]$ is essentially constant over $m\in\mc{M}_x$, close enough to replace the above with
\[\mb{P}\big[\sum x_e=\lfloor m^\ast\rfloor\big]\sum_{m\in\mc{M}_x}\frac{1}{\sigma}\mc{N}\left(\frac{f(y^{*})-f(y)}{\sigma}\right)\]
without increasing the error term. Now, using that $f(y)$ has derivative $p^{k-1/2}(1-p)^{1/2}k(n-1)/2 + O_\lambda(n^{1/2}\log n)$ for $|y|\lesssim_\lambda n^{1/2}\log n$, we can (up to acceptable errors) rewrite the above as
\[\mb{P}\big[\sum x_e=\lfloor m^\ast\rfloor\big]\sum_{m\in\mc{M}_x}\frac{1}{\sigma}\mc{N}\left(\frac{(y^{*}-y)(p^{k-1/2}(1-p)^{1/2}k(n-1)/2)}{\sigma}\right).\]
Using the value of $\sigma_{k}$ and completing the above sum we find that it is close to
\[\mb{P}\big[\sum x_e=\lfloor m^\ast\rfloor\big]\sum_{y\in(\mb{Z}-pn)/\sqrt{p(1-p)}}\frac{1}{\sigma}\mc{N}\left(\frac{(y^{*}-y)\sigma_k}{\sqrt{n}\cdot\sigma}\right).\]
Note that $m^{*}-m = (y^{*}-y)\sqrt{p(1-p)}$. Now we compute $y^{*}$ up to a $o(1)$ additive accuracy. Letting $z = (x-\mu_k)/\sigma_k$, and using $|z|\le (\log n)^C$ we find that
\begin{align*}
y^{*} & = z\sqrt{n} + \frac{(1-p)^{1/2}(k-1)}{2p^{1/2}} - \frac{(1-p)^{1/2}(k-1)z^2}{2p^{1/2}} + O_{\lambda,\eps}(n^{\eps-1/2}).\\
 & = z\sqrt{n} + O_\lambda((\log n)^{2C}).
\end{align*}
Substituting in this expression we find that the above, up to tolerable losses, is
\[\frac{1}{\sqrt{p(1-p)n}}\mc{N}(z)\sum_{y\in(\mb{Z}-pn)/\sqrt{p(1-p)}}\frac{1}{\sigma}\mc{N}\left(\frac{(y^{*}-y)\sigma_k}{\sqrt{n}\cdot\sigma}\right),\]
which up to appropriate errors is
\[\frac{1}{\sigma_{k}}\mathcal{N}(z)\sum_{m\in\mb{Z}}\frac{1}{C_{k,p}}\mathcal{N}\Bigg(\frac{pn + \sqrt{p(1-p)}z\sqrt{n}+\frac{(1-p)(k-1)(1-z^2)}{2} - m}{C_{k,p}}\Bigg).\]
The result follows.
\end{proof}
This allows us to answer a question of the authors and Berkowitz \cite[Question~16]{BSS20} regarding the maximum ratio between pointwise probabilities near the mean. Indeed \cref{thm:kAP-local-ensemble} precisely pins down these probabilities to what was expected given the heuristics in \cite{BSS20}. The answer ultimately is the (predicted) ratio of two infinite sums as given above; explicitly, if
\[g(x) = \sum_{m\in\mb{Z}}\frac{1}{C_{k,p}}\mc{N}\left(\frac{x-m}{C_{k,p}}\right),\]
the maximum ratio of pointwise probabilities near the mean is $\sup g(x)/\inf g(x)$.

This example highlights the power of deducing a local limit theorem from a ``fixed size'' model, especially in a case such as this where the independent model does not satisfy a local central limit theorem as demonstrated in \cite{BSS20}. Indeed, we end up with the ``central limit behavior'' at the scale of $\sigma_k$, along with a multiplier that depends on the $\sigma_k n^{-1/2}$ scale that oscillates according to a theta series.

This technique also immediately gives the precise asymptotic for the maximal point in the distribution of the number of $k$-term arithmetic progressions, which answers a question in \cite{FKS19}. From the above, we see that the answer is
\[\sup_{a\in\mb{Z}}\mb{P}[\kAP(\mbf{x})=a|\on{Ber}(p)] = \frac{\sup g(x)}{\sigma_k}(1+O_{\lambda,\eps}(n^{\eps-1/4})).\]

\section*{Acknowledgements}
We thank Yufei Zhao for suggesting the problem, and thank Ross Berkowitz for useful discussions about the subgraph count problem. We also thank Vishesh Jain for mentioning the trick of using hypercontractivity on the hypercube to deduce bounds on a slice.
\bibliographystyle{amsplain0}
\bibliography{references}
\appendix
\section{Adjacency matrix of the construction}\label{app:adjacency-matrix}
Below is the adjacency matrix of the counterexample mentioned in \cref{sec:counterexamples}. It is also included separately in the arXiv listing of this paper.
\begin{figure}[h!]
\centering
\scalebox{0.35}{
$\begin{pmatrix}
0& 0& 0& 0& 0& 0& 0& 1& 1& 0& 1& 1& 1& 0& 1& 1& 0& 1& 1& 0& 0& 1& 0& 0& 0& 0& 0& 0& 0& 0& 0& 0& 1& 0&

1& 1& 1& 1& 1& 0& 1& 1& 1& 0& 0& 1& 1& 0& 1& 1& 0& 1& 1& 0& 1& 1& 1& 0& 0& 0& 0& 0& 1& 0 \\

0& 0& 0& 1& 1& 0& 1& 1& 1& 1& 0& 0& 1& 1& 0& 1& 0& 0& 1& 0& 0& 0& 1& 0& 0& 1& 1& 1& 1& 1& 1& 0& 0& 1&

0& 0& 1& 0& 0& 1& 0& 1& 0& 0& 1& 0& 0& 0& 0& 0& 1& 1& 1& 0& 0& 1& 1& 0& 1& 0& 0& 0& 0& 1 \\

0& 0& 0& 0& 0& 0& 1& 0& 1& 0& 1& 1& 0& 0& 0& 0& 1& 1& 0& 0& 0& 1& 0& 0& 0& 1& 0& 1& 0& 1& 1& 0& 0& 0&

0& 1& 1& 1& 0& 0& 0& 1& 0& 0& 0& 0& 0& 0& 1& 0& 1& 0& 0& 1& 0& 1& 0& 0& 1& 1& 1& 0& 0& 0 \\

0& 1& 0& 0& 0& 0& 0& 0& 1& 0& 0& 0& 1& 0& 1& 1& 1& 0& 0& 0& 1& 1& 1& 0& 1& 0& 0& 0& 1& 0& 0& 1& 1& 1&

0& 0& 0& 0& 0& 1& 1& 0& 1& 0& 0& 1& 1& 1& 1& 1& 1& 0& 0& 1& 1& 0& 0& 1& 0& 1& 0& 1& 1& 0 \\

0& 1& 0& 0& 0& 0& 0& 1& 0& 0& 0& 1& 0& 0& 1& 1& 0& 0& 1& 1& 1& 0& 0& 1& 0& 0& 1& 1& 0& 0& 0& 1& 0& 0&

1& 1& 0& 0& 1& 0& 1& 0& 1& 0& 0& 1& 0& 0& 1& 1& 1& 1& 0& 1& 1& 1& 0& 1& 1& 1& 0& 1& 0& 1 \\

0& 0& 0& 0& 0& 0& 1& 1& 1& 1& 0& 0& 0& 0& 1& 1& 1& 0& 1& 0& 1& 1& 1& 1& 1& 0& 1& 0& 1& 1& 0& 0& 1& 0&

1& 1& 1& 0& 0& 1& 1& 0& 0& 0& 1& 0& 1& 0& 1& 1& 0& 0& 0& 0& 1& 1& 1& 0& 0& 0& 1& 0& 1& 0 \\

0& 1& 1& 0& 0& 1& 0& 0& 0& 0& 1& 0& 1& 1& 0& 0& 0& 1& 0& 0& 0& 0& 0& 1& 0& 0& 0& 1& 0& 1& 0& 1& 1& 1&

1& 1& 0& 0& 0& 1& 0& 0& 1& 0& 0& 0& 0& 0& 1& 0& 0& 0& 0& 0& 0& 0& 1& 0& 1& 0& 1& 0& 1& 1 \\

1& 1& 0& 0& 1& 1& 0& 0& 0& 1& 1& 1& 0& 1& 0& 0& 1& 0& 1& 0& 1& 1& 0& 1& 0& 1& 1& 0& 1& 1& 1& 1& 0& 0&

0& 0& 1& 1& 0& 0& 1& 0& 1& 0& 1& 0& 0& 0& 0& 0& 1& 1& 1& 1& 1& 0& 0& 1& 1& 0& 0& 0& 0& 0 \\

1& 1& 1& 1& 0& 1& 0& 0& 0& 0& 1& 1& 0& 1& 1& 1& 1& 1& 1& 1& 1& 0& 1& 1& 1& 1& 0& 1& 0& 1& 1& 1& 0& 1&

0& 0& 0& 1& 1& 1& 1& 0& 0& 1& 1& 0& 0& 0& 0& 0& 0& 0& 1& 0& 0& 0& 1& 0& 0& 1& 0& 0& 0& 0 \\

0& 1& 0& 0& 0& 1& 0& 1& 0& 0& 1& 1& 1& 0& 0& 0& 1& 0& 1& 0& 1& 0& 0& 1& 1& 0& 1& 1& 0& 0& 1& 0& 1& 0&

1& 0& 1& 1& 0& 1& 0& 0& 1& 1& 0& 1& 1& 1& 0& 1& 1& 0& 0& 1& 0& 1& 1& 1& 1& 1& 1& 1& 0& 1 \\

1& 0& 1& 0& 0& 0& 1& 1& 1& 1& 0& 0& 0& 0& 1& 0& 1& 0& 0& 0& 1& 1& 0& 0& 0& 0& 1& 0& 1& 1& 1& 1& 1& 1&

1& 1& 0& 1& 0& 1& 0& 0& 0& 0& 1& 1& 0& 0& 0& 1& 1& 0& 0& 0& 0& 1& 0& 1& 1& 1& 0& 1& 1& 0 \\

1& 0& 1& 0& 1& 0& 0& 1& 1& 1& 0& 0& 0& 1& 1& 0& 1& 0& 0& 1& 1& 0& 1& 1& 1& 0& 0& 0& 0& 1& 0& 1& 0& 0&

0& 1& 1& 1& 1& 0& 0& 1& 0& 0& 1& 1& 1& 0& 1& 0& 1& 0& 1& 1& 1& 1& 0& 1& 0& 0& 0& 0& 1& 1 \\

1& 1& 0& 1& 0& 0& 1& 0& 0& 1& 0& 0& 0& 0& 1& 0& 1& 1& 0& 0& 1& 1& 1& 1& 1& 0& 1& 1& 0& 0& 1& 1& 1& 1&

0& 1& 1& 0& 0& 0& 0& 0& 0& 0& 0& 1& 1& 1& 1& 1& 1& 0& 0& 0& 1& 0& 1& 0& 0& 0& 1& 1& 1& 0 \\

0& 1& 0& 0& 0& 0& 1& 1& 1& 0& 0& 1& 0& 0& 0& 0& 0& 0& 0& 1& 0& 0& 0& 0& 0& 0& 0& 1& 1& 0& 1& 1& 1& 1&

1& 1& 1& 0& 0& 1& 0& 1& 1& 0& 1& 0& 1& 1& 0& 0& 0& 0& 0& 0& 0& 0& 0& 0& 0& 0& 0& 0& 1& 0 \\

1& 0& 0& 1& 1& 1& 0& 0& 1& 0& 1& 1& 1& 0& 0& 0& 1& 0& 1& 0& 0& 1& 1& 0& 0& 1& 1& 0& 1& 0& 0& 1& 0& 0&

0& 0& 0& 0& 0& 1& 1& 1& 0& 1& 0& 0& 1& 0& 1& 0& 0& 0& 0& 0& 0& 1& 0& 1& 0& 1& 1& 1& 0& 1 \\

1& 1& 0& 1& 1& 1& 0& 0& 1& 0& 0& 0& 0& 0& 0& 0& 0& 1& 0& 0& 0& 1& 0& 0& 0& 0& 1& 0& 0& 1& 1& 0& 1& 1&

1& 0& 1& 1& 1& 1& 1& 0& 0& 0& 0& 0& 0& 1& 1& 0& 1& 0& 0& 1& 0& 0& 0& 0& 1& 1& 0& 1& 1& 1 \\

0& 0& 1& 1& 0& 1& 0& 1& 1& 1& 1& 1& 1& 0& 1& 0& 0& 0& 1& 1& 0& 1& 0& 1& 0& 0& 1& 1& 1& 1& 1& 1& 0& 0&

0& 1& 0& 0& 0& 1& 0& 0& 0& 0& 0& 1& 1& 1& 1& 1& 1& 0& 1& 1& 1& 0& 1& 0& 0& 0& 0& 1& 1& 0 \\

1& 0& 1& 0& 0& 0& 1& 0& 1& 0& 0& 0& 1& 0& 0& 1& 0& 0& 1& 1& 1& 1& 0& 0& 1& 0& 1& 1& 0& 0& 1& 0& 1& 1&

0& 1& 0& 1& 0& 1& 1& 0& 0& 1& 0& 1& 1& 1& 0& 0& 0& 1& 1& 1& 1& 0& 1& 1& 1& 0& 0& 0& 0& 0 \\

1& 1& 0& 0& 1& 1& 0& 1& 1& 1& 0& 0& 0& 0& 1& 0& 1& 1& 0& 0& 1& 1& 0& 0& 1& 0& 0& 0& 0& 1& 1& 0& 0& 1&

0& 1& 0& 0& 0& 0& 1& 1& 0& 1& 0& 1& 1& 1& 0& 1& 0& 1& 0& 1& 1& 0& 0& 0& 1& 1& 1& 0& 1& 0 \\

0& 0& 0& 0& 1& 0& 0& 0& 1& 0& 0& 1& 0& 1& 0& 0& 1& 1& 0& 0& 1& 1& 1& 0& 1& 0& 1& 0& 1& 0& 0& 1& 1& 1&

0& 0& 0& 0& 1& 0& 0& 0& 1& 0& 1& 1& 0& 1& 0& 0& 0& 1& 1& 1& 0& 1& 1& 0& 1& 1& 1& 1& 1& 0 \\

0& 0& 0& 1& 1& 1& 0& 1& 1& 1& 1& 1& 1& 0& 0& 0& 0& 1& 1& 1& 0& 0& 1& 1& 1& 0& 0& 1& 1& 0& 1& 0& 0& 0&

0& 1& 0& 1& 0& 1& 1& 0& 0& 1& 1& 1& 1& 0& 0& 1& 1& 0& 0& 1& 1& 1& 0& 1& 1& 0& 0& 1& 0& 1 \\

1& 0& 1& 1& 0& 1& 0& 1& 0& 0& 1& 0& 1& 0& 1& 1& 1& 1& 1& 1& 0& 0& 0& 1& 0& 0& 0& 0& 0& 0& 1& 1& 0& 1&

1& 0& 0& 0& 0& 0& 1& 0& 1& 0& 1& 0& 0& 1& 0& 0& 1& 1& 1& 0& 1& 0& 1& 0& 1& 0& 1& 0& 0& 0 \\

0& 1& 0& 1& 0& 1& 0& 0& 1& 0& 0& 1& 1& 0& 1& 0& 0& 0& 0& 1& 1& 0& 0& 0& 1& 0& 1& 1& 0& 0& 0& 0& 0& 0&

0& 0& 1& 0& 1& 1& 1& 0& 1& 1& 0& 0& 0& 0& 0& 0& 1& 1& 0& 0& 1& 0& 0& 1& 1& 0& 0& 1& 0& 0 \\

0& 0& 0& 0& 1& 1& 1& 1& 1& 1& 0& 1& 1& 0& 0& 0& 1& 0& 0& 0& 1& 1& 0& 0& 0& 1& 0& 0& 1& 1& 0& 1& 1& 0&

0& 1& 0& 1& 1& 1& 0& 0& 1& 0& 0& 1& 1& 1& 1& 0& 0& 0& 1& 0& 0& 1& 0& 1& 0& 0& 1& 1& 1& 0 \\

0& 0& 0& 1& 0& 1& 0& 0& 1& 1& 0& 1& 1& 0& 0& 0& 0& 1& 1& 1& 1& 0& 1& 0& 0& 1& 0& 1& 0& 0& 0& 0& 0& 1&

0& 1& 0& 1& 1& 1& 0& 0& 0& 0& 0& 0& 0& 1& 1& 0& 0& 0& 0& 1& 1& 0& 1& 0& 0& 0& 0& 0& 0& 0 \\

0& 1& 1& 0& 0& 0& 0& 1& 1& 0& 0& 0& 0& 0& 1& 0& 0& 0& 0& 0& 0& 0& 0& 1& 1& 0& 0& 0& 1& 0& 0& 1& 0& 1&

1& 1& 1& 0& 1& 0& 0& 0& 1& 0& 0& 1& 1& 1& 1& 1& 1& 1& 1& 0& 0& 0& 0& 0& 1& 1& 1& 1& 1& 0 \\

0& 1& 0& 0& 1& 1& 0& 1& 0& 1& 1& 0& 1& 0& 1& 1& 1& 1& 0& 1& 0& 0& 1& 0& 0& 0& 0& 1& 1& 1& 0& 0& 0& 0&

1& 1& 0& 1& 1& 0& 1& 0& 0& 1& 1& 0& 0& 1& 1& 1& 1& 0& 1& 1& 1& 0& 0& 1& 0& 0& 1& 0& 0& 0 \\

0& 1& 1& 0& 1& 0& 1& 0& 1& 1& 0& 0& 1& 1& 0& 0& 1& 1& 0& 0& 1& 0& 1& 0& 1& 0& 1& 0& 1& 1& 0& 1& 1& 1&

0& 0& 1& 1& 0& 0& 0& 0& 1& 0& 1& 1& 0& 0& 1& 1& 1& 0& 1& 1& 1& 1& 1& 1& 1& 1& 0& 1& 1& 1 \\

0& 1& 0& 1& 0& 1& 0& 1& 0& 0& 1& 0& 0& 1& 1& 0& 1& 0& 0& 1& 1& 0& 0& 1& 0& 1& 1& 1& 0& 0& 1& 0& 1& 1&

0& 1& 1& 1& 1& 0& 1& 0& 0& 1& 1& 0& 0& 1& 0& 1& 1& 1& 0& 1& 0& 0& 1& 0& 1& 0& 1& 1& 0& 0 \\

0& 1& 1& 0& 0& 1& 1& 1& 1& 0& 1& 1& 0& 0& 0& 1& 1& 0& 1& 0& 0& 0& 0& 1& 0& 0& 1& 1& 0& 0& 1& 1& 1& 1&

0& 1& 1& 1& 0& 1& 1& 1& 0& 0& 1& 0& 0& 0& 0& 1& 0& 0& 1& 0& 0& 0& 0& 0& 1& 1& 0& 1& 0& 1 \\

0& 1& 1& 0& 0& 0& 0& 1& 1& 1& 1& 0& 1& 1& 0& 1& 1& 1& 1& 0& 1& 1& 0& 0& 0& 0& 0& 0& 1& 1& 0& 0& 1& 1&

1& 0& 1& 0& 0& 0& 1& 1& 0& 0& 0& 0& 1& 1& 1& 0& 0& 1& 1& 0& 0& 0& 0& 1& 1& 0& 1& 0& 0& 1 \\

0& 0& 0& 1& 1& 0& 1& 1& 1& 0& 1& 1& 1& 1& 1& 0& 1& 0& 0& 1& 0& 1& 0& 1& 0& 1& 0& 1& 0& 1& 0& 0& 0& 0&

1& 1& 1& 1& 1& 1& 0& 1& 0& 0& 1& 0& 1& 0& 1& 0& 1& 1& 1& 1& 1& 1& 1& 1& 1& 0& 0& 0& 1& 0 \\

1& 0& 0& 1& 0& 1& 1& 0& 0& 1& 1& 0& 1& 1& 0& 1& 0& 1& 0& 1& 0& 0& 0& 1& 0& 0& 0& 1& 1& 1& 1& 0& 0& 0&

0& 1& 0& 0& 1& 1& 1& 0& 0& 1& 0& 1& 1& 1& 0& 1& 0& 1& 0& 1& 1& 0& 1& 1& 0& 0& 0& 1& 1& 1 \\

0& 1& 0& 1& 0& 0& 1& 0& 1& 0& 1& 0& 1& 1& 0& 1& 0& 1& 1& 1& 0& 1& 0& 0& 1& 1& 0& 1& 1& 1& 1& 0& 0& 0&

1& 1& 1& 1& 0& 0& 0& 1& 1& 1& 0& 0& 0& 0& 1& 1& 1& 1& 1& 0& 0& 1& 1& 0& 1& 0& 0& 0& 1& 0 \\

1& 0& 0& 0& 1& 1& 1& 0& 0& 1& 1& 0& 0& 1& 0& 1& 0& 0& 0& 0& 0& 1& 0& 0& 0& 1& 1& 0& 0& 0& 1& 1& 0& 1&

0& 1& 1& 1& 1& 0& 0& 1& 0& 1& 1& 1& 0& 0& 0& 0& 0& 0& 0& 0& 0& 0& 0& 1& 0& 1& 0& 0& 1& 1 \\

1& 0& 1& 0& 1& 1& 1& 0& 0& 0& 1& 1& 1& 1& 0& 0& 1& 1& 1& 0& 1& 0& 0& 1& 1& 1& 1& 0& 1& 1& 0& 1& 1& 1&

1& 0& 1& 1& 0& 1& 1& 1& 0& 0& 1& 1& 0& 1& 1& 1& 1& 0& 1& 0& 0& 1& 0& 1& 1& 1& 1& 0& 0& 0 \\

1& 1& 1& 0& 0& 1& 0& 1& 0& 1& 0& 1& 1& 1& 0& 1& 0& 0& 0& 0& 0& 0& 1& 0& 0& 1& 0& 1& 1& 1& 1& 1& 0& 1&

1& 1& 0& 1& 1& 0& 0& 0& 1& 1& 1& 1& 1& 0& 0& 1& 0& 1& 1& 0& 1& 0& 0& 0& 1& 1& 1& 1& 0& 0 \\

1& 0& 1& 0& 0& 0& 0& 1& 1& 1& 1& 1& 0& 0& 0& 1& 0& 1& 0& 0& 1& 0& 0& 1& 1& 0& 1& 1& 1& 1& 0& 1& 0& 1&

1& 1& 1& 0& 1& 1& 1& 1& 1& 1& 1& 0& 0& 1& 1& 0& 0& 1& 0& 0& 1& 0& 1& 1& 1& 1& 1& 0& 0& 1 \\

1& 0& 0& 0& 1& 0& 0& 0& 1& 0& 0& 1& 0& 0& 0& 1& 0& 0& 0& 1& 0& 0& 1& 1& 1& 1& 1& 0& 1& 0& 0& 1& 1& 0&

1& 0& 1& 1& 0& 0& 1& 1& 1& 1& 1& 0& 0& 1& 1& 0& 1& 0& 0& 1& 0& 0& 1& 0& 1& 1& 0& 0& 0& 1 \\

0& 1& 0& 1& 0& 1& 1& 0& 1& 1& 1& 0& 0& 1& 1& 1& 1& 1& 0& 0& 1& 0& 1& 1& 1& 0& 0& 0& 0& 1& 0& 1& 1& 0&

0& 1& 0& 1& 0& 0& 0& 0& 0& 0& 1& 0& 1& 0& 1& 1& 1& 0& 1& 1& 1& 0& 1& 1& 0& 1& 0& 0& 0& 1 \\

1& 0& 0& 1& 1& 1& 0& 1& 1& 0& 0& 0& 0& 0& 1& 1& 0& 1& 1& 0& 1& 1& 1& 0& 0& 0& 1& 0& 1& 1& 1& 0& 1& 0&

0& 1& 0& 1& 1& 0& 0& 1& 1& 1& 0& 0& 0& 0& 1& 1& 0& 1& 0& 1& 0& 1& 0& 0& 0& 0& 1& 1& 1& 0 \\

1& 1& 1& 0& 0& 0& 0& 0& 0& 0& 0& 1& 0& 1& 1& 0& 0& 0& 1& 0& 0& 0& 0& 0& 0& 0& 0& 0& 0& 1& 1& 1& 0& 1&

1& 1& 0& 1& 1& 0& 1& 0& 0& 0& 1& 1& 0& 1& 1& 0& 0& 0& 0& 0& 0& 0& 0& 1& 0& 1& 0& 0& 1& 0 \\

1& 0& 0& 1& 1& 0& 1& 1& 0& 1& 0& 0& 0& 1& 0& 0& 0& 0& 0& 1& 0& 1& 1& 1& 0& 1& 0& 1& 0& 0& 0& 0& 0& 1&

0& 0& 1& 1& 1& 0& 1& 0& 0& 1& 0& 0& 1& 1& 0& 1& 1& 1& 0& 1& 0& 1& 0& 1& 1& 1& 1& 1& 0& 0 \\

0& 0& 0& 0& 0& 0& 0& 0& 1& 1& 0& 0& 0& 0& 1& 0& 0& 1& 1& 0& 1& 0& 1& 0& 0& 0& 1& 0& 1& 0& 0& 0& 1& 1&

1& 0& 1& 1& 1& 0& 1& 0& 1& 0& 1& 0& 1& 0& 0& 0& 0& 1& 1& 1& 0& 1& 0& 0& 1& 1& 1& 0& 0& 1 \\

0& 1& 0& 0& 0& 1& 0& 1& 1& 0& 1& 1& 0& 1& 0& 0& 0& 0& 0& 1& 1& 1& 0& 0& 0& 0& 1& 1& 1& 1& 0& 1& 0& 0&

1& 1& 1& 1& 1& 1& 0& 1& 0& 1& 0& 0& 1& 0& 1& 0& 0& 1& 0& 0& 0& 1& 0& 0& 1& 0& 1& 0& 0& 1 \\

1& 0& 0& 1& 1& 0& 0& 0& 0& 1& 1& 1& 1& 0& 0& 0& 1& 1& 1& 1& 1& 0& 0& 1& 0& 1& 0& 1& 0& 0& 0& 0& 1& 0&

1& 1& 1& 0& 0& 0& 0& 1& 0& 0& 0& 0& 1& 1& 1& 1& 1& 0& 0& 0& 0& 0& 1& 1& 0& 1& 0& 1& 0& 1 \\

1& 0& 0& 1& 0& 1& 0& 0& 0& 1& 0& 1& 1& 1& 1& 0& 1& 1& 1& 0& 1& 0& 0& 1& 0& 1& 0& 0& 0& 0& 1& 1& 1& 0&

0& 0& 1& 0& 0& 1& 0& 0& 1& 1& 1& 1& 0& 1& 1& 1& 0& 0& 0& 1& 1& 0& 0& 0& 0& 0& 1& 1& 0& 1 \\

0& 0& 0& 1& 0& 0& 0& 0& 0& 1& 0& 0& 1& 1& 0& 1& 1& 1& 1& 1& 0& 1& 0& 1& 1& 1& 1& 0& 1& 0& 1& 0& 1& 0&

0& 1& 0& 1& 1& 0& 0& 1& 1& 0& 0& 1& 1& 0& 0& 1& 1& 1& 1& 0& 0& 1& 0& 0& 0& 1& 0& 1& 0& 1 \\

1& 0& 1& 1& 1& 1& 1& 0& 0& 0& 0& 1& 1& 0& 1& 1& 1& 0& 0& 0& 0& 0& 0& 1& 1& 1& 1& 1& 0& 0& 1& 1& 0& 1&

0& 1& 0& 1& 1& 1& 1& 1& 0& 0& 1& 1& 1& 0& 0& 0& 1& 0& 0& 0& 0& 1& 0& 0& 1& 1& 0& 0& 1& 1 \\

1& 0& 0& 1& 1& 1& 0& 0& 0& 1& 1& 0& 1& 0& 0& 0& 1& 0& 1& 0& 1& 0& 0& 0& 0& 1& 1& 1& 1& 1& 0& 0& 1& 1&

0& 1& 1& 0& 0& 1& 1& 0& 1& 0& 0& 1& 1& 1& 0& 0& 1& 1& 1& 0& 1& 0& 0& 0& 1& 1& 1& 0& 0& 0 \\

0& 1& 1& 1& 1& 0& 0& 1& 0& 1& 1& 1& 1& 0& 0& 1& 1& 0& 0& 0& 1& 1& 1& 0& 0& 1& 1& 1& 1& 0& 0& 1& 0& 1&

0& 1& 0& 0& 1& 1& 0& 0& 1& 0& 0& 1& 0& 1& 1& 1& 0& 0& 0& 0& 0& 1& 0& 0& 1& 0& 0& 1& 1& 0 \\

1& 1& 0& 0& 1& 0& 0& 1& 0& 0& 0& 0& 0& 0& 0& 0& 0& 1& 1& 1& 0& 1& 1& 0& 0& 1& 0& 0& 1& 0& 1& 1& 1& 1&

0& 0& 1& 1& 0& 0& 1& 0& 1& 1& 1& 0& 0& 1& 0& 1& 0& 0& 0& 1& 0& 0& 1& 0& 1& 0& 1& 0& 1& 0 \\

1& 1& 0& 0& 0& 0& 0& 1& 1& 0& 0& 1& 0& 0& 0& 0& 1& 1& 0& 1& 0& 1& 0& 1& 0& 1& 1& 1& 0& 1& 1& 1& 0& 1&

0& 1& 1& 0& 0& 1& 0& 0& 0& 1& 0& 0& 0& 1& 0& 1& 0& 0& 0& 0& 0& 1& 1& 0& 1& 1& 0& 1& 1& 1 \\

0& 0& 1& 1& 1& 0& 0& 1& 0& 1& 0& 1& 0& 0& 0& 1& 1& 1& 1& 1& 1& 0& 0& 0& 1& 0& 1& 1& 1& 0& 0& 1& 1& 0&

0& 0& 0& 0& 1& 1& 1& 0& 1& 1& 0& 0& 1& 0& 0& 0& 0& 1& 0& 0& 0& 0& 0& 0& 0& 1& 1& 0& 1& 1 \\

1& 0& 0& 1& 1& 1& 0& 1& 0& 0& 0& 1& 1& 0& 0& 0& 1& 1& 1& 0& 1& 1& 1& 0& 1& 0& 1& 1& 0& 0& 0& 1& 1& 0&

0& 0& 1& 1& 0& 1& 0& 0& 0& 0& 0& 0& 1& 0& 0& 1& 0& 0& 0& 0& 0& 1& 1& 1& 1& 0& 0& 1& 0& 0 \\

1& 1& 1& 0& 1& 1& 0& 0& 0& 1& 1& 1& 0& 0& 1& 0& 0& 0& 0& 1& 1& 0& 0& 1& 0& 0& 0& 1& 0& 0& 0& 1& 0& 1&

0& 1& 0& 0& 0& 0& 1& 0& 1& 1& 1& 0& 0& 1& 1& 0& 1& 0& 1& 0& 1& 0& 0& 0& 1& 0& 1& 1& 1& 0 \\

1& 1& 0& 0& 0& 1& 1& 0& 1& 1& 0& 0& 1& 0& 0& 0& 1& 1& 0& 1& 0& 1& 0& 0& 1& 0& 0& 1& 1& 0& 0& 1& 1& 1&

0& 0& 0& 1& 1& 1& 0& 0& 0& 0& 0& 1& 0& 0& 0& 0& 0& 1& 1& 0& 1& 0& 0& 1& 1& 1& 1& 1& 0& 1 \\

0& 0& 0& 1& 1& 0& 0& 1& 0& 1& 1& 1& 0& 0& 1& 0& 0& 1& 0& 0& 1& 0& 1& 1& 0& 0& 1& 1& 0& 0& 1& 1& 1& 0&

1& 1& 0& 1& 0& 1& 0& 1& 1& 0& 0& 1& 0& 0& 0& 0& 0& 0& 0& 0& 1& 0& 1& 0& 0& 1& 0& 0& 0& 0 \\

0& 1& 1& 0& 1& 0& 1& 1& 0& 1& 1& 0& 0& 0& 0& 1& 0& 1& 1& 1& 1& 1& 1& 0& 0& 1& 0& 1& 1& 1& 1& 1& 0& 1&

0& 1& 1& 1& 1& 0& 0& 0& 1& 1& 1& 0& 0& 0& 1& 1& 1& 1& 1& 0& 1& 1& 1& 0& 0& 1& 0& 1& 0& 1 \\

0& 0& 1& 1& 1& 0& 0& 0& 1& 1& 1& 0& 0& 0& 1& 1& 0& 0& 1& 1& 0& 0& 0& 0& 0& 1& 0& 1& 0& 1& 0& 0& 0& 0&

1& 1& 1& 1& 1& 1& 0& 1& 1& 1& 0& 1& 0& 1& 1& 1& 0& 0& 1& 1& 0& 0& 1& 1& 1& 0& 1& 0& 0& 0 \\

0& 0& 1& 0& 0& 1& 1& 0& 0& 1& 0& 0& 1& 0& 1& 0& 0& 0& 1& 1& 0& 1& 0& 1& 0& 1& 1& 0& 1& 0& 1& 0& 0& 0&

0& 1& 1& 1& 0& 0& 1& 0& 1& 1& 1& 0& 1& 0& 0& 1& 0& 1& 0& 1& 0& 1& 1& 0& 0& 1& 0& 0& 1& 0 \\

0& 0& 0& 1& 1& 0& 0& 0& 0& 1& 1& 0& 1& 0& 1& 1& 1& 0& 0& 1& 1& 0& 1& 1& 0& 1& 0& 1& 1& 1& 0& 0& 1& 0&

0& 0& 1& 0& 0& 0& 1& 0& 1& 0& 0& 1& 1& 1& 0& 0& 1& 0& 1& 0& 1& 1& 1& 0& 1& 0& 0& 0& 1& 0 \\

1& 0& 0& 1& 0& 1& 1& 0& 0& 0& 1& 1& 1& 1& 0& 1& 1& 0& 1& 1& 0& 0& 0& 1& 0& 1& 0& 1& 0& 0& 0& 1& 1& 1&

1& 0& 0& 0& 0& 0& 1& 1& 0& 0& 0& 0& 0& 0& 1& 0& 1& 1& 1& 1& 0& 1& 0& 0& 0& 0& 1& 1& 0& 0 \\

0& 1& 0& 0& 1& 0& 1& 0& 0& 1& 0& 1& 0& 0& 1& 1& 0& 0& 0& 0& 1& 0& 0& 0& 0& 0& 0& 1& 0& 1& 1& 0& 1& 0&

1& 0& 0& 1& 1& 1& 0& 0& 0& 1& 1& 1& 1& 1& 1& 0& 0& 0& 1& 1& 0& 0& 1& 0& 1& 0& 0& 0& 0& 0 \\
\end{pmatrix}$ }
\label{fig:counterexample}
\caption{Adjacency matrix of the 64 vertex counterexample}
\end{figure}\\

A visualization of this graph is shown below.
\begin{figure}[h!]
\begin{center}
\includegraphics[width=7in]{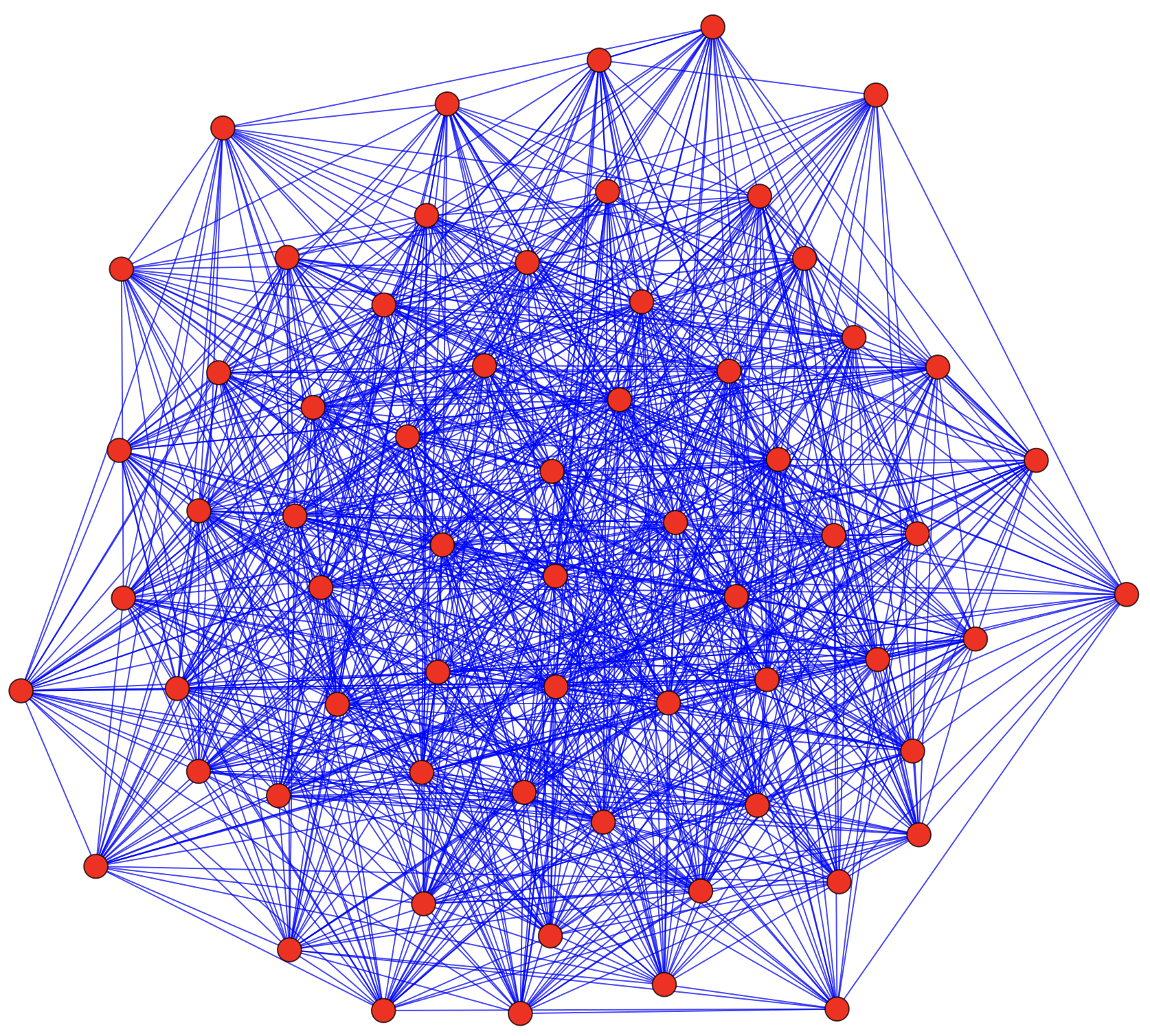}
\caption{Graph drawing of the 64 vertex counterexample}
\label{fig:experiment}
\end{center}
\end{figure}

\end{document}